\documentclass[a4paper,10pt]{article}
\usepackage{amsmath,amsfonts,amssymb}
\usepackage{mathrsfs}
\usepackage{multirow}
\usepackage{amsthm}
\usepackage[latin1]{inputenc}
\usepackage[T1]{fontenc}
\usepackage{color}
\usepackage{dsfont}
\usepackage{enumerate}
\usepackage{graphicx}
\usepackage{fullpage}
\usepackage{float}
\usepackage[titletoc,toc,title]{appendix}
\usepackage{pgf,tikz}
\usepackage{mathrsfs}
\usetikzlibrary{arrows}

\numberwithin{equation}{section}
\DeclareMathOperator{\dvol}{dvol}
\DeclareMathOperator{\dM}{dvol^{(M)}}

\newcommand{\reste}{\varepsilon}
\newcommand{\y}{ \left( \frac{y}{\lambda \kappa_n} \right)}
\DeclareMathOperator{\Ric}{Ric}
\DeclareMathOperator{\Scal}{Sc}
\newcommand{\com}[1]{}

\DeclareMathOperator{\vol}{vol}
\newcommand{\sph}{\mathcal{S}}
\newcommand{\Pl}{\mathcal{P}_\lambda}
\newcommand{\R}{\mathbb{R}}
\newcommand{\bo}{\mathcal{B}}
\newcommand{\E}{\mathbb{E}}
\newcommand{\C}{\mathcal{C}}
\newcommand{\Cl}{\mathcal{C}_{x_0,\lambda}^{(M)}}
\newcommand{\Nl}{N(\mathcal{C}_{x_0,\lambda}^{(M)})}

\newcommand{\proba}{\mathbb{P}}

\newcommand{\dd}{\mathrm{d}}
\newcommand{\J}{\mathcal{J}}

\author{}
\title{Mean asymptotics for a Poisson-Voronoi cell on a Riemannian manifold}
\date\today
\newtheorem{thm} {Theorem}[section]
\newtheorem*{cit}{Theorem}

\newtheorem{lem} [thm] {Lemma}

\newtheorem{prop} [thm] {Proposition}

\begin{document}
\author{Pierre Calka\textsuperscript{1}, Aur\'elie Chapron\textsuperscript{2} and Nathana\"el Enriquez\textsuperscript{3}}
\footnotetext{2010 {\it Mathematics Subject Classification}. Primary 60D05, 53B21; Secondary 53C65, 60G55. {\it Key words and phrases}. Poisson point process; Poisson-Voronoi tessellation; Riemannian manifold; Curvatures; Jacobi fields; Gauss-Bonnet theorem; Blaschke-Petkantschin formula.}
\footnotetext{\textsuperscript{1} Laboratoire de Math\'ematiques Rapha\"el Salem, UMR 6085, Universit\'e de Rouen, avenue de l'Universit\'e,
Technop\^ole du Madrillet, 76801 Saint-Etienne-du-Rouvray, France. E-mail: {\tt pierre.calka@univ-rouen.fr}}
\footnotetext{\textsuperscript{2} Laboratoire MODAL'X, EA 3454, Universit\'e Paris Nanterre, 200 avenue de la R\'e\-pu\-bli\-que, 92001 Nanterre, France. E-mail: {\tt achapron@parisnanterre.fr}}
\footnotetext{\textsuperscript{3} Laboratoire de Mathématiques d'Orsay, Universit\'e Paris-Sud, 
	Bâtiment 425, 91405 Orsay, France. E-mail: {\tt nathanael.enriquez@u-psud.fr}}
\maketitle
\begin{abstract}
In this paper, we consider a Riemannian manifold $M$ and the Poisson-Voronoi tessellation generated by the union of a fixed point $x_0$ and a Poisson point process of intensity $\lambda$ on $M$. We obtain asymptotic expansions up to the second order for the means of several characteristics of the Voronoi cell associated with $x_0$, including its volume and number of vertices. In each case, the first term of the estimate is equal to the mean characteristic in the Euclidean setting while the second term may contain a particular curvature of $M$ at $x_0$: the scalar curvature in the case of the mean number of vertices, the Ricci curvature in the case of the density of vertices and the sectional curvatures in the cases of the volume and number of vertices of a section of the Voronoi cell. Several explicit formulas are also derived in the particular case of constant curvature. The key tool for proving these results is a new change of variables formula of Blaschke-Petkantschin type in the Riemannian setting. Finally, a probabilistic proof of the Gauss-Bonnet Theorem is deduced from the asymptotic estimate of the total number of vertices of the tessellation in dimension two.  
\end{abstract}
\section{Introduction and results}

The Poisson-Voronoi tessellation is one of the most natural models of random tessellation of the Euclidean space \cite{Mol94}. It is used in many domains such as cristallography \cite{Mei53}, telecommunications \cite{Bac01} and astrophysics \cite{Wey94}. Available results include notably its mean characteristics \cite{Mol89} as well as some of its distributional \cite{Bau07} and asymptotical properties \cite{HRS04}. 

Let us now extend the definition of the Poisson-Voronoi tessellation to a Riemannian manifold. One of our goals is to use the mean  combinatorial characteristics of this tessellation in order to estimate the local geometric characteristics of the manifold.

Let $M$ be a complete and connected ${\mathcal C}^{\infty}$ Riemannian manifold of dimension $n\ge 2$ endowed with its Riemannian metric $d^{(M)}(\cdot,\cdot)$. For $x\in M$, let $\langle \cdot , \cdot \rangle_x$ and $\|\cdot\|_x$ be respectively the metric tensor and the induced norm on the tangent space of $M$ at $x$, denoted by $T_xM$. We denote by $\vol^{(M)}$ the volume measure on $M$ associated with $d^{(M)}$. Let $\Pl$ be a Poisson point process on $M$ of intensity measure $\lambda \vol^{(M)}$. The Poisson-Voronoi tessellation generated by the set of nuclei $\Pl$ is the collection of cells
\[ \C^{(M)}(x,\Pl)= \{ y \in M, d(x,y) \le d(x',y), \forall x' \in \Pl \},~x\in\Pl.\] 
In order to recover the local geometry around a point $x_0\in M$, we add this point to the process $\Pl$ and we investigate the characteristics at high intensity of the Voronoi cell associated with the nucleus $x_0$, i.e. $$\Cl:=\C^{(M)}(x_0,\Pl\cup \{x_0\}).$$ 
We do so in the high intensity setting which ensures that the cell $\Cl$ lies in a small neighborhood of $x_0$ with high probability. 
When $M$ is the Euclidean space, Slivnyak's theorem implies that this cell can be seen as a realization of the typical cell of the Poisson-Voronoi tessellation, see e.g. \cite{Mol89}. As in the Euclidean setting, we call vertex of the cell $\Cl$ 
any non-empty intersection of $\Cl$ with $n$ other Voronoi cells.

We are in particular interested in the mean values of two characteristics of $\Cl$, namely its volume $\vol^{(M)}(\Cl)$ and its number of vertices that we denote by $\Nl$. To the best of our knowledge, this has been considered up to now in the Euclidean case and for two non-Euclidean manifolds only, namely the sphere and the hyperbolic space. In both cases the constant curvature implies the invariance of the Voronoi tessellation generated by $\Pl$ under the action of a specific transformation group and makes it possible to define a typical cell, equal in distribution to $\Cl$. For $k>0$ and $n\ge 1$, we denote by ${\mathcal S}_{k}^{n}$ the $n$-dimensional sphere centered at the origin and of radius $1/k$ and by ${\mathcal H}_{k}^{n}$ the hyperbolic space of curvature $-k^2$. For sake of simplicity, the $n$-dimensional sphere of curvature $1$ will be denoted by ${\mathcal S}^n$.

Besides explicit formulas in the particular case of the constant curvature, our main results include high intensity asymptotics for the mean volume and mean number of vertices of $\Cl$ and its sections. Calculating these estimates requires several fundamental assumptions on the Riemannian manifold that have to be made for the whole paper. They are the following:
\begin{align*}
\mbox{($\mbox{A}_1$)} & \mbox{ the sectional curvatures of $M$ are uniformly bounded from above and from below},\\
\mbox{($\mbox{A}_2$)} & \mbox{ $M$ has a global injectivity radius},\\
\mbox{($\mbox{A}_3$)} & \mbox{ the number of geodesic balls containing $(n+1)$ fixed distinct points of $M$ in their boundaries is}\\& \mbox{ uni\-form\-ly bounded},\\  
\mbox{($\mbox{A}_4$)} & \mbox{ there exists $r_{\mbox{\tiny{max}}}>0$ such that the number of geodesic balls of radius less than $r_{\mbox{\tiny{max}}}$ and containing}\\& \mbox{ $(n+1)$ fixed distinct points of $M$ in their boundaries is at most $1$}.  
\end{align*}
The first two assumptions ($\mbox{A}_1$) and ($\mbox{A}_2$) are quite standard in Riemannian geometry and notably guarantee that several comparison theorems, in particular for the volume growth, can be applied. The third and fourth assumptions ($\mbox{A}_3$) and ($\mbox{A}_4$) are very specific to the substance of this paper and to the construction of Voronoi vertices at the intersection of $(n+1)$ different Voronoi cells. It says that there is only a finite number of circumscribed balls of a fixed $n$-dimensional simplex, that this number can be bounded independently of the simplex and that it is at most 1 if the radius of the ball is small enough. These properties might be subproducts of other more well-known results on Riemannian manifolds but to the best of our knowledge, this is not the case. Surprisingly, the question of describing the set of points which are equidistant from a finite number of fixed points is still largely open. In the rest of the paper, we will assume implicitly that the three conditions ($\mbox{A}_1$), ($\mbox{A}_2$), ($\mbox{A}_3$) and ($\mbox{A}_4$) are satisfied by $M$. Naturally, this includes the particular cases of ${\mathbb R}^n$, $\sph_k^n$ and ${\mathcal H}_k^n$.

In $\R^n$, it is well-known that the mean volume of the typical cell is equal to 
$1/\lambda$, see e.g. \cite[Theorem 7.2, case $s=d$]{Mol89}. This is due to the fact that there are, on average, $\lambda$ cells per unit volume. Similarly, Miles obtained the mean volume of the typical cell of the Voronoi tessellation generated by a fixed number of independent points uniformly distributed in ${\mathcal S}^2$ \cite{Mil71}. 

In Theorem \ref{thm:propvol}, we provide a general asymptotic expansion of $\E[\vol^{(M)}(\Cl)]$ up to the second order at high intensity as well as explicit formulas in the particular cases of ${\mathcal S}_{k}^n$ and ${\mathcal H}_{k}^n$. Henceforth, the equality $g(\lambda)= o(f(\lambda))$ means that $\lim_{\lambda \to \infty}\frac{g(\lambda)}{f(\lambda)}=0$.
\begin{thm}\label{thm:propvol}
(i)  
When $\lambda\to\infty$, we get
\begin{equation}\label{eq:volasymp}
 \E[\vol^{(M)}(\Cl)] = \frac{ 1}{\lambda} + o\left( \frac{1}{\lambda^{1+\frac 2n}}\right)
\end{equation}
(ii) For every $n\ge 2$, $k>0$, we get when $M={\mathcal S}_{k}^n$
$$\E[\vol^{({\mathcal S}_{k}^n)}({\mathcal C}_{x_0,\lambda}^{({\mathcal S}_{k}^n)}))] = \frac{ 1}{\lambda}\left(1-e^{-2\sigma_{n-1}W_{n-1}\frac{\lambda}{k^n}}\right)$$
and when $M={\mathcal H}_{k}^n$
$$\E[\vol^{({\mathcal H}_{k}^n)}({\mathcal C}_{x_0,\lambda}^{({\mathcal H}_{k}^n)}))] = \frac{ 1}{\lambda}$$
where $\sigma_{n}=\vol^{({\mathcal S}^n)}({{\mathcal S}^n})=\frac{2\pi^{\frac{n+1}{2}}}{\Gamma\left(\frac{n+1}{2}\right)}$ is the volume of the unit-sphere ${\mathcal S}^{n}$ and $W_{n}=\frac{\Gamma\left(\frac{n+1}{2}\right)\Gamma\left(\frac12\right)}{2\Gamma\left(\frac{n+2}{2}\right)}$ is the $n$-th Wallis integral $\displaystyle\int_0^{\frac{\pi}{2}}\sin^n(t) dt$.
\end{thm}
Remarkably the estimate \eqref{eq:volasymp} shows that this mean volume does not depend, at first and  second order, on the geometry of the manifold. In order to capture the effects of the local geometry of the manifold on the Voronoi cell, we now focus on the mean number of vertices of $\Cl$.

In $\R^n$, because of the scaling invariance of the Poisson point process, the mean number of vertices of the typical cell does not depend on the intensity $\lambda$ \cite[Theorem 7.2, case $s=0$]{Mol89} and is equal to 
\begin{equation}
  \label{eq:Nmoyeuc}
\E[N({\mathcal C}_{x_0,\lambda}^{(\R^n)})]= 2 \pi^{\frac{n-1}2} n^{n-2} \left(\frac{\Gamma(\frac n2)}{\Gamma(\frac{n+1}2)} \right)^n \frac{ \Gamma( \frac{n^2+1}2)}{\Gamma( \frac{n^2}2)}.  
\end{equation}
In the case of the two-dimensional sphere, Miles \cite{Mil71} obtains the mean number of vertices of the typical cell when the Voronoi tessellation is generated by a fixed number of independent points uniformly distributed in the sphere and conditional on the event that the points are not included in a half-sphere. This result is a consequence of Euler's formula applied to the convex hull of the random points. We can easily deduce from this work the following formula for the mean number of vertices when the Voronoi tessellation is generated by a homogeneous Poisson point process in  ${\mathcal S}_{k}^{2}$: 
\begin{equation}\label{sph}
\E\left[N({\mathcal C}_{x_0,\lambda}^{({\mathcal S}_{k}^2)})\right] =  6 - \frac{3k^2}{\pi \lambda} + e^{-\frac{4\pi\lambda}{k^2}}\left( 6 + \frac{3k^2}{ \pi \lambda}\right).
\end{equation}
Three decades later, a very simple exact formula for the mean number of vertices of $\Cl$ when $M={\mathcal H}^2_{k}$  is derived by Isokawa \cite{Iso2000}:
\begin{equation}\label{eq:hyp}
\E\left[N({\mathcal C}_{x_0,\lambda}^{({\mathcal H}^2_{k})})\right] =  6+\frac{3k^2}{\pi\lambda}.
\end{equation}
Simultaneously he gets an integral formula  in the case of ${\mathcal H}^3_{k}$ \cite{Iso00b}. These results are mainly based on the existence of exact hyperbolic trigonometric formulas and can hardly be directly extended to general manifolds.

 
Theorem \ref{THM} yields a two-term asymptotic expansion of the mean number of vertices of $\Cl$ at high intensity. 
 \begin{thm}\label{THM}
(i) 
When $\lambda\to\infty$, we get
\begin{equation}\label{th1} \E[\Nl] = e_n - d_n\frac{\Scal_{x_0}^{(M)}}{\lambda^{\frac 2n}} + o\left(\frac{1}{\lambda^{\frac 2n}}\right) \end{equation}
where the constant $e_n=\E[N({\mathcal C}_{x_0,\lambda}^{(\R^n)})]$ is given by \eqref{eq:Nmoyeuc}, 
\begin{align}
d_n &=\frac{\pi^{\frac{n-3}2}n^{n+\frac{2}{n}-1}}{n!2^{\frac{2}{n}} (n+2)} \frac{ \Gamma( n+ \frac 2n) \Gamma(\frac n2)^{n+\frac 2n} \Gamma( \frac{n^2 +1}2)}{\Gamma( \frac{n^2}2) \Gamma( \frac{n+1}2)^n},\label{eq:defdn}
\end{align}
and $\Scal_{x_0}^{(M)}$ is the scalar curvature of $M$ at $x_0$.\\~\\
(ii) For every $n\ge 2$, $k>0$, we get when $M={\mathcal S}_{k}^n$
$$\E[N({\mathcal C}_{x_0,\lambda}^{(\sph_k^n)})]=\sigma_{n-1}\int_0^{\lambda^{\frac{1}{n}}\frac{\pi}{k}}g_\lambda^{(\sph_k^n)}(r,u)\mathrm{d}r$$
and when $M={\mathcal H}_k^n$
$$\E[N({\mathcal C}_{x_0,\lambda}^{(\mathcal{H}_k^n)})]=\sigma_{n-1}\int_0^{\infty}g_\lambda^{({\mathcal H}_k^n)}(r,u)\mathrm{d}r$$
where $u$ is fixed in $\sph^{n-1}$ and $g_\lambda^{(\sph_k^n)}(r,u)$ and $g_\lambda^{({\mathcal H}_k^n)}(r,u)$ are given by \eqref{eq:densitysphere} and \eqref{eq:densityhyperb} respectively.
\end{thm}
The first values of the constants $e_n$ and $d_n$ are $e_2=6$, $e_3=\frac{96\pi^2}{35}$, $e_4=\frac{1430}{9}$ and $d_2=\frac{3}{2\pi}$, $d_3=\frac{12\cdot3^{\frac{2}{3}}\pi^{\frac{4}{3}}\Gamma\left(\frac{11}{3}\right)}{175\cdot 2^{\frac{1}{3}}}$, $d_4=\frac{25025}{864\sqrt{2\pi}}$. In particular, the asymptotic expansion \eqref{th1} shows that $\E[\Nl] $ converges to the constant $e_n$ which is naturally consistent with the Euclidean case \eqref{eq:Nmoyeuc}. Moreover, since the scalar curvature is twice the Gaussian curvature in dimension two, it is also consistent with the exact values \eqref{sph} and \eqref{eq:hyp} obtained in the respective cases of the two-dimensional sphere and of the hyperbolic plane. 

One of the key tools for proving \eqref{th1} is an extension of a renowned spherical change of variables formula of Blaschke-Petkantschin type, proved in the Euclidean space \cite[Chapter 7]{Wei08} and in the case of the sphere \cite{Mil71bis}. We calculate in Theorem \ref{Jacob} an asymptotic expansion of the corresponding Jacobian in the case of a general Riemannian manifold. This result is a close companion to a previous similar formula contained in \cite{AurelieNote} though the two underlying transformations are different. 

As expected, the (scalar) curvature only appears from the second term on in \eqref{th1}. This expansion will be the basis for the construction of an estimator of the scalar curvature which satisfies limit theorems, see \cite{CCElimite}. Nevertheless, the asymptotic expansion of $\E[\Nl]$ does not capture the possible anisotropy of the metric, which implies that
its mere knowledge is not enough to recover the metric for a manifold of dimension $n\ge 3$. This suggests that it will be necessary to study the set of vertices of $\Cl$ in a fixed direction in order to get the Ricci curvature.

Let us introduce the point process $\mathcal{V}_{x_0,\lambda}^{(M)}$ of normalized vertices of $\Cl$ as follows:
\[ \mathcal{V}_{x_0,\lambda}^{(M)}= \{ (r,u) \in \R_+ \times \sph^{n-1} : \exp_{x_0}( \lambda^{-\frac 1n}r u) \text{ is a vertex of }\Cl \} \] 
The renormalization factor $\lambda^{-\frac 1n}$ is due to the fact that the volume of $\Cl$ is of order $\lambda^{-1}$ by \eqref{eq:volasymp} and that its vertices are at a distance of order $\lambda^{-\frac1n}$ from $x_0$. In Theorem \ref{DENS} below, we provide an asymptotic expansion of the density of the intensity measure of the point process $\mathcal{V}_{x_0,\lambda}^{(M)}$ as well as explicit formulas in the particular cases $M={\mathcal S}_{k}^n$ and $M={\mathcal H}_{k}^n$.
\begin{thm}\label{DENS}
(i) 
When $\lambda\to\infty$, the density denoted by $g_\lambda^{(M)}(r,u)$ of the intensity measure of the point process $\mathcal{V}_{x_0,\lambda}^{(M)}$ satisfies 
\begin{align*} 
& g_\lambda^{(M)}(r,u)\\&=e^{-\kappa_n r^n} \left[ a_n r^{n^2-1} - \left(\frac{1}{6}\left(a_n\Ric_{x_0}^{(M)}(u) - n\Delta_{\Ric,n}(u)\right) r^{n^2+1}+ b_n{ \Scal_{x_0}^{(M)}}r^{n^2+n+1}\right)\frac{1}{\lambda^{\frac 2n}} + o\left( \frac{1}{\lambda^{ \frac 2n}} \right)\right] \end{align*}
where $\kappa_n$ is the Euclidean volume of the $d$-dimensional unit ball, i.e. $ \kappa_n=\frac{2 \pi^{\frac n2}}{n \Gamma\left(\frac n2 \right)}$,
\begin{align*}
a_n &= \frac{2^n \pi^{\frac{n^2-1}2} \Gamma\left(\frac{n^2+1}2\right) \Gamma\left( \frac n2 \right)}{ n!\Gamma\left(\frac{n+1}2\right)^n \Gamma\left( \frac{n^2}2 \right)},~~ b_n =  \frac{2^{n} \Gamma\left( \frac{n^2+1}2\right) \pi^{\frac{n^2+n-1}2}}{3n! n(n+2)\Gamma\left( \frac{n^2}2 \right) \Gamma\left( \frac{n+1}2\right)^n}, 
\end{align*}
$\Ric_{x_0}^{(M)}(u)$ denotes the Ricci curvature of $M$ at $x_0$ in direction $u$
and \[ \Delta_{\Ric,n}(u) = \int_{u_1,\dots u_n \in {\mathcal S}^{n-1}}\Delta(u,u_1,\dots,u_n)\Ric_{x_0}^{(M)}(u_1) \dvol^{({\mathcal S}^{n-1})}(u_1) \dots \dvol^{({\mathcal S}^{n-1})}(u_n), \] 
with $\Delta(u,u_1,\dots,u_n)$ as the Euclidean volume of the simplex spanned by $-u,u_1,\dots,u_n$.\\~\\
(ii) For every $n\ge 2$, $k>0$, we get when $M={\mathcal S}_{k}^n$  and $r\le \lambda^{\frac 1n}\frac{\pi}{k}$
\begin{equation}
  \label{eq:densitysphere}
 g_\lambda^{({\mathcal S}_{k}^n)}(r,u)=a_n\bigg(e^{-\sigma_{n-1} \displaystyle\int_{0}^{r}{\mathfrak{s}}_{k\lambda^{-\frac 1n}}^{n-1}(t) \mathrm{d}t } +e^{- \sigma_{n-1} \displaystyle\int_{r}^{\lambda^{\frac 1n} \frac{\pi}{k}}
{\mathfrak{s}}_{k\lambda^{-\frac 1n}}^{n-1}(t) \mathrm{d}t }\bigg)
{\mathfrak{s}}_{k\lambda^{-\frac 1n}}^{n^2-1}(r)  
\end{equation}
and when $M={\mathcal H}_{k}^n$ and $r>0$
 \begin{equation}
   \label{eq:densityhyperb}
 g_\lambda^{({\mathcal H}_{k}^n)}(r,u)=a_ne^{-\sigma_{n-1} \displaystyle\int_{0}^{r}{\mathfrak{s}}_{-k\lambda^{-\frac 1n}}^{n-1}(t) \mathrm{d}t }
{\mathfrak{s}}_{-k\lambda^{-\frac 1n}}^{n^2-1}(r),  
 \end{equation}
where ${\mathfrak{s}}_{\alpha}(t)=\begin{cases} 
                        \frac{\sin(\alpha t)}{\alpha} &\text{ if }\alpha>0 \\
												t & \text{ if } \alpha=0 \\
												\frac{\sinh((-\alpha)t)}{-\alpha} &\text{ if }\alpha<0 
											\end{cases}. $
\end{thm}
Again, the first term of the asymptotic expansion of $g_{\lambda}^{(M)}(r,u)$ is equal to $g_{\lambda}^{({\mathbb R}^n)}(r,u)$. The particular value $g_{\lambda}^{({\mathbb R}^n)}(r,u)$ is the well-known density of the circumscribed radius of the typical Poisson-Delaunay cell in the Euclidean space, see e.g. \cite[Proposition 4.3.1]{Mol94}.

The second term of the expansion in Theorem \ref{DENS} involves the Ricci curvature at $x_0$ in direction $u$ which provides information on the anisotropy of the manifold around $x_0$. Nevertheless, the quantity $\Delta_{\Ric,n}(u)$ appears at the same order so in practice, we are unfortunately not able to discriminate between them and deduce the Ricci curvature from the knowledge of the density.

We observe that the sectional curvatures are missing from the expansions contained in Theorems \ref{THM} and \ref{DENS}. The most natural way to find them consists in studying a section of the Poisson-Voronoi tessellation. When $M$ is the Euclidean space, such a section is the intersection of the Poisson-Voronoi tessellation with a linear subspace of dimension $s\le n$. Note that, in general, the resulting tessellation is not a Voronoi tessellation of the linear subspace \cite{chi96}. Nevertheless, it is still a stationary tessellation which makes it possible to define a typical cell of the sectional tessellation. Miles \cite{Mil84} provides explicit formulas for the mean $k$-dimensional content of the $k$-skeleton of this typical cell for $0\le k \le s$, see also \cite[Theorem 7.2]{Mol89}. 
On a general Riemannian manifold $M$, 
we need to define a different local notion of sectional Voronoi tessellation. Let $V_s$ be a linear subspace of dimension $s$ of the tangent space of M at $x_0$ denoted by $T_{x_0}M$. The image of $V_s$ by the exponential map at $x_0$, defined in Section 2, is a manifold of dimension $s$, denoted by $M_s$. We notice immediately that $M_s$ satisfies the two assumptions ($\mbox{A}_1$) and ($\mbox{A}_2$). We define the sectional tessellation as the intersection of the tessellation with the manifold $M_s$ 
and investigate the characteristics of the section of the cell $\Cl$, i.e. the set $\Cl\cap M_s$. Note that in the Euclidean case, this corresponds to the section of the typical cell which does not have the same distribution as the typical cell of the sectional tessellation. In Theorem \ref{volSec} below, we provide an asymptotic expansion at high intensity for the mean s-content $\vol^{(M_s)}(\Cl\cap M_s)$. 
\begin{thm}\label{volSec}
(i) Let us fix $1\le s \le n$ and let $V_s$ be a linear subspace of dimension $s$ of $T_{x_0}M$. When $\lambda\to\infty$, the mean volume of the section of $\Cl$ by $V_s$ satisfies
\[ \E[\vol^{(M_s)}(\Cl\cap M_s)] = v_{n,s}\frac{1}{\lambda^{\frac sn}}
+\left(u_{n,s}{\Scal_{x_0}^{(M)}}- w_{n,s}{\Scal_{x_0}^{(M_s)}} \right)\frac{1}{\lambda^{\frac{s+2}n}} + o\left(\frac{1}{\lambda^{\frac{s+2}n}}\right)\]
where
\begin{align*}
v_{n,s} &=\frac{2^{1-\frac{s}{n}}n^{\frac sn -1} \Gamma(\frac sn) \Gamma(\frac n2)^{\frac sn}}{\Gamma(\frac s2)},~~
u_{n,s} &=\frac{(s+2)n^{\frac{s+2}n-2}\Gamma(\frac{s+2}n)\Gamma(\frac n2)^{\frac{s+2}n}}{3\Gamma(\frac s2)(n+2)2^{\frac{s+2}n} \pi},~~
w_{n,s} =\frac{n^{\frac{s+2}n-1}\Gamma(\frac{s+2}n)\Gamma(\frac n2)^{\frac{s+2}n}}{3\Gamma(\frac s2)s2^{\frac{s+2}n} \pi}
\end{align*}
and $\Scal_{x_0}^{(M_s)}$ is the scalar curvature of $M_s$ at $x_0$.\\~\\
(ii) For every $n\ge 2$, $k>0$, we get when $M=\sph_k^{n}$
\[ \E[\vol^{(M_s)}({\mathcal C}_{x_0,\lambda}^{(\sph_k^n)}\cap M_s)] =\sigma_{s-1}\int_0^{\frac{\pi}{k}}e^{-\lambda \sigma_{n-1}\displaystyle\int_0^r \mathfrak{s}_{k}(t)\mathrm{d}t}\mathfrak{s}_{k}^{s-1}(r)\mathrm{d}r.\]
and  when $M={\mathcal H}_k^{n}$
\[ \E[\vol^{(M_s)}({\mathcal C}_{x_0,\lambda}^{(\mathcal{H}_k^n)}\cap M_s)] =\sigma_{s-1}\int_0^{\infty}e^{-\lambda \sigma_{n-1}\displaystyle\int_0^{r} \mathfrak{s}_{-k}(t)\mathrm{d}t}\mathfrak{s}_{-k}^{s-1}(r)\mathrm{d}r.\]
\end{thm}
The first term of the expansion, namely the term $v_{n,s}\lambda^{-\frac{s}{n}}$, is naturally equal to $\E[\vol^{({\mathbb R}^s)}({\mathcal C}_{x_0,\lambda}^{({\mathbb R}^n)}\cap {\mathbb R}^s)]$. 
 To the best of our knowledge, this is the first time that such mean characteristic of the Euclidean typical Poisson-Voronoi is calculated though it is certainly very close in spirit to the calculation of the mean $k$-th volume of the $k$-dimensional typical face of a stationary and isotropic Poisson hyperplane tessellation, see e.g. \cite[Theorem 10.3.3]{Wei08} and \cite{Sch09}.

Moreover, Theorem \ref{volSec} is consistent with \eqref{eq:volasymp} since $u_{n,n}=w_{n,n}$ and $v_{n,n}=1$. Surprisingly, while the expansion up to the second term of the mean volume of $\Cl$ provided at \eqref{eq:volasymp} is independent of the manifold $M$, a similar calculation for the section of $\Cl$ involves both the scalar and sectional curvatures at $x_0$. 

Theorem \ref{NSec} contains a similar asymptotic expansion for the mean number of vertices $N(\Cl\cap M_s)$ of $\Cl\cap M_s$.
\begin{thm}\label{NSec}
Let us assume that $M_s$ satisfies the two assumptions $(\mbox{A}_3)$ and $(\mbox{A}_4)$.\\~\\
(i) Let $1\le s \le n$ and let $V_s$ be a linear subspace of dimension $s$ of $T_{x_0}M$. When $\lambda\to\infty$, the mean number of vertices of the section of $\Cl$ by $V_s$ satisfies
\[ \E[N(\Cl\cap M_s)] = e_{n,s} +\left(f_{n,s}{\Scal_{x_0}^{(M)}} - g_{n,s} {\Scal_{x_0}^{(M_s)}} - h_{n,s} {\Delta_{\Ric,n,s}} \right)\frac{1}{\lambda^{\frac 2n}} + o\left( \frac{1}{\lambda^{\frac 2n}}\right)\]
where
\begin{align}
e_{n,s}&= \frac{ 2 \pi^{\frac{s}2} n^{s-1} \Gamma(\frac n2)^2 \Gamma\left( \frac 12(ns+n-s+1) \right)}{s\Gamma\left(\frac s2\right) \Gamma\left( \frac 12(ns+n-s) \right)\Gamma\left( \frac{n+1}2 \right)\Gamma\left( \frac{n-s+1}2 \right)} \label{eq:enssimp}\\
f_{n,s}&= \frac{ 2^{-\frac 2n} \pi^{\frac{s-1}2}(sn+2)\Gamma(s+\frac 2n) n^{s+\frac 2n-2}\Gamma(\frac n2)^{2+\frac 2n} \Gamma\left( \frac 12(ns+n-s+1) \right)}{3 (n+2)s! \Gamma\left(\frac s2\right)\Gamma\left( \frac 12(ns+n-s) \right)\Gamma\left( \frac{n+1}2 \right)\Gamma\left( \frac{n-s+1}2 \right) },\label{eq:fnssimp}\\
g_{n,s}&= \frac{\Gamma(s+\frac 2n) n^{s+\frac 2n-2}\Gamma(\frac n2)^{2+\frac 2n} \Gamma\left( \frac 12(ns+n-s+1) \right)}{3s!s(n+2)2^{\frac 2n-1}\pi\Gamma\left(\frac s2\right)^2 \Gamma\left( \frac 12(ns+n-s) \right)\Gamma\left( \frac{n+1}2 \right)\Gamma\left( \frac{n-s+1}2 \right)},\label{eq:gnssimp}\\
h_{n,s}&=\frac{ \Gamma(s+\frac 2n) n^{s+\frac 2n-2}\Gamma(\frac n2)^{s+\frac 2n}}{3 (n+2) 2^{s+\frac 2n+1} \pi^{\frac{sn}2+1}}\\ 
\Delta_{\Ric,n,s} &= \int_{u\in\sph^{s-1}} \int_{u_1,\dots,u_s \in \sph^{n-1}}  \Delta_s(u,u_1,\dots,u_s)\Ric(u_1)\dvol^{({\mathcal S}^{n-1})}(u_1)\dots \dvol^{({\mathcal S}^{n-1})}(u_s)\dvol^{({\mathcal S}^{s-1})}(u)\nonumber
\end{align}
 with $\Delta_s(u,u_1,\dots,u_s)$ being the $s$-dimensional volume of the simplex generated by $-u\in V_s$ and the projection of $u_1,\dots,u_s$ onto the subspace $V_s$.\\~\\
(ii) For every $n\ge 2$, $k>0$, we get when $M=\sph_k^n$
\begin{equation}
  \label{eq:densityspheresec}
\E[N({\mathcal C}_{x_0,\lambda}^{(\sph_k^n)}\cap M_s)] =\Delta_{n,s}\int_0^{\lambda^{\frac 1n}\frac{\pi}{k}}\bigg(e^{-\sigma_{n-1} \displaystyle\int_{0}^{r}{\mathfrak{s}}_{k\lambda^{-\frac 1n}}^{n-1}(t) \mathrm{d}t } +e^{- \sigma_{n-1} \displaystyle\int_{r}^{\lambda^{\frac 1n} \frac{\pi}{k}}
{\mathfrak{s}}_{k\lambda^{-\frac 1n}}^{n-1}(t) \mathrm{d}t }\bigg)
{\mathfrak{s}}_{k\lambda^{-\frac 1n}}^{sn-1}(r)\mathrm{d}r  
\end{equation}
and when $M={\mathcal H}_k^n$
 \begin{equation}
   \label{eq:densityhyperbsec}
\E[N({\mathcal C}_{x_0,\lambda}^{(\mathcal{H}_k^n)}\cap M_s)] =\Delta_{n,s}\int_0^{\infty}e^{-\sigma_{n-1} \displaystyle\int_{0}^{r}{\mathfrak{s}}_{-k\lambda^{-\frac 1n}}^{n-1}(t) \mathrm{d}t }
{\mathfrak{s}}_{-k\lambda^{-\frac 1n}}^{sn-1}(r)\mathrm{d}r
\end{equation}
where $$\Delta_{n,s}=\frac{ 2^{s+1} \pi^{\frac{s+ns}2} \Gamma\left( \frac 12(ns+n-s+1) \right)}{s!\Gamma\left(\frac s2\right)\Gamma\left( \frac n2 \right)^{s-2} \Gamma\left( \frac 12(ns+n-s) \right)\Gamma\left( \frac{n+1}2 \right)\Gamma\left( \frac{n-s+1}2 \right) }.$$
\end{thm}
The limit $e_{n,s}$ of $\E[N(\Cl\cap M_s)]$ is naturally equal to $\E[N({\mathcal C}_{x_0,\lambda}^{({\mathbb R^n})}\cap M_s)]$. To the best of our knowledge, the explicit value of $e_{n,s}$ is new, as well as the formulas \eqref{eq:densityspheresec} and \eqref{eq:densityhyperbsec} in the particular case of constant curvature.
Compared to Theorem \ref{thm:propvol}, Theorems \ref{THM} and \ref{DENS}, the statements contained in Theorems \ref{volSec} and \ref{NSec} involve all the local characteristics of the metric around $x_0$. In the particular case $s=2$, the curvature $\Scal_{x_0}^{(M_2)}$ is indeed twice the sectional curvature of $M$ at $x_0$ with respect to the plane $L_2$. This observation is in a way completely satisfying in regard of our initial purpose of recovering the local geometry of $M$ from the properties of the Poisson-Voronoi tessellation. Let us note, however, that their usefulness in practice is questionable since the determination of the section of the tessellation already requires the knowledge of the metric.

So far, we focused on the link between the characteristics of the Poisson-Voronoi tessellation and the local geometry around a point. Now, a natural question arises: can we get global information on the geometry of the manifold? Since the mean number of vertices of the cell associated with an extra nucleus at $x_0$ involves the scalar curvature at $x_0$, one could imagine that the mean number of vertices in the whole tessellation is connected to the integral of the curvature on the whole manifold $M$ when $M$ is a compact set. This is indeed the case,  and this fact actually implies, in the case of a compact manifold of dimension 2 without boundary, a probabilistic proof of the Gauss-Bonnet theorem: 

\begin{cit}[Gauss-Bonnet]
For a compact surface without boundary $M$, 
\[ \chi(M) = \frac{1}{2\pi} \int_{M} K(x) \dM(x) \]
where $\chi(M)$ denotes the Euler characteristic of $M$ and $K(x)=\frac{1}{2}\Scal_{x_0}^{(M)}$ is the Gaussian curvature at $x\in M$.
\end{cit}
There are naturally many classical proofs of this theorem.
They are in general rather technical, i.e. they rely notably on a triangulation of the surface and a clever application of Stokes' theorem in each triangle, see e.g. \cite[Chapter 9]{Lee06}. In our opinion, the proof that we provide in this paper is to some extent shorter and more elementary though it is based on an application of Theorem \ref{THM} in dimension 2. To the best of our knowledge, this argument is new though Leibon used a heuristic reasoning which led him to the intuition of the existence of such probabilistic proof \cite{Lei02}.

The paper is structured as follows. 
In Section \ref{sec:geom}, we introduce several fundamental tools from Riemannian geometry. Section \ref{sec:locdep} is devoted to showing that we can assume without loss of generality that $M$ is a compact Riemannian manifold. The calculation of $\E[\vol^{(M)}(\Cl)]$ and the proof of Theorem \ref{thm:propvol} take place in Section \ref{sec:vol}. Section \ref{sec:BP} is devoted to the statement and proof of Theorem \ref{Jacob} which is a new integral formula of Blaschke-Petkantschin type. In Section \ref{sec:vert}, we prove Theorems \ref{THM} and \ref{DENS} dealing with both the cardinality and density of the set $\mathcal{V}_{x_0,\lambda}^{(M)}$ of vertices of $\Cl$. We concentrate on the sectional tessellation in Section \ref{sec:sec} and prove
Theorems \ref{volSec} and \ref{NSec} therein. Finally, we postpone to Section \ref{sec:GB} the details for the probabilistic proof of the Gauss-Bonnet theorem which is deduced from Theorem \ref{THM}. 

This document is a comprehensive account of some of the results contained in a thesis manuscript \cite{ChaThese}. An abridged and more to the point version is to be submitted soon for publication.
\section{Geometric framework and preliminaries}
\label{sec:geom}
In this section, we introduce some useful notation and we survey several fundamental definitions and results from the theory of Riemannian geometry. For more details, we refer the reader to the reference books such as \cite{Do92}, \cite{Lee97} and \cite{Ber03}.
\paragraph{Exponential map.} Let $x_0 \in M$ and $v\in T_{x_0} M$. There exists a unique geodesic $\gamma_v$ such that $\gamma_v(0) =x_0$ and $\gamma'_v(0) =v$. Since $M$ is a complete Riemannian manifold, Hopf-Rinow theorem \cite[Theorem 52]{Ber03} guarantees that this geodesic is well defined on $[0,\infty)$. The exponential map of $v$ at $x_0$, denoted by $\exp_{x_0}(v)$ is then defined by the identity
\begin{equation}\label{exp}
\exp_{x_0}(v) = \gamma_v(1).
\end{equation}
For sake of simplicity, we omit in the notation of the exponential map the dependency on the manifold $M$ which should be implicit anyway. Intuitively, for a tangent vector
$v$, we get $\exp_{x_0}(v)$ by travelling on $M$ along the geodesic starting from $x_0$ in the direction given by $v$
over a length $\|v\|_{x_0}$. 
The radius of the largest open ball on which the exponential map is a diffeomorphism 
is called the injectivity radius that we denote by $R_{\mbox{\tiny inj}}$. The infimum of all injectivity radii over all $x_0\in M$ is called the global injectivity radius, assumed to be positive. For sake of simplicity, we will make a slight abuse of notation by calling $R_{\mbox{\tiny inj}}$ either the injectivity radius of a fixed point $x_0$ or the global injectivity radius. This allows us to define geodesic spherical coordinates of a point $x$ in a small neighbourhood of ${x_0}$ by $x=\exp_{x_0}(r u)$ with $r \in (0,R_{\mbox{\tiny inj}})$ and $u$ being a unit vector of $T_{x_0}M$.

\paragraph{Curvatures.} Let $\sigma_{x_0}$ be a plane of $T_{x_0} M$. The sectional curvature at ${x_0}$ with respect to $\sigma_{x_0}$, denoted by $K_{x_0}^{(M)}(\sigma_{x_0})$ is the Gaussian curvature at ${x_0}$ of the surface $\exp_{x_0}(\sigma_{x_0})$. For any vectors $u,v$ in $T_{x_0} M$ linearly independent, we write
$ K_{x_0}^{(M)}(u,v)=K_{x_0}^{(M)}(\sigma_{x_0})$, where $\sigma_{x_0}$ is the plane spanned by $u$ and $v$.  Sectional curvatures are particularly interesting because the knowledge of $K_{x_0}^{(M)}(\sigma_{x_0})$, for all $\sigma_{x_0}$, determines the metric at ${x_0}$ completely. The sectional curvature of $u$ and $v$ can also be defined through the identity 
\begin{equation}
  \label{eq:defseccurv}
 K_{x_0}^{(M)}(u,v)= \langle v, \mathcal{R}^{(M)}_{x_0}(u,v)u\rangle_{x_0}
\end{equation}
where $\mathcal{R}^{(M)}_{x_0}$ is the Riemann curvature tensor of $M$ at $x_0$. 

Let $u$ be a unit vector of $T_{x_0} M$ and let us extend it to an orthonormal basis $\{ u_1, \dots, u_{n-1},u \}$ of $T_{x_0} M$. The Ricci curvature of $M$ at $x_0$ in direction $u$, denoted by $\Ric_{x_0}^{(M)}(u)$ is defined by the identity
\begin{equation}\label{eq:defric}
 \Ric_{x_0}^{(M)}(u) = \sum_{i=1}^{n-1} K_{x_0}^{(M)}(u,u_i) 
\end{equation}
and the scalar curvature of $x_0$ is defined by 
\begin{equation}
  \label{eq:defscalaire}
\Scal_{x_0}^{(M)} = \sum_{i=1}^{n} \Ric_{x_0}^{(M)}(u_i)  
\end{equation}
with the convention $u_n=u$.
Note that $\Ric_{x_0}^{(M)}(u)$ and $\Scal_{x_0}^{(M)}$ do not depend on the choice of the basis.
Geometrically, the scalar curvature measures the volume defect between geodesic balls of small radius in $M$ and Euclidean balls with same radius. Let us denote by $\bo^{(M)}(x_0,r)$ the open geodesic ball in $M$ centered at $x_0$ and of radius $r>0$. We provide below a two-term expansion for the volume of $\bo^{(M)}(x_0,r)$ when $r$ is small, known as the Bertrand-Diquet-Puiseux theorem, see also \cite{Gra73} for subsequent terms:
\begin{equation}\label{vol}
\vol^{(M)}(\bo^{(M)}(x_0,r))= \kappa_n r^n - \frac{ \kappa_n \Scal_{x_0}^{(M)}}{6(n+2)} r^{n+2} + o(r^{n+2})
\end{equation}
where $\kappa_n = \frac{2 \pi^{\frac n2}}{n \Gamma\left(\frac n2 \right)}$ is the volume of the Euclidean $n$-dimensional unit-ball.
Similarly, the Ricci curvature in a direction measures the volume defects between small cones of $M$ in that direction with corresponding ones in the Euclidean space, see e.g. \cite{Tao08}. The following integral formula relates the Ricci curvature to the scalar curvature and is a continuous analogue of \eqref{eq:defscalaire}, see e.g. \cite[Exercise 9 p. 107]{Do92}:
\begin{equation} \label{eq:scalRicci}\Scal_{x_0}^{(M)} =\frac{n}{\vol^{({\mathcal S}^{n-1})}(\sph^{n-1})} \int_{u\in \sph^{n-1}} \Ric_{x_0}^{(M)}(u) \dvol^{(\sph^{n-1})}(u).\end{equation}

\paragraph{Jacobi fields.} A Jacobi field along a geodesic $\gamma$ is a vector field $J$ verifying the Jacobi equation 
\begin{equation}\label{eqJacobi} J''(t) = \mathcal{R}^{(M)}_{\gamma(t)}(\gamma'(t),J(t)) \gamma'(t) \end{equation}
where the derivative of $J$ is understood in the sense of the covariant derivative with respect to the Levi-Civita connection, see e.g. \cite[Chapter 5, \S 2]{Do92}. In particular, along $\gamma$, there exists a unique Jacobi field with given $J(0)$ and $J'(0)$. There are several ways to obtain Jacobi fields but in this paper, we only use the fact that they are connected to the derivative of the exponential map. We recall without proof the following general result which makes this connection more precise \cite[p. 119]{Do92}. This is a key tool of the proof of Theorem \ref{Jacob}.
\begin{lem}\label{lem:exodocarmo}
Let $\gamma$ be a geodesic, $\mathfrak{c}$ a curve on $M$ such that $\mathfrak{c}(0)=\gamma(0)$ and $V$ a vector field along $\mathfrak{c}$ such that $V(0)= \gamma'(0)$. Then the function $f(t,s) = \exp_{\mathfrak{c}(s)}(t V(s))$ satisfies
\begin{equation}\label{deriv}
\frac{ \partial f}{\partial s} (t,0) = J(t)
\end{equation}
where $J$ is the unique Jacobi field along $\gamma$ with $J(0)= \mathfrak{c}'(0)$ and $J'(0)=V'(0)$.   
\end{lem}
The Rauch comparison theorem yields bounds for the norm of a Jacobi field and as a consequence, we get the following expansion for any Jacobi field $J$ along $\gamma$ which satisties $J(0)=0$, see \cite[Chapter 5]{Do92}:
\begin{equation}\label{field} \|J(t)\|_{\gamma(t)} = t - \frac{K_{\gamma(0)}^{(M)}(J'(0),\gamma'(0))}6 t^3 +o(t^3). \end{equation}
Though we will not use this actual estimate in the sequel, we will prove similar expansions for  scalar products involving Jacobi fields, see in particular Lemma \ref{lem:expCBi}.
\paragraph{Parallel transport.} Let $V$ be a vector field along a curve $\gamma$. $V$ is called a parallel vector field if $V'(t) =0$ for all $t$, in the sense of the covariant derivative. Now, for any $u \in T_{\gamma(0)}$, there exists a unique parallel vector field $V$ along $\gamma$ such that $V(0)=u$ which is called parallel transport of $u$ along $\gamma$. In this paper, $V(t)$ will be denoted by $P_{\gamma(0)\to\gamma(t)}(u)$ and sometimes by $u(t)$ or even $u$ itself with a slight abuse of notation. Note that the parallel transport is a linear isomorphism from $T_{\gamma(0)}M$ to $T_{\gamma(t)}M$ which preserves the scalar product. 
\paragraph{Jacobian of the spherical change of variables.} Let us consider the following transformation into spherical coordinates:
\begin{equation}\label{eq:changeSpherical} \varphi_{x_0} : \left\{\begin{array}{rll}(0,R_{\mbox{\tiny inj}})\times \{u\in T_{x_0}M: \|u\|_{x_0}=1\}&\longrightarrow& M\\(r,u)&\longmapsto &\exp_{x_0}(ru)\end{array}\right.\end{equation}
and let us denote by $\J_{x_0}^{(M)}(r,u)$ the associated Jacobian determinant. In particular, the volume element $\dvol^{(M)}(x)$ satisfies 
\begin{equation}\label{eq:Jru}
\dvol^{(M)}(x) = |\J_{x_0}^{(M)}(r,u)| \dd r \dvol^{(\sph^{n-1})}(u).
\end{equation}
The following lemma shows the asymptotic expansion of $\J_{x_0}^{(M)}(r,u)$ when $r\to 0$ as well as a uniform lower bound for $r$ small enough.
\begin{lem}\label{lem:jacobchtsph}
(i) When $r\to 0$,
\[ \J_{x_0}^{(M)}(r,u) = r^{n-1} - \frac{ \Ric^{(M)}_{x_0}(u)}{6}r^{n+1} + o(r^{n+1}).\]
and $$\sup_{u\in T_{x_0}M, \|u\|_{x_0}=1}r^{-(n+1)}|\J_{x_0}^{(M)}(r,u) - r^{n-1} + \frac{ \Ric^{(M)}_{x_0}(u)}{6}r^{n+1}|\underset{r\to 0}{\to}0.$$
(ii) In the particular case when $M$ is a compact Riemannian manifold, we get the additional result
$$\sup_{x_0\in M}\;\sup_{u\in T_{x_0}M, \|u\|_{x_0}=1}r^{-(n+1)}|\J_{x_0}^{(M)}(r,u) - r^{n-1} + \frac{ \Ric^{(M)}_{x_0}(u)}{6}r^{n+1}|\underset{r\to 0}{\to}0.$$
(iii) For any Riemannian manifold $M$, there exists $r_0>0$ such that for every $x_0\in M$, $u\in T_{x_0}M$  of norm $1$  and $r\in [0,r_0]$, we get
\[\frac{1}{2}r^{n-1}\le \J_{x_0}^{(M)}(r,u)\le \frac{3}{2}r^{n-1}.\] 
\end{lem}
The proof of Lemma \ref{lem:jacobchtsph} relies on the exact calculation and Taylor expansion of each entry of the Jacobian determinant in terms of Jacobi fields. It is deferred to the Appendix.
\paragraph{Uniform estimates of the volume of small balls.} We need a refinement of the Bertrand-Diquet-Puiseux estimate given at \eqref{vol} which guarantees that the expansion is the same in a neighborhood of $x_0$ and that the remaining term $o(r^{n+2})$ is uniform with respect to the center of the ball. This is described in Lemma \ref{lem:volunif} below.
\begin{lem}
\label{lem:volunif}
(i) There exists $r_0>0$  such that for every $x\in {\mathcal B}^{(M)}(x_0,r_0)$, we get
 $$ \vol^{(M)}(\bo^{(M)}(x,r))= \kappa_n r^n - \frac{ \kappa_n \Scal_{x_0}^{(M)}}{6(n+2)} r^{n+2} + o(r^{n+2})$$
and 
$$\sup_{x\in {\mathcal B}^{(M)}(x_0,r_0)} r^{-(n+2)}\left(\vol^{(M)}(\bo^{(M)}(x,r))-\kappa_n r^n + \frac{ \kappa_n \Scal_{x_0}^{(M)}}{6(n+2)} r^{n+2}\right)\underset{r\to 0}{\to}0.$$
(ii) In the particular case when $M$ is a compact Riemannian manifold, we get the additional result
$$\sup_{x_0\in M}\sup_{x\in {\mathcal B}^{(M)}(x_0,r_0)} r^{-(n+2)}\left(\vol^{(M)}(\bo^{(M)}(x,r))-\kappa_n r^n + \frac{ \kappa_n \Scal_{x_0}^{(M)}}{6(n+2)} r^{n+2}\right)\underset{r\to 0}{\to}0.$$
\end{lem}
Lemma \ref{lem:volunif} is essentially based on the Bishop-Gromov theorem and related comparison inequalities. The proof is postponed to the Appendix.

In the final lemma, we exhibit a general lower bound for the volume of a ball in the particular case when $M$ is not a compact set. Here and in the sequel, $c$ denotes a generic positive constant which may change from line to line. 
\begin{lem}
\label{lem:borneinfvolboule} 
When $M$ is non-compact, there exists a constant $c>0$ such that for every $x_0\in M$ and $r>0$,
$$\vol^{(M)}({\mathcal B}^{(M)}(x_0,r))\ge c \min(r,r^n).$$
\end{lem}
This result is in the same spirit as the classical estimate due to Calabi and Yau in the case of positive Ricci curvature, see e.g. \cite[Section 0]{CK88} and \cite[(iii) p. 669]{Yau76}. 
Again, the proof of Lemma \ref{lem:borneinfvolboule} can be found in the Appendix.
\section{Reduction to the case when $M$ is compact}\label{sec:locdep}
This section aims at showing that it is enough to show Theorems \ref{thm:propvol}-\ref{NSec} in the particular case when $M$ is a compact Riemannian manifold. This is needed for the following reason: each of the considered expectations will be written as an integral of an integrand of type $e^{-\lambda \vol^{(M)}({\mathcal B}^{(M)}(x,r))}$ with respect to $\mathrm{d}\mbox{vol}^{(M)}$ or to the product $(\mathrm{d}\mbox{vol}^{(M)})^n$, see e.g. \eqref{eq:volfub} and \eqref{Sliv}. In order to get the required asymptotics, we need to replace both the integrand and the Jacobian of the change of variables by precise estimates in the vicinity of the point $x_0$, see Lemma \ref{lem:volunif} (i), Lemma \ref{lem:jacobchtsph} (i) and Theorem \ref{Jacob}. Being able to integrate these estimates means that they are uniform with respect to the variable(s) of integration and that the contribution of points far from $x_0$ is negligible. It turns out that it will be much more convenient to show the uniformity in the context of a compact Riemannian manifold, see Lemma \ref{lem:volunif} (ii), Lemma \ref{lem:jacobchtsph} (ii) and Proposition \ref{prop:jacobunif}. Similarly, the negligibility  of the contribution of points far from $x_0$ is easily proved as soon as the volume of $M$ is finite, see \eqref{eq:conclusionsurtildeI} and \eqref{eq:ItildeNeg}.

Let us fix $x_0\in M$ and consider $r>0$ small enough to be chosen later such that the closure of the geodesic ball ${\mathcal B}^{(M)}(x_0,r)$ is a compact neighborhood of $x_0$. We define the modification $\vol^{(M)}({\widetilde{\mathcal C}}^{(M)}_{x_0,\lambda})$ (resp. $N({\widetilde{\mathcal C}}^{(M)}_{x_0,\lambda})$) of the variable $\vol^{(M)}(\Cl)$ (resp. $\Nl$) as the volume (resp. the number of vertices) of the Voronoi cell ${\widetilde{\mathcal C}}^{(M)}_{x_0,\lambda}$ associated with $x_0$ when $M$ is replaced by the new manifold ${\mathcal B}^{(M)}(x_0,r)$, namely
$${\widetilde{\mathcal C}}^{(M)}_{x_0,\lambda}={\mathcal C}^{({\mathcal B}^{(M)}(x_0,r))}(x_0,{\mathcal P}_{\lambda}\cap {\mathcal B}^{(M)}(x_0,r)\cup\{x_0\}).$$
Both variables $\vol^{(M)}({\widetilde{\mathcal C}}^{(M)}_{x_0,\lambda})$ and $\vol^{(M)}(\Cl)$ (resp. $N({\widetilde{\mathcal C}}^{(M)}_{x_0,\lambda})$ and $\Nl$) are naturally coupled. The main result of the section is the following proposition which says that the difference in expectation between the two variables is negligible in front of $\lambda^{-(1+\frac{2}{n})}$ (resp. $\lambda^{-\frac 2n}$), i.e. in front of the second term of the desired two-term expansion in Theorem \ref{thm:propvol} (resp. Theorem \ref{THM}). 
\begin{prop}
\label{prop:localdep}
There exists $c>0$ such that for $\lambda$ large enough, we get
\begin{equation}
  \label{eq:resultatlocaldepvol}
\E\left[|\vol^{(M)}(\Cl)-\vol^{(M)}({\widetilde{\mathcal C}}^{(M)}_{x_0,\lambda})|{\bf 1}_{\{\Cl\ne {\widetilde{\mathcal C}}^{(M)}_{x_0,\lambda}\}}\right]\le c e^{-\lambda/c}
\end{equation}
and
\begin{equation}
  \label{eq:resultatlocaldep}
\E\left[|\Nl-N({\widetilde{\mathcal C}}^{(M)}_{x_0,\lambda})|{\bf 1}_{\{\Cl\ne {\widetilde{\mathcal C}}^{(M)}_{x_0,\lambda}\}}\right]\le c e^{-\lambda/c}.  
\end{equation}
\end{prop}
\begin{proof}
Let $h$ be a function equal to either $\vol^{(M)}(\cdot)$ or $N(\cdot)$.
By the Cauchy-Schwarz inequality, we get
\begin{align}
  \label{eq:diffesperances}
&\E\left[|h(\Cl)-h({\widetilde{\mathcal C}}^{(M)}_{x_0,\lambda})|{\bf 1}_{\{\Cl\ne {\widetilde{\mathcal C}}^{(M)}_{x_0,\lambda}\}}\right]\nonumber\\
&\hspace*{4cm}\le \sqrt{\max(\E[h(\Cl)^2],\E[h({\widetilde{\mathcal C}}^{(M)}_{x_0,\lambda})^2)]}\sqrt{\proba[\Cl\ne {\widetilde{\mathcal C}}^{(M)}_{x_0,\lambda}]}.
\end{align}
Consequently, it is enough to show the two following facts:
on one hand, $\E[h(\Cl)^2]$ and its modification $\E[h({\widetilde{\mathcal C}}^{(M)}_{x_0,\lambda})^2]$ are bounded from above by a constant not depending on $\lambda$ and on the other hand, the probability $\proba[\Cl\ne {\widetilde{\mathcal C}}^{(M)}_{x_0,\lambda}]$ is exponentially decreasing, like $e^{-\lambda/c}$. We do so in the next two lemmas, whose proofs are postponed to the appendix.

The next lemma provides basic estimates for the second moments of both the volume and the number of vertices of the Voronoi cell.
\begin{lem}
  \label{lem:secondmoment}
There exists $c>0$ such that for every $\lambda>0$,
\begin{equation}
  \label{eq:momentsfinisvol}
\max(\E[\vol^{(M)}(\Cl)^2],\E[\vol^{(M)}({\widetilde{\mathcal C}}^{(M)}_{x_0,\lambda})^2])\le \frac{c}{\lambda^2}
\end{equation}
and 
\begin{equation}
  \label{eq:momentsfinisN}
\max(\E[N(\Cl)^2],\E[N({\widetilde{\mathcal C}}^{(M)}_{x_0,\lambda})^2])\le c.
\end{equation}
\end{lem}
The second lemma shows the localization of the random variable $\Nl$, namely that for $r$ chosen to be small enough, $\Cl$ and ${\widetilde{\mathcal C}}^{(M)}_{x_0,\lambda}$ differ with a probability decreasing exponentially fast to zero when $\lambda\to\infty$.
\begin{lem}
 \label{lem:stabilization}
For $r>0$ small enough, there exists $c>0$ such that $$\proba(\Cl\ne {\widetilde{\mathcal C}}^{(M)}_{x_0,\lambda})\le c e^{-\lambda/c}.$$
\end{lem}
Inserting the results of Lemmas \ref{lem:secondmoment} and \ref{lem:stabilization} into \eqref{eq:diffesperances}, we obtain the required result \eqref{eq:resultatlocaldep}.
\end{proof}
Thanks to Proposition \ref{prop:localdep}, we can henceforth assume in the rest of the paper that $M$ is a compact Riemannian manifold and when necessary that $M$ is equal to the closure of a geodesic ball ${\mathcal B}^{(M)}(x_0,r)$ for $r$ small enough.
\section{Mean volume of $\Cl$: proof of Theorem \ref{thm:propvol}}\label{sec:vol}

In this section, we prove Theorem \ref{thm:propvol} which contains an asymptotic expansion of $\E(\vol^{(M)}(\Cl))$ and explicit formulas in the particular cases of ${\mathcal S}_k^{n}$ and ${\mathcal H}_k^n$.

\begin{proof}[Proof of Theorem \ref{thm:propvol} (i)]
Let us fix $x_0\in M$. Thanks to Fubini's theorem and to the definition of the Poisson point process ${\mathcal P}_{\lambda}$, we get
\begin{equation} \label{eq:volfub}
\E[\vol^{(M)}(\Cl)] = \int_{M} \proba( x\in \Cl) \dvol^{(M)}(x)= \int_{M} e^{-\lambda\vol^{(M)}({\mathcal B}^{(M)}(x,d^{(M)}(x_0,x)))} \dvol^{(M)}(x).
\end{equation}

The computation of this integral requires to calculate the volume of ${\mathcal B}^{(M)}(x,d^{(M)}(x_0,x)))$ and to rewrite the volume element $\dvol^{(M)}(x)$. The key idea is to discriminate between points $x$ close to $x_0$ and points $x$ far from $x_0$. On the one hand, when the distance between $x$ and $x_0$ tends to $0$, Lemmas \ref{lem:jacobchtsph} and \ref{lem:volunif} provide asymptotics for the volume element $\dvol^{(M)}(s)$ and the volume 
$\vol^{(M)}({\mathcal B}^{(M)}(x,d^{(M)}(x_0,x)))$ respectively. On the other hand, it is expected that the contribution of points $x$, when the distance between $x$ and $x_0$ is `large', is negligible since the integrand decreases exponentially fast with the distance.\\

\noindent \textit{Step 1: decomposition of $\E[\vol^{(M)}(\Cl)]$ into two integrals}.
We rewrite the integral in \eqref{eq:volfub} as the sum of two integrals over ${\mathcal B}^{(M)}(x_0,r_\lambda)$ and $M\setminus {\mathcal B}^{(M)}(x_0,r_\lambda)$ where $r_\lambda$ is a positive radius depending on $\lambda$. We do so by choosing $r_\lambda=\lambda^{-\frac{n+1}{n(n+2)}}$ so that $r_\lambda$ is at the same time slightly larger than the diameter of $\Cl$ which is of order $\lambda^{-\frac{1}{n}}$ and small enough to guarantee that $\lambda r^{n+2}$ is negligible when $\lambda\to\infty$. The first requisite is natural whereas the second one will become necessary in Step 2. 
Using \eqref{eq:volfub}, we rewrite the expectation $\E[\vol^{(M)}(\Cl)]$ as
\begin{align}
\E[\vol^{(M)}(\Cl)] &= I_\lambda+\tilde{I}_{\lambda}
\label{eq:decoupInt}
\end{align}
where 
\begin{equation}
  \label{eq:defIinterm}
I_\lambda=\int_{\bo^{(M)}(x_0,\lambda^{-\frac{n+1}{n(n+2)}})} e^{-\lambda\vol^{(M)}({\mathcal B}^{(M)}(x,d^{(M)}(x_0,x)))} \dvol^{(M)}(x)  
\end{equation}
and
\begin{equation}
  \label{eq:defI'interm}
\tilde{I}_\lambda=\int_{M\backslash \bo^{(M)}(x_0,\lambda^{-\frac{n+1}{n(n+2)}})} e^{-\lambda\vol^{(M)}({\mathcal B}^{(M)}(x,d^{(M)}(x_0,x)))} \dvol^{(M)}(x).  
\end{equation} 
\noindent\textit{Step 2: $\tilde{I}_\lambda$ is negligible.} Let us show that $\tilde{I}_\lambda$ given at \eqref{eq:defI'interm} is negligible in front of $\frac{1}{\lambda^{1+\frac 2n}}$. To this end, let us observe that, since $M$ is compact, by Lemma \ref{lem:volunif} (i), (ii), there exists a constant $c$ such that for $\lambda$ large enough and every $x\in M\backslash \bo^{(M)}(x_0,\lambda^{-\frac{n+1}{n(n+2)}})$, 
\begin{equation}
  \label{eq:minorantvolbouletildeI}
 \vol^{(M)}( \bo^{(M)}(x,d^{(M)}(x_0,x)) ) \ge \vol^{(M)}( \bo^{(M)}(x,\lambda^{-\frac{n+1}{n(n+2)}}) ) \ge c \lambda^{-\frac{n+1}{n+2}} .
\end{equation}
Inserting \eqref{eq:minorantvolbouletildeI} into \eqref{eq:defI'interm}, we get
\begin{align*}
\tilde{I}_\lambda
&\le \vol^{(M)}(M) e^{-c\lambda^{1-\frac{n+1}{n+2}}}. 
\end{align*}
This implies that when $\lambda\to\infty$,
\begin{equation}
  \label{eq:conclusionsurtildeI}
\tilde{I}_\lambda=o\left(\frac{1}{\lambda^{1+\frac 2n}}\right).  
\end{equation}
\noindent\textit{Step 3 : estimate of $I_\lambda$}. We prove now that when $\lambda\to\infty$
\begin{equation}
  \label{eq:resultatvouluI}
I_\lambda=\frac{1}{\lambda}+o\left(\frac{1}{\lambda^{1+\frac 2n}}\right).  
\end{equation}
Applying the spherical change of variables provided by \eqref{eq:Jru}, we obtain
\begin{equation}
  \label{eq:Ireecrit}
I_\lambda=\int_0^{\lambda^{-\frac{n+1}{n(n+2)}}}\int e^{-\lambda\vol^{(M)}({\mathcal B}^{(M)}(x,r))}|{\mathcal J}_{x_0}^{(M)}(r,u)|\mathrm{d}r\dvol^{(\sph^{n-1})}(u).   
\end{equation}
We then replace both functionals $\vol^{(M)}({\mathcal B}^{(M)}(x,r))$ and ${\mathcal J}_{x_0}^{(M)}(r,u)$ by suitable estimates. Let $\varepsilon>0$. Thanks to Lemmas \ref{lem:jacobchtsph} (ii) and \ref{lem:volunif} (i),  we get for $\lambda$ large enough, $0<r<\lambda^{-\frac{n+1}{n(n+2)}}$ and $u\in T_{x_0}M$ with $\|u\|_{x_0}=1$,
\begin{align}
r^{n-1} - \frac{ \Ric^{(M)}_{x_0}(u)}{6}r^{n+1} -\varepsilon r^{n+1}&\le \J_{x_0}^{(M)}(r,u) \le  r^{n-1} - \frac{ \Ric^{(M)}_{x_0}(u)}{6}r^{n+1} + \varepsilon r^{n+1} \label{eq:inegJacobExp2},
\end{align}
and
\begin{align*}
     e^{-\lambda (\kappa_n r^n-(\frac{\kappa_n \Scal_{x_0}^{(M)}}{6(n+2)}  -\varepsilon)r^{n+2})} &\le e^{-\lambda \vol^{(M)}(\bo^{(M)}(x,r))} 
\le e^{-\lambda (\kappa_n r^n-(\frac{\kappa_n \Scal_{x_0}^{(M)}}{6(n+2)}  +\varepsilon)r^{n+2})}
\nonumber
\end{align*}
which implies
\begin{align}
     e^{-\lambda \kappa_n r^n} \left[ 1+ \left(\frac{\kappa_n \Scal_{x_0}^{(M)}}{6(n+2)}  -\varepsilon\right)\lambda r^{n+2} \right]&\le e^{-\lambda \vol^{(M)}(\bo^{(M)}(x,r))} 
\le e^{-\lambda \kappa_n r^n} e^{\left(\frac{\kappa_n \Scal_{x_0}^{(M)}}{6(n+2)}  +\varepsilon\right)\lambda r^{n+2}}.
\label{eq:inegVolumeB} 
\end{align}
We notice that the right hand side of \eqref{eq:inegVolumeB} can be simplified. Indeed, since $\lambda r^{n+2}\le \lambda^{-\frac{1}{n}}$ goes to zero, 
we deduce from \eqref{eq:inegVolumeB} that for $\lambda$ large enough and $r<\lambda^{-\frac{n+1}{n(n+2)}}$,
\begin{align}
     e^{-\lambda \kappa_n r^n} \left[ 1+ \left(\frac{\kappa_n \Scal_{x_0}^{(M)}}{6(n+2)}  -\varepsilon\right)\lambda r^{n+2} \right]&\le e^{-\lambda \vol^{(M)}(\bo^{(M)}(x,r))} 
\le e^{-\lambda \kappa_n r^n} \left[1+\left(\frac{\kappa_n \Scal_{x_0}^{(M)}}{6(n+2)}  +2\varepsilon\right)\lambda r^{n+2}\right].
\label{eq:inegVolumeC} 
\end{align}
Inserting \eqref{eq:inegJacobExp2} and \eqref{eq:inegVolumeC} into \eqref{eq:defIinterm}, we get
\begin{equation} I_{\lambda,-} \le I_{\lambda} \le I_{\lambda,+} ,\label{eq:infIsup}\end{equation}

where 
\begin{align}
I_{\lambda,-} &=\int_{0}^{\lambda^{-\frac{n+1}{n(n+2)}}} \int
e^{-\lambda \kappa_n r^n} \left[ 1+ \left(\frac{\kappa_n \Scal_{x_0}^{(M)}}{6(n+2)}  -\varepsilon\right)\lambda r^{n+2} - \left(\frac{ \Ric^{(M)}_{x_0}(u)}{6} +\varepsilon\right) r^{2}\right.\nonumber \\
&\hspace*{4cm}\left.-\left(\frac{\kappa_n \Scal_{x_0}^{(M)}}{6(n+2)}  -\varepsilon\right)\left(\frac{ \Ric^{(M)}_{x_0}(u)}{6} +\varepsilon\right)\lambda r^{n+4}\right]  \dvol^{(\sph^{n-1})}(u)r^{n-1}\dd r \notag
\end{align}
and 
\begin{align}
I_{\lambda,+} &=\int_{0}^{\lambda^{-\frac{n+1}{n(n+2)}}} \int  
e^{-\lambda \kappa_n r^n} \left[ 1+ \left(\frac{\kappa_n \Scal_{x_0}^{(M)}}{6(n+2)}  +2\varepsilon\right)\lambda r^{n+2} - \left(\frac{ \Ric^{(M)}_{x_0}(u)}{6} -\varepsilon\right) r^{2}\right.\nonumber \\
&\hspace*{4cm}\left.-\left(\frac{\kappa_n \Scal_{x_0}^{(M)}}{6(n+2)}  +2\varepsilon\right)\left(\frac{ \Ric^{(M)}_{x_0}(u)}{6} -\varepsilon\right)\lambda r^{n+4}\right]  \dvol^{(\sph^{n-1})}(u)r^{n-1}\dd r. \notag
\end{align}
Let us determine an upper bound for $I_{\lambda,+}$. 
Using the change of variables $y=\lambda\kappa_n r^n$, we obtain that
\begin{align}
\label{eq:I_+interm}
I_{\lambda,+} &\le \int_{0}^{\infty} e^{-y}\int 
\left[ 1+ \left(\frac{\kappa_n \Scal_{x_0}^{(M)}}{6(n+2)}  +2\varepsilon\right)\left(\frac{y}{\kappa_n}\right)^{1+\frac{2}{n}}\frac{1}{\lambda^{\frac{2}{n}}} - \left(\frac{ \Ric^{(M)}_{x_0}(u)}{6} -\varepsilon\right) \left(\frac{y}{\kappa_n}\right)^{\frac{2}{n}}\frac{1}{\lambda^{\frac{2}{n}}}\right.\nonumber \\
&\hspace*{2.3cm}\left.-\left(\frac{\kappa_n \Scal_{x_0}^{(M)}}{6(n+2)}  +2\varepsilon\right)\left(\frac{ \Ric^{(M)}_{x_0}(u)}{6} -\varepsilon\right)\left(\frac{y}{\kappa_n}\right)^{1+\frac{4}{n}}\frac{1}{\lambda^{\frac{4}{n}}}\right]
\dvol^{(\sph^{n-1})}(u)\frac{\dd y}{n\kappa_n \lambda}.
\end{align}
Before going further, we notice that, thanks to ($\mbox{A}_1$), the integration over $y$ and $u$ of the term 
$$
e^{-y}\left(\frac{\kappa_n \Scal_{x_0}^{(M)}}{6(n+2)}  +2\varepsilon\right)
\left(\frac{ \Ric^{(M)}_{x_0}(u)}{6} -\varepsilon\right)
\left(\frac{y}{\kappa_n}\right)^{1+\frac{4}{n}}
\frac{1}{\lambda^{\frac{4}{n}}}\frac{1}{n\kappa_n \lambda}
$$
is bounded by $\frac{1}{\lambda^{1+\frac{4}{n}}}$ up to a multiplicative constant.
Consequently, an integration over $u \in \sph^{n-1}$ in \eqref{eq:I_+interm} combined with \eqref{eq:scalRicci} provides
\begin{align}
I_{\lambda,+} &\le \int_{0}^{\infty} e^{-y} \left\{
\frac{1}{\lambda}  +
\frac{1}{\kappa_n^{\frac{2}{n}}}\left[\left(\frac{\Scal_{x_0}^{(M)}}{6(n+2)} +2\varepsilon\right)y^{1+\frac{2}{n}}-\left(\frac{\Scal_{x_0}^{(M)}}{6n} -\varepsilon\right)y^{\frac{2}{n}}\right]
\frac{1}{\lambda^{1+\frac{2}{n}}}\right\}\mathrm{d}y +\frac{c}{\lambda^{1+\frac{4}{n}}}.
\label{eq:supI3}
\end{align}
where $c$ is a positive constant. Integrating now over $y$ in the right hand side of \eqref{eq:supI3}, we obtain for $\lambda$ large enough
\begin{align}
I_{\lambda,+} &\le  \frac{1}{\lambda} + \frac{\Scal_{x_0}^{(M)}}{6\kappa_n^{\frac{2}{n}}}\left(\frac{\Gamma\left(2+\frac 2n \right)}{n+2}- \frac{\Gamma\left( 1+\frac 2n\right) }{n} +c\varepsilon\right)\frac{1}{\lambda^{1+ \frac 2n}} \notag \\
& \le \frac{1}{\lambda} + c\varepsilon\frac{1}{\lambda^{1+ \frac 2n}}.\label{eq:supI4} 
\end{align}
Similarly, we can prove that 
\begin{align}
I_{\lambda,-}  & \ge \frac{1}{\lambda} - c\varepsilon\frac{1}{\lambda^{1+ \frac 2n}}.\label{eq:supI5} 
\end{align}
Inserting \eqref{eq:supI4} and \eqref{eq:supI5} into \eqref{eq:infIsup}, we obtain \eqref{eq:resultatvouluI} which, combined with \eqref{eq:conclusionsurtildeI} and \eqref{eq:decoupInt}, completes the proof of point (i) of Theorem \ref{thm:propvol}.
\end{proof}
\begin{proof}[Proof of Theorem \ref{thm:propvol} (ii)]
\noindent{\textbf{Case $M={\mathcal S}_k^n$}}.
We go back to \eqref{eq:volfub} when $M={\mathcal S}_k^n$ and we use the spherical change of variables $x=\exp_{x_0}(ru)$, with $r\in [0,\frac{\pi}{k}]$ and $ u\in \sph^{n-1}$. The Jacobian determinant of this change of variables satisfies
\begin{equation}
  \label{eq:sphsph}
 \dvol^{(\sph_k^n)}(x) = \left(\frac{\sin(kr)}k \right)^{n-1} \dd r\dvol^{(\sph_1^{n-1})}(u) .
\end{equation}
The ball $\bo^{(\sph^n_k)}(x,d^{(\sph^n_k)}(x_0,x))$ has radius $r$ and volume
\begin{equation}
  \label{eq:boulesph}
 \vol^{(\sph^n_k)}(\bo^{(\sph^n_k)}(x,r)) = \sigma_{n-1}\int_{t=0}^{r} \left(\frac{\sin(kt)}k \right)^{n-1} \dd t  
\end{equation}
where we recall that $\sigma_{n-1}=\vol^{(\sph^{n-1})}(\sph^{n-1})$.
Combining \eqref{eq:volfub} applied to $M={\mathcal S}_k^n$, \eqref{eq:sphsph} and \eqref{eq:boulesph},  we obtain
\begin{align*}
 \E[\vol^{(\sph^n_k)}({\mathcal C}_{x_0,\lambda}^{(\sph^n_k)})] &= \int_{r=0}^{\frac{\pi}{k}} \int_{u\in\sph_1^{n-1}}e^{-\lambda\sigma_{n-1}\int_{t=0}^{r} \left(\frac{\sin(kt)}k \right)^{n-1} dt }  \left(\frac{\sin(kr)}k \right)^{n-1} \dd r\dvol^{(\sph_1^{n-1})}(u) \\
&= \sigma_{n-1}\int_{r=0}^{\frac{\pi}{k}}e^{-\lambda\sigma_{n-1}\int_{t=0}^{r} \left(\frac{\sin(kr)}k \right)^{n-1} dt }  \left(\frac{\sin(kr)}k \right)^{n-1} \dd r. 
\end{align*}
Using the change of variables $\displaystyle y= \sigma_{n-1}\lambda\int_{t=0}^{r} \left(\frac{\sin(kr)}k \right)^{n-1} \dd t $ in the integral above, we get 
      \begin{align*}
 \E[\vol^{(\sph^n_k)}({\mathcal C}_{x_0,\lambda}^{(\sph^n_k)})] &= \frac{1}{\lambda} \int_{y=0}^{\sigma_{n-1}\frac{\lambda}{k^n}\int_{t=0}^{\pi} (\sin(t))^{n-1} \dd t}e^{-y} \dd y \\
& = \frac{1}{\lambda} \left( 1 - e^{-2\sigma_{n-1}W_{n-1}\frac{\lambda}{k^n}} \right)
\end{align*}
where $\displaystyle W_{n-1}=\int_{t=0}^{\frac{\pi}{2}} \sin^{n-1}(t) \dd t$ is the $(n-1)$-th Wallis integral which is classically know to be equal to $\frac{1}{2}B\left(\frac{n}{2},\frac{1}{2}\right)$, see e.g. \cite[Formula (5.6)]{Art64}.\\

\noindent{\textbf{Case $M={\mathcal H}^n_k$}}. 
The identity \eqref{eq:Jru} can be rewritten as 
\begin{equation}\label{eq:sphhyp}
 \dvol^{({\mathcal H}_k^n)}(x) = \left(\frac{\sinh(kr)}k \right)^{n-1} \dd r\dvol^{(\sph_1^{n-1})}(u) .
\end{equation}
Moreover, for any $x\in \mathcal{H}^n_k$ and $r>0$,  
\begin{equation}
  \label{eq:boulehyp}
 \vol(\bo^{(\mathcal{H}^n_k)}(x,r)) = \sigma_{n-1}\int_{t=0}^{r} \left(\frac{\sinh(kt)}k \right)^{n-1} \dd t .
\end{equation}
Combining \eqref{eq:volfub} applied to $M={\mathcal H}_k^n$, \eqref{eq:sphhyp} and \eqref{eq:boulehyp}, we obtain
\begin{align*}
 \E[\vol^{(\mathcal{H}^n_k)}({\mathcal C}_{x_0,\lambda}^{(\mathcal{H}^n_k)})] &= \int_{r=0}^{\frac{\pi}{k}} \int_{u\in\sph^{n-1}}e^{-\lambda\sigma_{n-1}\int_{t=0}^{r} \left(\frac{\sinh(kt)}k \right)^{n-1} \dd t }  \left(\frac{\sinh(kr)}k \right)^{n-1} \dd r\dvol^{(\sph_1^{n-1})}(u) \\
&= \sigma_{n-1}\int_{r=0}^{\frac{\pi}{k}}e^{-\lambda\sigma_{n-1}\int_{t=0}^{r} \left(\frac{\sinh(kr)}k \right)^{n-1} \dd t }  \left(\frac{\sinh(kr)}k \right)^{n-1} \dd r 
\end{align*}
Finally, the change of variables $y= \lambda\sigma_{n-1}\int_{t=0}^{r} \left(\frac{\sinh(kr)}k \right)^{n-1} \dd t $ implies that 
      \begin{align*}
 \E[\vol^{(\mathcal{H}^n_k)}({\mathcal C}_{x_0,\lambda}^{(\mathcal{H}^n_k)})] &= \frac{1}{\lambda} \int_{y=0}^{\infty}e^{-y} \dd y=\frac{1}{\lambda}. 
\end{align*}
\end{proof}

\section{A local change of variables formula of Blaschke-Petkantschin type}\label{sec:BP} 
The key tool for proving Theorem \ref{THM} is an extension of a spherical Blaschke-Petkantschin formula, known in ${\mathbb R}^n$ and in ${\mathcal S}_k^{n}$, to the general setting of a Riemannian manifold.

A Blaschke-Petkantschin formula is a rewriting of the $m$-fold product of the volume measure for $m\le (n+1)$ which is based on a suitable geometric decomposition of a $m$-tuple of points. It provides the calculation of the Jacobian of the associated change of variables in an integral. Classically, the geometric decomposition consists in fixing the linear or affine subspace which contains all the points, integrating over all $m$-tuples in that subspace and then integrating over the Grassmannian of all subspaces, see \cite[Chapter 7]{Wei08}. There exists another kind of formulas of Blaschke-Petkantschin type with a spherical decomposition: the sphere containing all the points is fixed, we integrate over the positions of the points on the sphere and then integrate over the Grassmannian of all subspaces spanned by the sphere and over the radius and center of the sphere. Such formulas have been provided by Miles in both cases of the Euclidean space ${\mathbb R}^n$ \cite[Formula (70)]{Mil70}, see also \cite[Proposition 2.2.3]{Mol94}, and of the sphere ${\mathcal S}_k^n$ \cite[Theorem 4]{Mil71bis}. 

In this paper, we concentrate on the particular case of the rewriting of the $n$-fold product of the volume measure with a slightly different spherical geometric transformation: we fix the circumscribed sphere containing both a fixed origin and the $n$ points. Let us consider the example of $\R^n$ first: for almost all given points $x_1,\dots x_n$, there exists a unique circumscribed ball associated with $0,x_1,\dots,x_n$ that we will denote by $\bo^{({\mathbb R}^n)}(0,x_1,\dots,x_n)$. This makes it possible to define the change of variables $(x_1,\dots,x_n) \mapsto (r,u,u_1,\dots,u_n)$ given by the relations \[ x_i = ru+ ru_i \]
where $(r,u)\in \R_+ \times \sph^{n-1}$ are the polar coordinates of the circumcenter of $0,x_1,\dots,x_n$, and $u_1,\dots,u_n \in (\sph^{n-1} )^{n}$ indicate the positions of the points $x_1,\dots,x_n$ on the boundary of $\bo^{({\mathbb R}^n)}(0,x_1,\dots,x_n)$. Now this approach can be extended to any compact Riemannian manifold M. We start by fixing an origin $x_0\in M$ and we recall that thanks to the Hopf-Rinow theorem \cite[Theorem 52]{Ber03}, the compacity of $M$ guarantees that $M$ is geodesically complete, which implies that the exponential map $\exp_{x_0}$ is defined on the whole tangent space $T_{x_0}M$. Let us define the following set ${\mathcal E}_n^{(M)}$ 
\begin{align*}
 {\mathcal E}_n^{(M)}
&= \{ (r, u, u_1,\dots,u_n): r>0, u\in T_{x_o}M \mbox{ with $\|u\|_{x_0}=1$},\\ &\hspace*{4.3cm} u_i\in T_{\exp_{x_0}(ru)}M \mbox{ and $\|u_i\|_{\exp_{x_0}(ru)}=1$ $\forall\;i=1,\dots,n$} \}.
\end{align*}
We then consider the following transformation $\Phi_{x_0}^{(M)}$, see Figure \ref{fig:BP}:
$$\Phi_{x_0}^{(M)}:\left\{\begin{array}{lll}{\mathcal E}_n^{(M)}&\longrightarrow&M^n\\(r,u,u_1,\cdots,u_n)&\longmapsto&(x_1,\cdots,x_n) \end{array}\right.$$
where
\begin{equation}\label{change}
x_i = \exp_{|\exp_{x_0}(ru)}(ru_i). 
\end{equation}
\begin{figure}[H]
\centering
\begin{tikzpicture}[line cap=round,line join=round,>=triangle 45,x=2.0cm,y=2.0cm]
\clip(-2.,-2.) rectangle (2.,1.7);
\draw [->] (1.317104704423232,1.0801321756759872) -- (0.8350020080213368,0.947553934165466);
\draw [->] (1.317104704423232,1.0801321756759872) -- (1.4496829459337532,0.598029479274092);
\draw [->] (0.,0.) -- (-0.47569000304322695,-0.1540098081445942);
\draw [dash pattern=on 1pt off 1pt,rounded corners=10pt] (-0.47569000304322695,-0.1540098081445942)-- (-0.8278962297811208,-0.3156579059115343)-- (-1.0228963147087893,-0.4696053413807462)-- (-1.1889468727758623,-0.8057348983381241)-- (-1.1871069125426152,-1.0546055961637515)-- (-1.1027283453431442,-1.2982451678019533);
\draw [->] (-1.1027283453431442,-1.2982451678019533) -- (-0.8854418303841418,-1.748563466804971);
\draw [->] (-1.1027283453431442,-1.2982451678019533) -- (-0.6511816995113886,-1.0835229597328042);
\draw (0.06,0.04355277684996016) node[anchor=north west] {$z= \exp_{x_0}(Ru)$};
\draw (1.3,1.4) node[anchor=north west] {$x_0$};
\draw (0.9,1.25) node[anchor=north west] {$u$};
\draw (1.5,1.069869013311373) node[anchor=north west] {$v_l^{(0)}$};
\draw (-0.3557907610088712,0.25) node[anchor=north west] {$u_i$};
\draw (-0.04,1.) node[anchor=north west] {$R$};
\draw (-1.2897385361887568,-0.2) node[anchor=north west] {$R$};
\draw (-1.6,-1.331710980008333) node[anchor=north west] {$v_0^{(i)}(R)$};
\draw (-0.9202646910626483,-1.15) node[anchor=north west] {$v_l^{(i)}(R)$};
\draw (-1.4128964845641263,-1.1) node[anchor=north west] {$x_i$};
\draw [dash pattern=on 1pt off 1pt,rounded corners=10pt] (0.8350020080213368,0.947553934165466)-- (0.47552539052487347,0.8235531165606339)-- (0.198420006680292,0.5669740574452807)-- (0.,0.);
\begin{scriptsize}
\draw [color=black] (0.,0.)-- ++(-2.5pt,-2.5pt) -- ++(5.0pt,5.0pt) ++(-5.0pt,0) -- ++(5.0pt,-5.0pt);
\draw [color=black] (1.317104704423232,1.0801321756759872)-- ++(-2.5pt,-2.5pt) -- ++(5.0pt,5.0pt) ++(-5.0pt,0) -- ++(5.0pt,-5.0pt);
\draw [color=black] (-1.1027283453431442,-1.2982451678019533)-- ++(-2.5pt,-2.5pt) -- ++(5.0pt,5.0pt) ++(-5.0pt,0) -- ++(5.0pt,-5.0pt);
\end{scriptsize}
\end{tikzpicture}
\caption{The geometric decomposition provided by $\Phi_{x_0}^{(M)}$}
\label{fig:BP}
\end{figure}
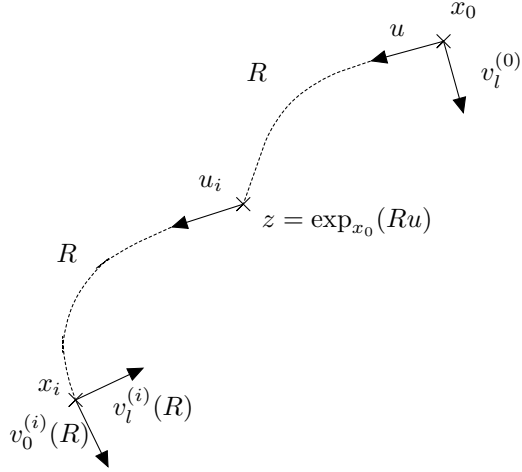
Let us denote by $\widetilde{\mathcal J}_{x_0}^{(M)}$ the Jacobian function associated with the transformation $\Phi_{x_0}^{(M)}$ so that the following equality of measures is satisfied:
\begin{equation}
  \label{eq:egalitemesuressec5}
\dvol^{(M)}(x_1) \dots \dvol^{(M)}(x_n)=|\widetilde{\mathcal J}_{x_0}^{(M)}(r,u,u_1,\cdots,u_n)|{\mathrm d}r \dvol^{({\mathcal S}_1^{n-1})}(u)\dvol^{({\mathcal S}_1^{n-1})}(u_1) \dots \dvol^{({\mathcal S}_1^{n-1})}(u_n).  
\end{equation}
We recall that as soon as $\widetilde{\mathcal J}_{x_0}^{(M)}$ is different from zero, the transformation $\Phi_{x_0}^{(M)}$ is a local diffeomorphism thanks to the inverse function theorem.

In Proposition \ref{const} we provide explicit formulas for $\widetilde{\mathcal J}_{x_0}^{(M)}$ in the three cases when $M$ has constant sectional curvature. 
\begin{prop}\label{const}
For every $r>0$, $(n+1)$-tuple $(u,u_1,\dots,u_n)$ of unit-vectors, $x_0$ in either $\R^n$ or ${\mathcal S}_k^n$ or ${\mathcal H}_k^n$ and every $k>0$, we get
\begin{align}
\widetilde{\mathcal J}_{x_0}^{(\R^n)}(r,u,u_1,\dots,u_n) &= n! \Delta(u,u_1,\dots,u_n) r^{n^2-1},\label{eq:formuleBPeuc}\\ 
\widetilde{\mathcal J}_{x_0}^{({\mathcal S}_k^n)}(r,u,u_1,\dots,u_n) &= n! \Delta(u,u_1,\dots,u_n) \left(\frac{\sin(kr)}{k}\right)^{n^2-1},\label{eq:formuleBPsph}\\ 
\widetilde{\mathcal J}_{x_0}^{({\mathcal H}_k^n)}(r,u,u_1,\dots,u_n) &= n! \Delta(u,u_1,\dots,u_n) \left(\frac{\sinh(kr)}{k}\right)^{n^2-1}.\label{eq:formuleBPhyp}
\end{align}
\end{prop}
The equalities \eqref{eq:formuleBPeuc} and \eqref{eq:formuleBPsph} come from an almost-direct adaptation of the proofs of the classical spherical Blaschke-Petkantschin formulas due to Miles in  \cite[Formula (70)]{Mil70} and \cite[Theorem 4]{Mil71bis} respectively. The identity \eqref{eq:formuleBPhyp}  was included in the work of Isokawa for $n=2$ \cite{Iso2000} and $n=3$ \cite{Iso00b} though it was not precisely stated. To the best of our knowledge, the formula in dimension $n$ is new. Its proof is postponed to the end of this section.

In the general case, we are unable to derive an exact formula for $\widetilde{\mathcal J}_{x_0}^{(M)}$. Nevertheless, we provide in Theorem \ref{Jacob} its two-term asymptotic expansion when $r$ tends to 0, which will be enough for our purpose in this paper.
\begin{thm}\label{Jacob}
When $r\to 0$, we get
\begin{equation}\label{JJ} \widetilde{\mathcal J}_{x_0}^{(M)}(r,u,u_1,\dots,u_n) = n! \Delta(u,u_1,\dots,u_n)\left( r^{n^2-1} - \frac{L_{x_0}^{(M)}(u,u_1,\dots,u_n)}6 r^{n^2+1} + o(r^{n^2+1}) \right) \end{equation}
where $L_{x_0}^{(M)}(u,u_1,\dots,u_n) = \sum_{i=1}^n \Ric_{x_0}^{(M)}(u_i) + \Ric_{x_0}^{(M)}(u)$,  $\Ric_{x_0}^{(M)}(u_i)$ stands for $\Ric_{x_0}^{(M)}(P_{\exp_{x_0}(ru)\to x_0}(u_i))$ and $\Delta(u,u_1,\dots,u_n)$ is the Euclidean volume of the $n$-dimensional simplex of $T_{\exp_{x_0}(ru)}M$ spanned by $P_{x_0 \to \exp_{x_0}(ru)}(-u)$ and $u_1,\dots,u_n$.\end{thm}
Note that the volume $ \Delta(u,u_1,\dots,u_n)$ does not depend on the geometry of the manifold, because the vectors involved  lie in the tangent space at the circumcenter $\exp_{x_0}(ru)$, which is naturally identified with $\R^n$.

As expected, since the manifold can be approximated at the first order by the tangent space at $x_0$, the first term of the expansion \eqref{JJ} corresponds to the Euclidean case. Moreover, when $M$ has constant sectional curvature, the function of $r$ in \eqref{JJ}  is consistent with the expansions of $\widetilde{\mathcal J}_{x_0}^{({\mathcal S}_k^n)}$ and $\widetilde{\mathcal J}_{x_0}^{({\mathcal H}_k^n)}$ for small $r$.
	
In particular, for $r$ small enough, the Jacobian $\widetilde{\mathcal J}_{x_0}^{(M)}$ is non zero almost everywhere. Thus, the inverse function theorem implies that $\Phi_{x_0}^{(M)}$ is a local $C^1$-diffeomorphism. This remark combined with Assumption \mbox{($\mbox{A}_4$)} guarantees that $\Phi_{x_0}^{(M)}$, as an injective local $C^1$-diffeomorphism, defines a change of variables of Blaschke-Petkantschin type for $r$ small enough. 

Theorem \ref{Jacob} is a close companion to the main result from \cite{AurelieNote} which states an explicit calculation and asymptotic estimate for a change of variables in $M$ of Blaschke-Petkantschin flavor. In \cite{AurelieNote}, the underlying application consists in associating to a $(n+1)$-tuple of points in $M$ 
 its circumcenter, circumradius, and $(n+1)$ points on the circumsphere whereas in Theorem \ref{Jacob}, the point $x_0$ is fixed once and for all and we associate to a $n$-tuple $(x_1,\dots,x_n)$ the circumcenter, circumradius and $n$ points on the circumsphere of the $(n+1)$-tuple $(x_0,\cdots,x_n)$. This induces a new difficulty in the calculation of partial derivatives since the parameter $r$ appears in both the reference point and the entry of the exponential map. This will require an extra chain rule and the use of Lemma \ref{lem:exodocarmo} in all its power. Moreover, another refinement here is that we require the uniformity of the expansion of the Jacobian with respect to $x_0$ and the vectors $u$, $u_1,\dots,u_n$. This is done at the end of the section in Proposition \ref{prop:jacobunif}. 
\begin{proof}[Proof of Theorem \ref{Jacob}]
Let us fix $(r,u,u_1,\dots,u_n)$ and let $z=\exp_{x_0}(ru)$ be the circumcenter of $x_0,\dots,x_n$. We endow the tangent space $T_{x_0}M$ with an orthonormal basis $\mathcal{V}^{(0)} =\{ v_1^{(0)}, \dots, v_n^{(0)} \}$ where $v_1^{(0)}=u$. We consider in $T_zM$, the orthonormal basis $\mathcal{V}^{(0)}(r) =\{ v_1^{(0)}(r), \dots, v_n^{(0)}(r) \}$ obtained by parallel transport of $\mathcal{V}^{(0)}$ along $\gamma_0(t)= \exp_{x_0}(tu)$ and we write, for each $i$,  
\[ u_i= \sum_{j=1}^{n} u_i^j v_j^{(0)}(r) .\]
Now, for each $i$, consider an orthonormal basis of $T_z M$, $\mathcal{V}^{(i)} = \{ v_1^{(i)}, \dots v_n^{(i)} \}$, with $v_1^{(i)}=u_i$. We endow $T_{x_i} M$, with the orthonormal basis $\mathcal{V}^{(i)}(r) = \{ v_1^{(i)}(r), \dots v_n^{(i)}(r) \}$ obtained by parallel transport of $\mathcal{V}^{(i)}$ along $\gamma_i(t)= \exp_z(ru_i)$. 

\noindent \textit{Step 1: derivatives with respect to $r$.} Let us consider the vector of size $n$, $\left( \frac{\partial x_i}{\partial r} \right)$ whose $l$th component is the projection onto $v^{(i)}_l(r)$ of the derivative of $x_i$ with respect to $r$. This is obtained by applying Lemma \ref{lem:exodocarmo} to $\gamma= \gamma_i$, 
$\mathfrak{c}(s)= \exp_{x_0}((r+s)u)$
and $V(s)=(r+s)u_i(s)$, 
where $u_i(s)$ denotes the parallel transport of $u_i$ along $\mathfrak{c}$. Then,  
\begin{equation}\label{derivR}
\left( \frac{\partial x_i}{\partial r} \right) = J^{(i)}_{r}(1)
\end{equation}
where $J^{(i)}_{r}$ is the unique Jacobi field along $\tilde{\gamma}_i(t) = \exp_{z}(tru_i)$, with $J^{(i)}_{r}(0)= \mathfrak{c}'(0) = u(r)$, parallel transport of $u$ along $\gamma_0$ and $J^{(i)\prime}_{r}(0)= V'(0) = u_i$. 
We deduce from \eqref{derivR} and \cite[Chapter 5, Proposition 3.6]{Do92} that the first component of the vector $\left( \frac{\partial x_i}{\partial r} \right)$ is 
\begin{align*}
\left( \frac{\partial x_i}{\partial r} \right)_l &=\langle v_1^{(i)}(r), J^{(i)}_{r}(1) \rangle_{x_i}\\
																			& = \langle \frac{\tilde{\gamma}_i(1)}{r},  J^{(i)}_{r}(1) \rangle_{x_i} \\
																			&= \langle u_i, J^{(i)}_{r}(0) \rangle_{z} + \langle u_i, J^{(i)\prime}_{r}(0) \rangle_{z}\\
																			&= \langle u_i, u(r) \rangle_{z} + \langle u_i, u_i \rangle_{z} \\
																			&= u_i^1 +1.
\end{align*}
Consequently,
\begin{equation}\label{eq:matDerivR}
\left( \frac{\partial x_i}{\partial r} \right)=  \left(\begin{array}{lll}& 1+ u_i^1& \\\hline\\ & \mbox{ n/a } & \\ & & \end{array}\right)
\end{equation} 
where `n/a' only means that the vector under its first component does not need to be explicit in the rest of the proof. 

\noindent \textit{Step 2: derivatives with respect to $u$.} Let us consider $\left( \frac{\partial x_i}{\partial u} \right)$, whose entry $\left( \frac{\partial x_i}{\partial u} \right)_{l,m}$ is the projection onto $v^{(i)}_l(r)$ of the derivative of $x_i$ with respect to $u$ in the direction $v_{m+1}^{(0)}$. As before, we compute these derivatives, using Lemma \ref{lem:exodocarmo}. To this end, let us consider 
$\mathfrak{c}(s)= \exp_{x_0}(r(u+sv^{(0)}_{m+1}))$ and $V(s)= u_i(s)$,
where $u_i(s)$ is the parallel transport of $u_i$ along $\mathfrak{c}$. In particular, note that $V'(s)=0$. The derivative of $\mathfrak{c}$ requires to apply again Lemma \ref{lem:exodocarmo} to $\gamma=\gamma_0$, $\tilde{\mathfrak{c}}(s) = x_0$ and $\tilde{V}(s) = u+sv^{(0)}_{m+1}$. Thus, we have 
\begin{equation}\label{eq:cprime}
\mathfrak{c}'(0) = J^{(0)}_{m+1}(r)
\end{equation}
where $J^{(0)}_{m+1}$ is the unique Jacobi field along $\gamma_0$ with $J^{(0)}_{m+1}(0)=0$ and $J^{(0)\prime}_{m+1}(0)=v^{(0)}_{m+1}$. Now, Lemma \ref{lem:exodocarmo}, applied to $\mathfrak{c}$ and $V$ given above, implies
\begin{equation}\label{eq:derivU}
\left( \frac{\partial x_i}{\partial u} \right)_{l,m} = \langle v^{(i)}_l(r), J^{(i)}_{r,(m+1)}(r) \rangle_{x_i} 
\end{equation}
where $J^{(i)}_{r,(m+1)}$ is the unique Jacobi field along $\gamma_i$ with $J^{(i)}_{r,(m+1)}(0)= J^{(0)}_{m+1}(r)$, given by \eqref{eq:cprime} and $J^{(i)\prime}_{r,(m+1)}(0) =0$. Thanks to \eqref{eq:derivU} and \cite[Chapter 5, Proposition 3.6]{Do92}, we are able to compute the first line of $\left( \frac{\partial x_i}{\partial u} \right)$, 
\begin{align*}
\left( \frac{\partial x_i}{\partial u} \right)_{1,m}& = \langle v^{(i)}_1(r), J^{(i)}_{r,(m+1)}(r) \rangle_{x_i} \\
& = \langle u_i , J^{(0)}_{m+1}(r) \rangle_z + r\langle u_i , J^{(0)\prime}_{m+1}(r) \rangle_z \\
&= \langle u_i , J^{(0)}_{m+1}(r) \rangle_z 
\end{align*}
that is 

\begin{equation}\label{eq:matDerivU}
\left( \frac{\partial x_i}{\partial u} \right)=  \left(\begin{array}{ccc}
																									\langle u_i , J^{(0)}_{2}(r) \rangle_z	& \cdots &\langle u_i , J^{(0)}_{n}(r) \rangle_z \\ \hline \\
																									 ~~&~~&~~\\
																									~~&\mbox{n/a}&~~\\
																									~~&~~&~~
																										\end{array}\right).																			
\end{equation}

\noindent \textit{Step 3: derivatives with respect to $u_i$.} We consider the submatrix $\left( \frac{\partial x_i}{\partial u_i} \right)$ of size $n\times (n-1)$ whose entry $\left( \frac{\partial x_i}{\partial u_i} \right)_{l,m}$ is the projection onto $v^{(i)}_l(r)$ of the derivative of $x_i$ with respect to $u_i$ in the direction $v^{(i)}_{m+1}$. We apply again Lemma \ref{lem:exodocarmo} to $\gamma_i$,
$\mathfrak{c}(s) = x_0$  and 
$V(s) = u_i +s v^{(i)}_{m+1}$,
then
\begin{equation}\label{eq:derivUi}
\left( \frac{\partial x_i}{\partial u_i} \right)_{l,m} = \langle v^{(i)}_l(r), \tilde{J}^{(i)}_{m+1}(r) \rangle_{x_i},
\end{equation}
where $\tilde{J}^{(i)}_{m+1}$ is the unique Jacobi field along $\gamma_i$ with $\tilde{J}^{(i)}_{m+1}(0)=0$ and $\tilde{J}^{(0)\prime}_{m+1}(0)=v^{(i)}_{m+1}$.
An application of \cite[Chapter 5, Proposition 3.6]{Do92} shows that $\langle v_1^{(i)}(r), \tilde{J}_m^{(i)}(r)\rangle_{x_i} =0$ which means that $\tilde{J}_m^{(i)}$ is a normal Jacobi field. This implies, with \eqref{eq:derivUi}, 
\begin{equation}\label{eq:matDerivUi}
  \left( \frac{\partial x_i}{\partial u_i} \right) = \left(\begin{array}{lll} & {\mathbf 0}_{n-1} & \\\hline\\ & B^{(i)} & \\ & & \end{array}\right).
\end{equation} 
where $B^{(i)}$ denotes the $(n-1)\times (n-1)$ matrix with entries 

 \begin{equation}\label{eq:coefB} B^{(i)}_{k,m} = \langle v^{(i)}_{k+1}(r), \tilde{J}^{(i)}_{m+1}(r) \rangle_{x_i}. \end{equation}

\noindent \textit{Step 4: rewriting of the Jacobian determinant.} The Jacobian determinant $\widetilde{\J}_{x_0}^{(M)}$ can be written as

$$\widetilde{\J}_{x_0}^{(M)}(r,u,u_1,\dots,u_n) =\det \left(\begin{array}{cccccc}
\cline{3-3} 
\frac{\partial x_1}{\partial r}&\frac{\partial x_1}{\partial u}&\multicolumn{1}{|c|}{\frac{\partial x_1}{\partial u_1}}&~&\multirow{2}{0cm}{\text{\bf 0}}&~\\
\cline{3-3}
 & & & & &\\
\cline{4-4}
\frac{\partial x_2}{\partial r}&\frac{\partial x_2}{\partial u}&~&\multicolumn{1}{|c|}{\frac{\partial x_2}{\partial u_2}}&~~&~\\
\cline{4-4}
 & & & & &\\
\vdots&\vdots&~&\multirow{2}{0cm}{\text{\bf 0}}~~&\ddots&~\\
 & & & & &\\
\cline{6-6}
\frac{\partial x_n}{\partial r}&\frac{\partial x_n}{\partial u}&~&~&~&\multicolumn{1}{|c|}{\frac{\partial x_n}{\partial u_n}}\\
\cline{6-6}
\end{array}\right)$$
Now, thanks to \eqref{eq:matDerivR}, \eqref{eq:matDerivU} and \eqref{eq:matDerivUi}, we obtain 
\begin{equation}\label{eq:jac1}
\widetilde{\J}_{x_0}^{(M)}(r,u,u_1,\dots,u_n)=\det \left(\begin{array}{cccc|ccc|ccc|ccc} 
1+u_1^1&\langle u_1 , J^{(0)}_{2}(r) \rangle_z &\cdots & \langle u_1 , J^{(0)}_{n}(r) \rangle_z &&{\mathbf 0}&& &{\mathbf 0}& & &{\mathbf 0}&\\\hline&&&&&&&&&&&&\\
&&\mbox{n/a}&&& B^{(1)}&& &{\mathbf 0}& &&{\mathbf 0}&\\&&&&&&&&&&&&\\\hline &&&&&&&&&&&&\\
\vdots&&\vdots&& &{\mathbf 0}&& 
&\ddots& &&{\mathbf 0}& \\ &&&&&&&&&&&&\\\hline
1+u_n^1&\langle u_1 , J^{(0)}_{2}(r) \rangle_z&\cdots&\langle u_1 , J^{(0)}_{n}(r) \rangle_z&&{\mathbf 0}&& & {\mathbf 0}& &&{\mathbf 0}&\\\hline&&&&&&&&&&&&\\
&&\mbox{n/a}&&&{\mathbf 0}& & &{\mathbf 0}& &&B^{(n)}&\\&&&&&&&&&&&&
\end{array}\right)
\end{equation}

where for sake of simplicity, we have written ${\mathbf 0}$ for any zero matrix independently of its size. Then, we apply a permutation on lines so that the lines $(1+u_i^1|\langle u_i , J^{(0)}_{2}(r) \rangle_z,\cdots,\langle u_i , J^{(0)}_{n}(r) \rangle_z)$ appear in the first
$n$ lines of a new matrix which is a block lower triangular matrix and which has the same determinant as the Jacobian determinant up to a possible minus sign: 
\begin{equation}\label{eq:jac2}
|\widetilde{\J}_{x_0}^{(M)}(r,u,u_1,\dots,u_n)|= |\det \left(\begin{array}{ccccccc}
\cline{1-4}
\multicolumn{1}{|c}{1+u_1^1}&\langle u_1 , J^{(0)}_{2}(r) \rangle_z&\cdots&\multicolumn{1}{c|}{\langle u_1 , J^{(0)}_{n}(r) \rangle_z}&~~&~~&~~\\
\multicolumn{1}{|c}{\vdots}&\vdots&~&\multicolumn{1}{c|}{\vdots}&~~&{\bf 0}&~~\\
\multicolumn{1}{|c}{1+u_n^1}& \langle u_n , J^{(0)}_{2}(r) \rangle_z&\cdots&\multicolumn{1}{c|}{\langle u_n , J^{(0)}_{n}(r) \rangle_z}& ~~&~~&~~\\
\cline{1-5}
&~~&~~&~~&\multicolumn{1}{|c|}{B^{(1)}}&~~&~~\\
\cline{5-5}\\
&~~&\mbox{n/a}&~~&~~&\ddots&~~\\
\cline{7-7}
&~~&~~&~~&~~&~~&\multicolumn{1}{|c|}{B^{(n)}}\\
\cline{7-7}
\end{array}\right)|. 
\end{equation}
Let us notice now that 
\begin{equation}
\langle u_i , J^{(0)}_{k}(r) \rangle_z = \sum_{j=1}^{n} u_i^j \langle v_j^{(0)}(r) , J^{(0)}_{k}(r) \rangle_z =  \sum_{j=2}^{n} u_i^j \langle v_j^{(0)}(r) , J^{(0)}_{k}(r)\rangle_z 
\end{equation}
since, from \cite[Chapter 5, Proposition 3.6]{Do92}, $\langle v_1^{(0)}(r) , J^{(0)}_{2}(r)\rangle_z=0$. It follows that 
\begin{equation}\label{eq:JacProd}
 \left( \begin{array}{cccc}
         1+u_1^1 & \langle u_1 , J^{(0)}_{2}(r) \rangle_z&  \cdots &\langle u_1 , J^{(0)}_{n}(r) \rangle_z \\
           \vdots & \vdots &~~&\vdots \\
         1+u_n^1 & \langle u_n , J^{(0)}_{2}(r) \rangle_z & \cdots &\langle u_n , J^{(0)}_{n}(r) \rangle_z
				\end{array} \right) =  \left( \begin{array}{cccc}
         1+u_1^1 & u_1^2&  \cdots & u_1^n \\
           \vdots & \vdots &~~&\vdots \\
         1+u_n^1 & u_n^2 & \cdots &u_n^n
				\end{array} \right) \times C 
\end{equation}
where 
\begin{equation}\label{eq:matC} C= \left(    \begin{array}{cccc}
                 1& 0& \cdots &0 \\
								 0&  \langle v_2^{(0)}(r) , J^{(0)}_{2}(r) \rangle_z& \cdots & \langle v_2^{(0)}(r) , J^{(0)}_{n}(r) \rangle_z \\
						     \vdots& \vdots& ~~&\vdots \\
								0&  \langle v_n^{(0)}(r) , J^{(0)}_{2}(r) \rangle_z& \cdots & \langle v_n^{(0)}(r) , J^{(0)}_{n}(r) \rangle_z
						 \end{array} \right) .\end{equation}

Inserting \eqref{eq:JacProd} into \eqref{eq:jac2}, we can write

\begin{align}
|\widetilde{\J}_{x_0}^{(M)}(r,u,u_1,\dots,u_n)|& = | \det \left( \begin{array}{cccc}
         1+u_1^1 & u_1^2&  \cdots & u_1^n \\
           \vdots & \vdots &~~&\vdots \\
         1+u_n^1 & u_n^2 & \cdots &u_n^n
				\end{array} \right) | \times |\det C| \prod_{i=1}^{n} |\det B^{(i)}| \notag \\
				&=n! \Delta(u,u_1,\dots,u_n)  |\det C| \prod_{i=1}^{n} |\det B^{(i)}| .\label{eq:jac3}
\end{align}

\noindent \textit{Step 5: expansion of the Jacobian determinant.} We derive from \eqref{eq:jac3} the expansion of the Jacobian determinant $\widetilde{\J}_{x_0}^{(M)}$ for small values of $r$. This requires to expand the determinants of $C$ and $B^{(i)}$. To this end, we first expand in Lemma \ref{lem:expCBi} below, the coefficients of $C$, given by \eqref{eq:matC} and $B^{(i)}$, given by \eqref{eq:coefB}. Here and in the sequel, $\reste_{*}(x_0,r,u,u_1,\dots,u_n)$ denotes a generic function that tends to $0$ when $r$ tends to $0$. 

\begin{lem}\label{lem:expCBi} 
The coefficients of $C$ satisfy for $l,m=2,\dots,n$ 
\begin{align}
 C_{m,m} &= \langle v_m^{(0)}(r) , J^{(0)}_{m}(r) \rangle_z = r-\frac{K_{x_0}^{(M)}(u,v_m^{(0)})}6r^3 +r^3\reste_{m,m}^{(0)}(x_0,r,u,u_1,\dots,u_n),\label{eq:Cdiag}\\
 C_{l,m} &= \langle v_l^{(0)}(r) , J^{(0)}_{m}(r) \rangle_z = r^2 \reste_{l,m}^{(0)}(x_0,r,u,u_1,\dots,u_n), l\neq m\label{eq:Cnondiag}.
\end{align}
For $1\le i \le n$, the coefficients of $B^{(i)}$ satisfy for $l,m=1,\dots,n$, 
\begin{align}
B^{(i)}_{m,m} &= \langle v^{(i)}_{m+1}(r), \tilde{J}^{(i)}_{m+1}(r) \rangle_{x_i}= r-\frac{K_z^{(M)}(u_i,v_{m+1}^{(i)})}6r^3 +r^3\reste_{m,m}^{(i)}(x_0,r,u,u_1,\dots,u_n)\label{eq:Bidiag}, \\
B^{(i)}_{l,m} &= \langle v^{(i)}_{l+1}(r), \tilde{J}^{(i)}_{m+1}(r) \rangle_{x_i}= r^2\reste_{l,m}^{(i)}(x_0,r,u,u_1,\dots,u_n), l\neq m \label{eq:Binondiag} ,
\end{align}
where $K_{x_0}^{(M)}$ and $K_z^{(M)}$ denotes the sectional curvatures of $M$ at $x_0$ and $z$ respectively.  
\end{lem}

\begin{proof}[Proof of Lemma \ref{lem:expCBi}]
let us consider the function $$ f_{l,m}(r,x_0,u,u_1,\dots,u_n) = C^{(0)}_{l,m} =  \langle v_l^{(0)}(r) , J^{(0)}_{m}(r) \rangle_z.$$
In anticipation of the proof of Lemma \ref{lem:resteCoef}, we make visible the dependency of the functional $f_{l,m}$ on $x_0$ and $u,u_1,\dots,u_n$ though only the derivatives of $f_{l,m}$ with respect to $r$ will be used in the lines below. Indeed, we wish to apply Taylor's theorem  to the function $r\mapsto f_{l,m}(r,x_0,u,u_1,\dots,u_n)$ for fixed $x_0,u,u_1,\dots,u_n$. We start by calculating the consecutive derivatives of $f_{l,m}$ with respect to $r$ at the point $r=0$.

Since $J^{(0)}_{m}(0)=0$ and $v_l^{(0)}(r)=v_l^{(0)}$, 
\begin{equation} \label{eq:deriv0} f_{l,m}(0,x_0,u,u_1,\dots,u_n) = \langle v_l^{(0)} , J^{(0)}_{m}(0) \rangle_{x_0}=0. \end{equation}
Let us notice that $v_l^{(0)}(r)$ is a parallel transport thus $v_l^{(0)\prime}(r)=0$. This implies that 
\begin{equation}\label{eq:deriv1} \frac{\partial f_{l,m}}{\partial r}(0,x_0,u,u_1,\dots,u_n) = \langle v_l^{(0)} , J^{(0)\prime}_{m}(0) \rangle_{x_0}=  \langle v_l^{(0)}, v^{(0)}_{m} \rangle_{x_0} = \delta_{l,m} .\end{equation}
Since $J^{(0)}_m$ is a Jacobi field, it satisfies the Jacobi equation \eqref{eqJacobi} and consequently,
\begin{equation}\label{eq:deriv2} \frac{\partial^2 f_{l,m}}{\partial r^2}(0,x_0,u,u_1,\dots,u_n) = \langle v_l^{(0)} , J^{(0)\prime\prime}_{m}(0) \rangle_{x_0} = - \langle v_l^{(0)}, \mathcal{R}_{x_0}^{(M)}( v_1^{(0)}, J^{(0)}_{m}(0))v_1^{(0)} \rangle_{x_0} = 0. \end{equation}
The expansion of $C_{l,m}$, $l\neq m$, now follows from Taylor's theorem combined with \eqref{eq:deriv0}-\eqref{eq:deriv2}. The expansion of $C_{m,m}$ requires the third derivative of $f_{m,m}$. Thanks to \cite[p.115]{Do92} and since $J^{(0)}_{m}(0)=0$, 
\begin{align}\label{eq:deriv3}
\frac{\partial^3 f_{l,m}}{\partial r^3}(0,x_0,u,u_1,\dots,u_n) = \langle v_m^{(0)} , J^{(0)\prime\prime\prime}_{m}(0) \rangle_{x_0} &=  - \langle v_m^{(0)}, \mathcal{R}_{x_0}^{(M)}( v_1^{(0)}, J^{(0)\prime}_{m}(0))v_1^{(0)} \rangle_{x_0} \notag\\&= - \langle v_m^{(0)}, \mathcal{R}_{x_0}^{(M)}( v_1^{(0)}, v_m^{(0)})v_1^{(0)} \rangle_{x_0} \notag\\&= K_{x_0}^{(M)}(u,v_m^{(0)}).
\end{align}
Now, applying Taylor's theorem at the fourth order to $f_{m,m}(\cdot,x_0,u,u_1,\dots,u_n)$ combined with \eqref{eq:deriv0}-\eqref{eq:deriv3}, we obtain the expansion \eqref{eq:Cdiag}. We omit the proof of \eqref{eq:Bidiag} and \eqref{eq:Binondiag} as it is very similar to the proof of \eqref{eq:Cdiag} and \eqref{eq:Cnondiag}. 
\end{proof}
We are now able to expand $\det C$ and $\det B^{(i)}$ and prove that 
\begin{align}
\det C & = r^{n-1} - \frac{ \Ric_{x_0}^{(M)}(u)}6r^{n+1} + r^{n+1} \reste_1(x_0,r,u,u_1,\dots,u_n), \label{eq:devDetC}\\
\det B^{(i)} &= r^{n-1} - \frac{ \Ric_{z}^{(M)}(u_i)}6r^{n+1} + r^{n+1} \reste_2(x_0,r,u,u_1,\dots,u_n), \label{eq:devDetB}
\end{align}
where $\Ric_{x_0}^{(M)}$ and $\Ric_{z}^{(M)}$ denote the Ricci curvatures at $x_0$ and $z$ respectively. 
. To this end, let us write 
\begin{align}
 \det C &= \sum_{\sigma \in \mathfrak{S}_{n-1}} \mathrm{sgn}(\sigma) \prod_{j=2}^{n} C_{\sigma(j),j} \notag \\
& = \prod_{j=2}^{n} C_{j,j} + \sum_{\sigma \in \mathfrak{S}_{n-1}\backslash\{ \mathrm{Id}\}} \mathrm{sgn}(\sigma) \prod_{j=2}^{n} C_{\sigma(j),j} \label{eq:detCsum}\\
\end{align} 
where $\mathfrak{S}_{n-1}$ denotes the set of permutations of $\{2,\dots,n\}$, $\mathrm{sgn}(\sigma)$ is the signature of the permutation $\sigma$ and $\mathrm{Id}$ is the identity of $\mathfrak{S}_{n-1}$. We expect the contribution of the first term in \eqref{eq:detCsum} to be dominant in the expansion \eqref{eq:devDetC}. Using expansion \eqref{eq:Cdiag}, we have
\begin{align}
\prod_{j=2}^{n} C_{j,j} &= \prod_{j=2}^{n} (r-\frac{K_z^{(M)}(u,v_j^{(0)})}6r^3 +r^3\reste^{(0)}_{j,j}(x_0,r,u,u_1,\dots,u_n)) \notag \\
&= r^{n-1} - \frac{r^{n+1}}6 \sum_{j=2}^n K_{x_0}^{(M)}(u,v_j^{(0)}) + r^{n+1}\reste(x_0,r,u,u_1,\dots,u_n) \notag\\ 
&= r^{n-1} - \frac{ \Ric_{x_0}^{(M)}(u)}6r^{n+1} + r^{n+1} \reste(x_0,r,u,u_1,\dots,u_n)\label{eq:mainC}
\end{align}
since $\mathcal{V}^{(0)}=\{ u=v_1^{(0)}, \dots,v_n^{(0)} \}$ is an orthonormal basis of $T_{x_0}M$. It remains to prove that the contribution of the second term in \eqref{eq:detCsum} is negligible with respect to $r^{n+1}$. Let $\sigma\neq \mathrm{Id}$. Then there are at least two indices $j$ such that $\sigma(j)\neq j$. Without loss of generality, we can assume $\sigma(2)\neq 2$ and $\sigma(3)\neq 3$ . From \eqref{eq:Cnondiag}, we know that, $C_{\sigma(2),2}= r^2 \reste_{\sigma(2),2}^{(0)}(x_0,r,u,u_1,\dots,u_n)$ and $C_{\sigma(3),3}=r^2 \reste_{\sigma(3),3}^{(0)}(x_0,r,u,u_1,\dots,u_n)$. Moreover, for $j\neq 2,3$, $C_{\sigma(j),j}$ is of order at least $r$, which means that there is a constant $c$ such that $ \prod_{j=4}^{n} C_{\sigma(j),j} \le cr^{n-3}$. 
Thus, 
\begin{align}
 \lim_{r\to 0}|\frac{1}{r^{n+1}}\prod_{j=2}^{n} C_{\sigma(j),j}| &= \lim_{r\to 0 } |\frac{C_{\sigma(2),2}}{r^2} \frac{C_{\sigma(3),3}}{r^2} \frac{\prod_{j=4}^{n} C_{\sigma(j),j}}{r^{n-3} } |\notag\\ 
& \le \lim_{r\to 0} c|\reste_{\sigma(2),2}^{(0)}(x_0,r,u,u_1,\dots,u_n)\reste_{\sigma(3),3}^{(0)}(x_0,r,u,u_1,\dots,u_n)|\notag \\
& = 0 \label{eq:Cneg}.
\end{align}
Combining \eqref{eq:mainC}, \eqref{eq:Cneg} with \eqref{eq:detCsum}, we obtain \eqref{eq:devDetC}. Similarly, if we now denote by $\mathfrak{S}_{n-1}$ the set of the permutation of $\{1,\dots,n-1\}$, we can write 
\begin{align}
 \det B^{(i)}= \prod_{j=1}^{n-1} B^{(i)}_{j,j} + \sum_{\sigma \in \mathfrak{S}_{n-1}\backslash\{ \mathrm{Id}\}} \mathrm{sgn}(\sigma) \prod_{j=1}^{n-1} B^{(i)}_{\sigma(j),j} \label{eq:detBsum}
\end{align}
As for \eqref{eq:mainC}, the expansion \eqref{eq:Bidiag} implies that 
\begin{equation}\label{eq:mainB}
\prod_{j=1}^{n-1} B^{(i)}_{j,j} = r^{n-1} - \frac{ \Ric_{z}^{(M)}(u_i)}6r^{n+1} + r^{n+1} \reste(x_0,r,u,u_1,\dots,u_n).
\end{equation}
Moreover, using same arguments as for \eqref{eq:Cneg}, we can prove that
\begin{equation}\label{eq:Bneg}
\lim_{r\to 0}|\frac{1}{r^{n+1}}\prod_{j=1}^{n-1} B^{(i)}_{\sigma(j),j}| =0.
\end{equation}
Inserting \eqref{eq:mainB} and \eqref{eq:Bneg} into \eqref{eq:detBsum}, we get the expansion \eqref{eq:devDetB}. 
Note that, when $r$ tends to $0$, $z=\exp_{x_0}(ru)$ tends to $x_0$, thus, by continuity of the Ricci curvature, 
\begin{equation}\label{eq:limRicci}
\Ric_{z}^{(M)}(u_i) =  \Ric_{x_0}^{(M)}(u_i) + \reste(x_0,r,u_i),
\end{equation}
where we recall that the vector $u_i$ on the right hand side must be understood as its parallel transport in $T_{x_0}M$.
Then, combining \eqref{eq:devDetB} with \eqref{eq:limRicci}, we are able to write the following expansion of $\det B^{(i)}$
\begin{equation}\label{eq:devDetBfinal}
\det B^{(i)}= r^{n-1} - \frac{\Ric_{x_0}^{(M)}(u_i) }6r^{n+1} + r^{n+1} \reste(x_0,r,u,u_1,\dots,u_n)
\end{equation}
The expansion of the Jacobian determinant $\widetilde{\J}_{x_0}^{(M)}$ now follows from \eqref{eq:jac3}, \eqref{eq:devDetC} and \eqref{eq:devDetBfinal}. 
\end{proof} 
We conclude this section by adding an additional information to the expansion of the Jacobian: the uniformity of the approximation with respect to $x_0\in M$ and to the vectors $u,u_1,\dots,u_n$. This is done in Proposition \ref{prop:jacobunif} below. 

For $x_0\in M$, we define the function $\reste$ through the identity 
\begin{equation}\label{eq:Jacobien}\widetilde{\mathcal J}_{x_0}^{(M)}(r,u,u_1,\dots,u_n) = r^{n^2-1} - \frac{\Ric_{x_0}^{(M)}(u)+\sum_{i=1}^{n}\Ric_{x_0}^{(M)}(u_i) }6r^{n^2+1} + r^{n^2+1} \reste(x_0,r,u,u_1,\dots,u_n).\end{equation}
\begin{prop}
\label{prop:jacobunif}
There exists $c>0$ such that for $r$ small enough, every $x_0\in M$ and vectors $u,u_1,\dots,u_n$, $$|\reste(x_0,r,u,u_1,\dots,u_n)|\le cr.$$  
\end{prop}
\begin{proof}
We start by recalling that because of \eqref{eq:jac3}, the determinant $\widetilde{\mathcal J}_{x_0}^{(M)}(r,u,u_1,\dots,u_n)$ is a homogeneous polynomial of degree $n^2$ in the coefficients of the matrices $C$ and $B^{(i)}$, $i=1,\dots,n$, and containing exclusively monomials of type $\prod_{l=1}^nC_{l,\sigma(l)}\prod_{i=1}^n\prod_{m=1}^nB^{(i)}_{m,\sigma_i(m)}$ where $\sigma,\sigma_1,\dots,\sigma_n$ are permutations of $\{1,\dots,n\}$. This implies that the remainder $r^{n^2+1}\reste(x_0,r,u,u_1,\dots,u_n)$ in \eqref{eq:Jacobien}  is a linear combination of products containing exclusively factors of type $r$, $K_{x_0}^{(M)}(\cdot,\cdot)r^3$ and at least one factor of type $r^3\reste_{m,m}^{(i)}$ or $r^2\reste_{l,m}^{(i)}$, $l\ne m$. A use of ($\mbox{A}_1$) and Lemma \ref{lem:resteCoef} below completes the proof of Proposition \ref{prop:jacobunif}. 
\end{proof}
The next lemma shows that the functions $\reste_{l,m}^{(i)}$ appearing in the Taylor expansions of Lemma \ref{lem:expCBi} are bounded by a linear function of $r$. 
\begin{lem}\label{lem:resteCoef}
There exist positive constants $c_{l,m}^{(k)}$, $k=0,\dots,n$, $l,m=1,\dots,n$ such that for every $r\in (0,1)$, $x_0$ and vectors $u,u_1,\dots,u_n$, $k=0,\dots,n$,
\begin{align}\label{unifReste}
|\reste_{l,m}^{(k)}(x_0,r,u,u_1,\dots,u_n)|  \le c_{l,m}^{(k)}r.
\end{align}
\end{lem}
\begin{proof}
We go back to the proof of Lemma  \ref{lem:expCBi} and. 
instead of using Taylor's theorem, we wish to apply Taylor's inequality  to the function $r\mapsto f_{l,m}(r,x_0,u,u_1,\dots,u_n)$ for fixed $x_0,u,u_1,\dots,u_n$ combined with uniform upper bounds for the third (resp. fourth)
derivative of $f_{l,m}$ with respect to $r$ when $l\neq m$ (resp. $l=m$). We start by calculating the derivatives of $f_{l,m}$ with respect to $r$: since $ v_l^{(0)}(r)$ is a parallel transport, its successive derivatives are zero. Moreover, applying the Jacobi equation \eqref{eqJacobi} to $J^{(0)}_{m}$, we get 
\begin{equation}
\frac{\partial^2 f_{l,m}}{\partial r^2}(r,x_0,u,u_1,\dots,u_n) =  \langle v_l^{(0)}(r) , J^{(0)\prime\prime}_{m}(r) \rangle_z = -  \langle v_l^{(0)}(r) , \mathcal{R}^{(M)}_z(v_1^{(0)}(r),J^{(0)}_{m}(r) ) v_1^{(0)}(r) \rangle_z.
\end{equation}
It follows that the third and fourth derivatives of $f_{l,m}$ with respect to $r$ can be written as 
\begin{align}
\frac{\partial^3 f_{l,m}}{\partial r^3}(r,x_0,u,u_1,\dots,u_n)& =-  \langle v_l^{(0)}(r) , \frac{\partial }{\partial r} \mathcal{R}^{(M)}_z(v_1^{(0)}(r),J^{(0)}_{m}(r) ) v_1^{(0)}(r) \rangle_z \label{eq:df3} 
\end{align}
and
\begin{align}
\frac{\partial^4 f_{l,m}}{\partial r^4}(r,x_0,u,u_1,\dots,u_n)&= -  \langle v_l^{(0)}(r) , \frac{\partial^2 }{\partial r^2}\mathcal{R}^{(M)}_z(v_1^{(0)}(r),J^{(0)}_{m}(r) ) v_1^{(0)}(r) \rangle_z. \label{eq:df4}
\end{align}
By the regularity of the curvature tensor, the functions $\frac{\partial^3 f_{l,m}}{\partial r^3}$ and $\frac{\partial^4 f_{l,m}}{\partial r^4}$ are continuous on the compact set $[0,1]\times M\times ({\mathcal S}^{n-1})^{n+1}$ where we  use with a slight abuse the same notation ${\mathcal S}^{n-1}$ for the unit sphere of both $T_{x_0}M$ and $T_zM$. Consequently, there exists a positive constant $c$ such that for every $r\le 1$, $x_0\in M$ and every $u,u_1,\dots,u_n$,
\begin{align}
\bigg|\max\left(\frac{\partial^3 f_{l,m}}{\partial r^3},\frac{\partial^4 f_{l,m}}{\partial r^4}\right)(r,x_0,u,u_1,\dots,u_n)\bigg|&\le c. \label{eq:bornef3}
\end{align}
Taylor's inequality combined with \eqref{eq:bornef3} 
implies that for every $r\le 1$, $x_0\in M$ and $u,u_1,\dots,u_n$,
\begin{align*}
r^3|\reste_{m,m}^{(0)}(x_0,r,u,u_1,\dots,u_n)|&\le \frac{r^4}{24}\sup_{[0,1]\times M\times ({\mathcal S}^{n-1})^{n+1}}\bigg|\frac{\partial^4 f_{l,m}}{\partial r^4}\bigg|\le \frac{c}{24} r^4
\end{align*}
and for $l\ne m$,
\begin{align*}
r^2|\reste_{l,m}^{(0)}(x_0,r,u,u_1,\dots,u_n)| &\le \frac{r^3}{6}\sup_{[0,1]\times M\times ({\mathcal S}^{n-1})^{n+1}}\bigg|\frac{\partial^3 f_{l,m}}{\partial r^3}\bigg|\le \frac{c}{6} r^3. 
\end{align*}
This proves \eqref{unifReste} for $k=0$. For $k=1,\dots,n$,  \eqref{unifReste} follows by the same method. 
\end{proof}
We end this section with the proof of the explicit formula of $\widetilde{\mathcal J}_{x_0}^{(M)}(r,u,u_1,\dots,u_n)$ in the case $M=\mathcal{H}_k^n$.
\begin{proof}[Proof of \eqref{eq:formuleBPhyp}] Let us go back to \eqref{eq:jac3}. From \cite[p 113]{Do92}, we know that for $M=\mathcal{H}_k^n$, the Jacobi fields satisfy 
\begin{align}
J_j^{(0)}(r) &= \frac{\sinh(kr)}{k} v_j^{(0)}(r),\\
\tilde{J}_j^{(i)}(r) &= \frac{\sinh(kr)}{k} v_j^{(i)}(r).
\end{align}
Hence, we can rewrite the matrices $C$ and $B^{(i)}$ as 
\begin{equation}
C=\left(  \begin{array}{c|c}
           1&0 \\
					\hline\\
					0& \frac{\sinh(kr)}k \mathrm{Id}_{(n-1)}\\
					&~
					\end{array}
   \right), ~~ B^{(i)}= \frac{\sinh(kr)}k \mathrm{Id}_{(n-1)},
\end{equation} 
where $\mathrm{Id}_{(n-1)}$ denotes the identity matrix of size $(n-1)$. 
Inserting into \eqref{eq:jac3}, we obtain that 
\begin{equation}
\widetilde{\mathcal J}_{x_0}^{({\mathcal H}_k^n)}(r,u,u_1,\dots,u_n) = n! \Delta(u,u_1,\dots,u_n) \left(\frac{\sinh(kr)}{k}\right)^{n^2-1}, 
\end{equation}
which completes the proof of \eqref{eq:formuleBPhyp}. 
\end{proof}

\section{Mean cardinality and density of the set of vertices of $\Cl$: proofs of Theorem \ref{THM} and Theorem \ref{DENS}}\label{sec:vert}
In this section, we prove Theorems \ref{THM} and \ref{DENS} which contain an asymptotic expansion of $\E[\Nl]$ and of the density of the intensity measure of the point process ${\mathcal V}_{x_0,\lambda}^{(M)}$ as well as explicit formulas in the particular cases of ${\mathcal S}_k^{n}$ and ${\mathcal H}_k^n$. 

\subsection{ Proof of Theorem \ref{THM} (i) } 
Each vertex of ${\mathcal C}_{x_0,\lambda}^{(M)}$ is a circumcenter of $x_0$ and $n$ distinct points $x_1,\dots,x_n$ of $\Pl$ such that the associated circumscribed ball contains no  point of the point process $\Pl$ in its interior. Thanks to Section \ref{sec:locdep} and Assumption \mbox{($\mbox{A}_4$)}, we can assume that $M$ is included in ${\mathcal B}^{(M)}(x_0,r_{\mbox{\tiny{max}}})$ where $r_{\mbox{\tiny{max}}}$ is given by \mbox{($\mbox{A}_4$)}. This implies that for any $x_1,\dots,x_n\in M$, there is at most one unique circumscribed ball $\bo^{(M)}(x_0,x_1,\dots,x_n)$ with center in ${\mathcal B}^{(M)}(x_0,r_{\mbox{\tiny{max}}})$ and radius less than $r_{\mbox{\tiny{max}}}$.
Consequently, using the Mecke-Slivnyak formula, see e.g. \cite[Proposition 4.1.1]{Mol94}, we get 
\begin{align}
 \E[\Nl] &= \E\left[ \sum_{\{x_1,\dots,x_n\} \subset \Pl} \mathds{1}_{\{ \bo^{(M)}(x_0,x_1,\dots,x_n) \mbox{\tiny{ exists and }} \bo^{(M)}(x_0,x_1,\dots,x_n) \cap \Pl= \emptyset \}} \right], \notag \\
&=  \frac{\lambda^n}{n!} \int_{M^n} \proba(\bo^{(M)}(x_0,x_1,\dots,x_n) \mbox{ exists and does not meet $\Pl$}) 
\dM(x_1) \dots \dM(x_n), \notag \\
&=  \frac{\lambda^n}{n!} \int_{\tilde{M}^n} e^{-\lambda \vol^{(M)}(\bo^{(M)}(x_0,x_1,\dots,x_n))} \dM(x_1) \dots \dM(x_n).\label{Sliv}
\end{align}
where 
 \[ \tilde{M}^n = \{ (x_1,\dots,x_n) \in M^n, \text{ the circumscribed ball }\bo^{(M)}(x_0,x_1,\dots,x_n) \text{ exists} \} .\]
The integral above is of the type which is classically treated via Laplace's method \cite[Chap 5 and 9]{Ble86}. Indeed, we expect that only the small circumscribed balls which naturally correspond to points $x_1,\dots,x_n$ in a small neighbourhood of $x_0$, will contribute significantly and the rest will decay exponentially fast. The expansion of the integral requires an expansion of the volume of the ball $\bo^{(M)}(x_0,x_1,\dots,x_n)$ given in Lemma \ref{lem:volunif} and of the volume element $\dM(x_1) \dots \dM(x_n)$ given by \eqref{eq:egalitemesuressec5} combined with Theorem \ref{Jacob}. 

\noindent \textit{Step 1: decomposition of $\E[\Nl]$ into two integrals.} As in Section \ref{sec:vol}, let $r_\lambda = \lambda^{-\frac{n+1}{n(n+2)}}$ and let $M^n_\lambda$ be the subset of $M^n$ defined by 

\[ M^{(n)}_\lambda = \{ (x_1,\dots,x_n)\in M^n \mbox{ s.t. $\bo^{(M)}(x_0,x_1,\dots,x_n)$ exists and has radius smaller than }r_\lambda \} .\]
Then the integral in \eqref{Sliv} can be written as 
\begin{equation}
\E[\Nl] = I_\lambda + \tilde{I}_\lambda, \label{eq:decompEN}
\end{equation}
where 
\begin{align}
I_\lambda&= \frac{\lambda^n}{n!} \int_{M_\lambda^{(n)}} e^{-\lambda \vol^{(M)}(\bo^{(M)}(x_0,x_1,\dots,x_n))} \dM(x_1) \dots \dM(x_n) \label{eq:defIvert}\\
\intertext{and}
\tilde{I}_\lambda&=  \frac{\lambda^n}{n!} \int_{\tilde{M}^n\backslash M_\lambda^{(n)}} e^{-\lambda \vol^{(M)}(\bo^{(M)}(x_0,x_1,\dots,x_n))} \dM(x_1) \dots \dM(x_n). \label{eq:defItildevert}
\end{align} 

\noindent \textit{Step 2: $\tilde{I}_\lambda$ is negligible.} We show that $\tilde{I}_\lambda$ given by \eqref{eq:defItildevert} is negligible in front of $\lambda^{-\frac 2n}$. By definition of the set $M^{(n)}_\lambda$, for every $(x_1,\dots,x_n) \in \tilde{M}^n\setminus M_\lambda^{(n)}$, the circumscribed ball $\bo^{(M)}(x_0,\dots,x_n)$ has a radius greater than $\lambda^{-\frac{n+1}{n(n+2)}}$. Then, since $M$ is compact, Lemma \ref{lem:volunif} (ii) implies that there exists a positive constant $c$ such that for $\lambda$ large enough and for every $(x_1,\dots,x_n) \in M^n\setminus M_\lambda^{(n)}$, 
\begin{equation}\label{eq:minorVolIneg}
\vol^{(M)}( \bo^{(M)}(x_0,\dots,x_n)) \ge c\lambda^{-\frac{n+1}{n+2}}.
\end{equation}
Inserting \eqref{eq:minorVolIneg} into \eqref{eq:defItildevert}, we obtain that
\[ \tilde{I}_\lambda \le \vol^{(M)}(M)^n e^{-c\lambda^{1-\frac{n+1}{n+2}}} ,\]
which implies 
\begin{equation}\label{eq:ItildeNeg}
\tilde{I}_\lambda = o\left( \frac{1}{\lambda^{\frac 2n}}\right). 
\end{equation}

\noindent \textit{Step 3: estimate of $I_\lambda$.} Let us now prove that when $\lambda\to \infty$, 

\begin{equation}\label{eq:devFinI}
I_\lambda = e_n - \frac{d_n \Scal_{x_0}^{(M)}}{\lambda^{\frac 2n}} + o\left( \frac{1}{\lambda^{\frac 2n}}\right), 
\end{equation}
where the constants $e_n$ and $d_n$ are given in Theorem \ref{THM} (i). To this end, we apply the change of variables defined by the transformation $\Phi_{x_0}^{(M)}$ given by \eqref{change}. We can do so because $\Phi_{x_0}^{(M)}$ is a global diffeomorphism as an injective function with non-vanishing Jacobian determinant. We get
\begin{multline}\label{eq:INouvCoor}
I_\lambda = \frac{\lambda^n}{n!} \int_{0}^{\lambda^{-\frac{n+1}{n(n+2)}}} \int e^{-\lambda\vol^{(M)}(\bo^{(M)}(\exp_{x_0}(ru),r))} | \widetilde{\J}_{x_0}^{(M)}(r,u,u_1,\dots,u_n)| \\  \dvol^{(\sph^{n-1})}(u)\dvol^{(\sph^{n-1})}(u_1)\dots\dvol^{(\sph^{n-1})}(u_n)\dd r .
\end{multline}
The estimation of $I_\lambda$ relies on the expansion of both $\vol^{(M)}(\bo^{(M)}(\exp_{x_0}(ru),r))$ and $\widetilde{\J}_{x_0}^{(M)}(r,u,u_1,\dots,u_n)$. Let $\varepsilon>0$. Theorem \ref{Jacob} implies that for $\lambda$ large enough, $0<r<\lambda^{-\frac{n+1}{n(n+2)}}$, $u\in T_{x_0}M$, $\|u\|_{x_0} =1$ and $u_i\in T_{\exp_{x_0}(ru)}M$ with $\|u_i\|_{\exp_{x_0}(ru)}=1$, $i=1,\dots,n$, 
 \begin{align}
\bigg|\frac{\widetilde{\J}_{x_0}^{(M)}(r,u,u_1,\dots,u_n)}{n! \Delta(u,u_1,\dots,u_n)}-r^{n^2-1} + \frac{L_{x_0}^{(M)}(u,u_1,\dots,u_n)}6 r^{n^2+1}\bigg|\le \varepsilon r^{n^2+1}.
\label{eq:inegJacobExp}
\end{align}
Inserting \eqref{eq:inegJacobExp} and \eqref{eq:inegVolumeC} into \eqref{eq:INouvCoor}, we obtain
\begin{equation}\label{eq:inegI-+} I_{\lambda,-} \le I_\lambda \le I_{\lambda,+} \end{equation} 
where
\begin{multline}
I_{\lambda,\pm} = \int_{0}^{\lambda^{-\frac{n+1}{n(n+2)}}} \int \varphi_{\lambda,\pm}(r,u,u_1,\dots,u_n)
\dvol^{(\sph^{n-1})}(u)\dvol^{(\sph^{n-1})}(u_1)\dots \dvol^{(\sph^{n-1})}(u_n) r^{n-1} \dd r, 
\end{multline}
and the functions $\varphi_{\lambda,+}$ and $\varphi_{\lambda,-}$ are defined by the equalities
\begin{align*}
\varphi_{\lambda,-}(r,u,u_1,\dots,u_n)&=e^{-\lambda\kappa_n r^n} \left[ r^{n(n-1)} +\left( \frac{\kappa_n \Scal_{x_0}^{(M)}}{6(n+2)} -\varepsilon \right) \lambda{r^{n^2+2}}- \left(\frac{L_{x_0}^{(M)}(u,u_1,\dots,u_n)}6 +\varepsilon  \right)r^{n^2-n+2}\notag\right.\\&\hspace*{2cm}\left. -\left( \frac{\kappa_n \Scal_{x_0}^{(M)}}{6(n+2)} -\varepsilon \right)\left(\frac{L_{x_0}^{(M)}(u,u_1,\dots,u_n)}6 +\varepsilon  \right)\lambda r^{n^2+4} \right] 
\Delta(u,u_1,\dots,u_n)
\end{align*}
and
\begin{align*}
\varphi_{\lambda,+}(r,u,u_1,\dots,u_n)&=e^{-\lambda\kappa_n r^n} \left[ r^{n(n-1)} +\left( \frac{\kappa_n \Scal_{x_0}^{(M)}}{6(n+2)} +2\varepsilon \right) \lambda{r^{n^2+2}}- \left(\frac{L_{x_0}^{(M)}(u,u_1,\dots,u_n)}6 -\varepsilon  \right)r^{n^2-n+2} \right. \\&\hspace*{2cm}
 \left. -\left( \frac{\kappa_n \Scal_{x_0}^{(M)}}{6(n+2)} +2\varepsilon \right)\left(\frac{L_{x_0}^{(M)}(u,u_1,\dots,u_n)}6 -\varepsilon  \right)\lambda r^{n^2+4} \right] \Delta(u,u_1,\dots,u_n). \\  
\end{align*}
In order to derive an upper bound for $I_{\lambda,+}$, we proceed in the exact same way  as in the proof of Theorem \ref{thm:propvol} (i), Step 3, i.e.
\begin{itemize}
\item[-]  we start by applying the change of variables $y=\lambda\kappa_n r^n$, taking $y$ from $0$ to $\infty$,
\item[-]  we use Assumption \mbox{($\mbox{A}_1$)} to show that the integral of the last term of $\varphi_{\lambda,+}$ is bounded by $\lambda^{-\frac4n}$ up to a multiplicative constant,
\item[-]  we calculate the integral of the remaining terms over $u,u_1,\dots,u_n$ first and then over $y$.
\end{itemize}
We do not add details for the first two points since they come from a straightforward adaptation of Step 3 from the proof of Theorem \ref{thm:propvol} (i). The calculation of the integral over $u,u_1,\dots,u_n$ goes along the following lines. We start by recalling the formula below, see \cite[Theorem 2]{Mil71bis}:
 \begin{equation*}
 \int_{u,u_1,\dots,u_n} \Delta(u,u_1,\dots,u_n) \dvol^{(\sph^{n-1})}(u)\dvol^{(\sph^{n-1})}(u_1)\dots \dvol^{(\sph^{n-1})}(u_n) = \frac{2^{n+1}}{n!} \frac{ \Gamma\left( \frac{n^2+1}2 \right)}{ \Gamma\left( \frac{n^2}2 \right)} \frac{\pi^{\frac{n^2+n-1}2}}{ \Gamma\left( \frac{n+1}2 \right)^n}.
 \end{equation*}
 Denoting by $k_n$ the constant above and using Fubini's theorem combined with \eqref{eq:scalRicci}, we obtain in particular
  \begin{align*}
  \int_{u,u_1,\dots,u_n}&\Delta(u,u_1,\dots,u_n)\Ric_{x_0}^{(M)}(u)\dvol^{(\sph^{n-1})}(u)\dvol^{(\sph^{n-1})}(u_1)\dots \dvol^{(\sph^{n-1})}(u_n)\notag\\ & = \int_{u} \Ric_{x_0}^{(M)}(u) \int_{u_1,\dots,u_n}\Delta(u,u_1,\dots,u_n))\dvol^{(\sph^{n-1})}(u_1)\dots \dvol^{(\sph^{n-1})}(u_n)\dvol^{(\sph^{n-1})}(u) \notag \\
  & =  \int \Ric_{x_0}^{(M)}(u) \frac{k_n}{\vol^{(\sph^{n-1})}(\sph^{n-1})} \dvol^{(\sph^{n-1})}(u) \notag \\
  & =  \frac{k_n \Scal_{x_0}^{(M)}}{n}. 
  \end{align*}
 Consequently, we get
 \begin{align*}
 &\int_{u,u_1,\dots,u_n}\Delta(u,u_1,\dots,u_n)L_{x_0}^{(M)}(u,u_1,\dots,u_n)\dvol^{(\sph^{n-1})}(u)\dvol^{(\sph^{n-1})}(u_1)\dots \dvol^{(\sph^{n-1})}(u_n)\notag \\ & = \sum_{i=0}^{n} \int_{u_0,u_1,\dots,u_n}\Delta(u_0,u_1,\dots,u_n)\Ric_{x_0}^{(M)}(u_i)\dvol^{(\sph^{n-1})}(u_0)\dvol^{(\sph^{n-1})}(u_1)\dots \dvol^{(\sph^{n-1})}(u_n) \notag \\
 &= \frac{k_n (n+1) \Scal_{x_0}^{(M)}}{n}. 
 \end{align*}
The integration of $\varphi_{\lambda,+}$ over $u,u_1,\dots,u_n$ then leads us to 
\begin{align}\label{eq:majorIfinal}
I_{\lambda,+} &\le k_n \int_{0}^{\infty} e^{-y} \left[ \y^{n-1} + \left( \frac{\kappa_n \Scal_{x_0}^{(M)}}{6(n+2)} +2\varepsilon \right)\y^{n+\frac 2n}\frac{1}{\lambda^{\frac 2n}}\right.\notag\\
&\hspace*{5cm} \left.-\left( \frac{ (n+1)\Scal_{x_0}^{(M)}}{n} - \varepsilon \right) y^{n-1+\frac 2n}\frac{1}{\lambda^{\frac 2n}} \right] \frac{\dd y}{\kappa_n}+\frac{c}{\lambda^{\frac4n}}. 
\end{align}
Finally, integrating over $y$ in \eqref{eq:majorIfinal}, for $\lambda$ large enough, we obtain 
\begin{align}
I_{\lambda,+} \le  e_n - \frac{\Scal_{x_0}^{(M)}}{\lambda^{\frac 2n}} d_n + \frac{c\varepsilon}{\lambda^{\frac 2n}}, \label{eq:I+final}
\end{align}
where $e_n$ and $d_n$ are given by \eqref{eq:Nmoyeuc} and \eqref{eq:defdn} respectively.
The lower bound $I_{\lambda,-}$ can be handled in the same way, so that 
\begin{equation}
I_{\lambda,-} \ge  e_n - \frac{\Scal_{x_0}^{(M)}}{\lambda^{\frac 2n}} d_n - \frac{c\varepsilon}{\lambda^{\frac 2n}} .\label{eq:minorI-}
\end{equation}
Using \eqref{eq:inegI-+}, \eqref{eq:I+final} and \eqref{eq:minorI-}, we get \eqref{eq:devFinI} which, combined with \eqref{eq:ItildeNeg}, concludes the proof of Theorem \ref{THM} (i). 
\subsection{ Proof of Theorem  \ref{DENS}, (i)}
Let $h$ be a measurable, bounded and non-negative test function on $\R_+\times {\mathcal S}^{n-1}$. We aim at finding a function $g_\lambda^{(M)}$ which satisfies 
$$\E\left(\sum_{(r,u)\in {\mathcal V}_{x_0,\lambda}^{(M)}}h(r,u)\right)=\int h(r,u)g_\lambda^{(M)}(r,u)\mathrm{d}r\mathrm{d}\mbox{vol}^{({\mathcal S}^{n-1})}(u).$$
To do so, we use again the fact that each vertex of $\Cl$ is the center of a ball which is circumscribed to $x_0$ and to $n$ distinct points $x_1,\dots,x_n$ of ${\mathcal P}_\lambda$. For any $x_1,\dots,x_n$, let us define $R:=R(x_0,\dots,x_n)>0$ and $U:=U(x_0,\dots,x_n)\in {\mathcal S}^{n-1}$ such that $\exp_{x_0}(\lambda^{-\frac1n}R(x_0,\dots,x_n)U(x_0,\dots,x_n))$ is the associated circumscribed center when it exists. Using the Mecke-Slivnyak formula, we get
\begin{align*}
& \E\left(\sum_{(r,u)\in {\mathcal V}_{x_0,\lambda}^{(M)}}h(r,u)\right)\\&=\E\left(\sum_{\{x_1,\dots,x_n\} \subset \Pl} 
h(R,U)\mathds{1}_{\{ \bo^{(M)}(x_0,x_1,\dots,x_n) \mbox{\tiny{ exists and }} \bo^{(M)}(x_0,x_1,\dots,x_n) \cap \Pl= \emptyset \}} \right)\\ 
&=\frac{\lambda^{n}}{n!}\int_{M^n} h(R,U)\proba(\bo^{(M)}(x_0,x_1,\dots,x_n) \mbox{{\small{ exists, }}$\bo^{(M)}(x_0,x_1,\dots,x_n) \cap \Pl= \emptyset$}) 
\dM(x_1) \dots \dM(x_n).
\end{align*}
We then follow line by line each step of the proof of Theorem \ref{THM} (i). For sake of brevity, we skip the details and only sketch the strategy, i.e.
\begin{itemize}
\item[-] we decompose the integral above and isolate the contribution of $(x_1,\dots,x_n)\in M_\lambda^{(n)}$,
\item[-] we show that the remainder is negligible,
\item[-] we apply the change of variables provided by \eqref{change} in the integral over $M_\lambda^{(n)}$, use the approximation of the Jacobian and of the volume of a ball given in Theorem \ref{Jacob}  and Lemma \ref{lem:volunif} respectively,
\item[-] we finally integrate over $u_1,\dots,u_n$ and the result follows.
\end{itemize}

\subsection{Proofs of Theorem \ref{THM} and Theorem \ref{DENS}, (ii)}
Let us derive the expectation of both the number of vertices of $\Cl$ and the density of the intensity measure of the point process ${\mathcal V}_{x_0,\lambda}^{(M)}$ in the particular cases  $M=\sph_k^n$ and $M=\mathcal{H}_k^n$. These both rely on the fact that the Jacobian determinant of the change of variables given by \eqref{change} and the volume of a geodesic ball have an explicit expression. Since  Theorem \ref{THM} (ii)  is immediately obtained by integrating the densities of Theorem \ref{DENS} (ii), we only prove the latter. 
\paragraph{Case $M=\sph^{n}_k$.} For any $(n+1)$-tuple, $x_0,\dots,x_n$  of $\sph^{n}_k$, there are exactly two circumscribed balls with antipodal centers: the smallest one with a radius $r<\frac{\pi}{2k}$, denoted by $\bo^{(\sph_k^n)}_1(y_0,\dots,y_n)$ and the largest one, denoted by $\bo^{(\sph_k^n)}_2(y_0,\dots,y_n)$, with radius $\frac{\pi}k -r$. In particular, $\bo^{(\sph_k^n)}_2(y_0,\dots,y_n) = \sph^{n-1}_k \backslash \bo^{(\sph_k^n)}_1(y_0,\dots,y_n)$. 
Then 
applying the Mecke-Slivnyak formula, we obtain
\begin{align}\label{eq:MeckeSn}
&\E\left(\sum_{(r,u)\in {\mathcal V}_{x_0,\lambda}^{({\mathcal{S}_k^n})}}h(r,u)\right)\nonumber\\
&= \frac{\lambda^n}{n!} \int_{x_1,\dots,x_n \in {\mathcal{S}_k^n}} h(R,U) \left( e^{-\lambda\vol^{({\mathcal{S}_k^n})}(\bo^{(\sph_k^n)}_1(x_0,\dots,x_n))} +   e^{-\lambda\vol^{({\mathcal{S}_k^n})}(\bo^{(\sph_k^n)}_2(x_0,\dots,x_n))} \right)\dvol^{({\mathcal{S}_k^n})}(x_1) \dots \dvol^{({\mathcal{S}_k^n})}(x_n).
\end{align}
where $R=R(x_1,\dots,x_n)$ and $U=U(x_1,\dots,x_n)$ are the functionnals introduced in the previous section.
In \eqref{eq:MeckeSn}, we proceed with the following change of variables:
\begin{equation}
\label{eq:changeRescale}
  x_i = \exp_{|\exp_{x_0}(\lambda^{-\frac 1n}ru)} (\lambda^{-\frac 1n}r u_i).
 \end{equation}
From Proposition \ref{const}, the Jacobian determinant of this change of variables in $\sph_k^n$ is 
\begin{equation}\label{eq:BPSn}
n! \Delta(u,u_1,\dots,u_n)\lambda^{-\frac 1n} \left(\frac{\sin(k\lambda^{-\frac 1n}r)}{k}\right)^{n^2-1}.  
\end{equation}
Now, in these new coordinates, 
\begin{align}
\vol^{({\mathcal{S}_k^n})}(\bo^{({\mathcal{S}_k^n})}_1(x_0,\dots,x_n)) &= \sigma_{n-1} \int_{0}^{\lambda^{-\frac 1n}r} \left(\frac{\sin(kt)}{k}\right)^{n-1} \mathrm{d}t \notag \\
& = \sigma_{n-1} \int_{0}^{r} \left(\frac{\sin(k\lambda^{-\frac 1n}t)}{k}\right)^{n-1} \lambda^{\frac 1n}\mathrm{d}t
\label{eq:volSn}
\end{align}
and similarly, 
\begin{equation}\label{eq:volSnbis}
\vol^{({\mathcal{S}_k^n})}(\bo^{({\mathcal{S}_k^n})}_2(x_0,\dots,x_n)) = \sigma_{n-1} \int_{r}^{\lambda^{\frac 1n}\frac{\pi}k} \left(\frac{\sin(k\lambda^{-\frac 1n}t)}{k}\right)^{n-1} \lambda^{\frac 1n}\mathrm{d}t.
\end{equation}
Inserting \eqref{eq:BPSn}, \eqref{eq:volSn} and \eqref{eq:volSnbis} into \eqref{eq:MeckeSn}, we obtain 
\begin{align*}
&\E\left(\sum_{(r,u)\in {\mathcal V}_{x_0,\lambda}^{({\mathcal{S}_k^n})}}h(r,u)\right)\nonumber\\&= \lambda^{n}\int_{r=0}^{\lambda^{\frac 1n}\frac{\pi}k}\int_{u,u_1,\dots,u_n \in \sph^{n-1}} h(R,U) \left( e^{-\lambda\sigma_{n-1} \int_{0}^{r} \left(\frac{\sin(k\lambda^{-\frac 1n}t)}{k}\right)^{n-1} \lambda^{\frac 1n}\mathrm{d}t} + e^{-\lambda  \sigma_{n-1} \int_{r}^{\lambda^{\frac 1n}\frac{\pi}k} \left(\frac{\sin(k\lambda^{-\frac 1n}t)}{k}\right)^{n-1} \lambda^{\frac 1n}\mathrm{d}t}     \right) \\&\times \Delta(u,u_1,\dots,u_n)\lambda^{-\frac 1n}\left(\frac{\sin(k\lambda^{-\frac 1n}r)}{k}\right)^{n^2-1} \mathrm{d}r \dvol^{(\sph^{n-1})}(u)\dvol^{(\sph^{n-1})}(u_1)\dots\dvol^{(\sph^{n-1})}(u_n) \\
&= k_n  (k\lambda^{-\frac 1n})^{1-n^2 }\int_{r=0}^{\lambda^{\frac 1n}\frac{\pi}k}\int_{u\in \sph^{n-1}} h(R,U) \left( e^{- (k\lambda^{-\frac 1n})^{1-n}\sigma_{n-1} \int_{0}^{r}\sin(k\lambda^{-\frac 1n}t) \mathrm{d}t } +e^{- (k\lambda^{-\frac 1n})^{1-n}\sigma_{n-1} \int_{r}^{\lambda^{\frac 1n}\frac{\pi}k}\sin(k\lambda^{-\frac 1n}t) \mathrm{d}t } \right)\\& \times\sin(k\lambda^{-\frac 1n}r) \mathrm{d}r \dvol^{\sph^{n-1}}(u)
\end{align*}
where $k_n$ is the same constant as in the case $\R^n$.

It follows, since $h$ is a test function, that the intensity measure has the density 

\begin{equation} \label{eq:densSn}
g_\lambda^{({\mathcal{S}_k^n})}(r,u) =a_n\bigg(e^{-\sigma_{n-1} \displaystyle\int_{0}^{r}{\mathfrak{s}}_{k\lambda^{-\frac 1n}}^{n-1}(t) \mathrm{d}t } +e^{- \sigma_{n-1} \displaystyle\int_{r}^{\lambda^{\frac 1n} \frac{\pi}{k}}
{\mathfrak{s}}_{k\lambda^{-\frac 1n}}^{n-1}(t) \mathrm{d}t }\bigg)
{\mathfrak{s}}_{k\lambda^{-\frac 1n}}^{n^2-1}(r) 
\end{equation}
where $a_n$ is given in Theorem \ref{DENS} and ${\mathfrak{s}}_{\alpha}(t)=
                        \frac{\sin(\alpha t)}{\sqrt{\alpha}} $.

\paragraph{Case $M=\mathcal{H}_k^{n}$.} 
We proceed in the exact same way, save for the notable difference that the circumscribed ball is unique.
We deduce that
\begin{align}\label{eq:MeckeHn}
&\E\left(\sum_{(r,u)\in {\mathcal V}_{x_0,\lambda}^{({\mathcal{H}_k^n})}}h(r,u)\right)\nonumber\\
&= \frac{\lambda^n}{n!} \int_{x_1,\dots,x_n \in {\mathcal{H}_k^n}} h(R,U) e^{-\lambda\vol^{({\mathcal{H}_k^n})}(\bo^{({\mathcal{H}_k^n})}(x_0,\dots,x_n))} \dvol^{({\mathcal{H}_k^n})}(x_1) \dots \dvol^{({\mathcal{H}_k^n})}(x_n).
\end{align} 
where $R=R(x_1,\dots,x_n)$ and $U=U(x_1,\dots,x_n)$ are the functionnals introduced in the previous section.
From Proposition \ref{const}, the Jacobian determinant of the change of variables given by \eqref{eq:changeRescale} is 
\begin{equation}\label{eq:BPHn}
 \widetilde{\J}_{x_0,\lambda}^{({\mathcal{H}_k^n})}(r,u,u_1,\dots,u_n) = n! \Delta(u,u_1,\dots,u_n)\lambda^{-\frac 1n} \left(\frac{\sinh(k\lambda^{-\frac 1n}r)}{k}\right)^{n^2-1}.  
\end{equation}
Now, in this new coordinates, 
\begin{align}
\vol^{({\mathcal{H}_k^n})}(\bo^{({\mathcal{H}_k^n})}(x_0,\dots,x_n)) &= \sigma_{n-1} \int_{0}^{\lambda^{-\frac 1n}r} \left(\frac{\sinh(kt)}{k}\right)^{n-1} \mathrm{d}t \notag \\
& = \sigma_{n-1} \int_{0}^{r} \left(\frac{\sinh(k\lambda^{-\frac 1n}t)}{k}\right)^{n-1} \lambda^{\frac 1n}\mathrm{d}t
\label{eq:volHn}
\end{align}
Inserting \eqref{eq:BPHn} and \eqref{eq:volHn} into \eqref{eq:MeckeHn}, we obtain that
the intensity measure has the density 
\begin{equation} \label{eq:densHn}
g_\lambda^{({\mathcal{H}_k^n})}(r,u) =a_ne^{-\sigma_{n-1} \displaystyle\int_{0}^{r}{\mathfrak{s}}_{-k\lambda^{-\frac 1n}}^{n-1}(t) \mathrm{d}t }
{\mathfrak{s}}_{-k\lambda^{-\frac 1n}}^{n^2-1}(r).
\end{equation}
where 
${\mathfrak{s}}_{\alpha}(t)=
												\frac{\sinh((-\alpha)t)}{-\alpha}$.  
			


\section{Sectional Voronoi tessellation: proof of Theorems \ref{volSec} and \ref{NSec}}\label{sec:sec}
In this section, we prove Theorems \ref{volSec} and \ref{NSec} which contain asymptotic expansions of both the volume and the number of vertices of a cell from a section of the original Voronoi tessellation. The methods being very close to what has already been done in the proofs of Theorem \ref{thm:propvol} and Theorem \ref{THM}, we will omit several technical details. 

Let $V_s \subset T_{x_0}M$ be a vector space of dimension $s$ and recall that $M_s = \exp_{x_0}(V_s)$. 

\subsection{Proof of Theorem  \ref{volSec}, (i)}
By Fubini's theorem, the mean $s$-content of $\Cl\cap M_s$ is 
\begin{equation}\label{eq:fubSection}
\E[ \vol^{(M_s)}(\Cl\cap M_s)] = \int_{M_s} e^{-\lambda \vol^{(M)}(\bo^{(M)}(x, d^{(M)}(x_0,x)))} \dvol^{(M_s)}(x).
\end{equation}
As in the proof of Theorem \ref{thm:propvol}, we calculate this integral  by discriminating between points $x$ close to $x_0$ and points $x$ far from $x_0$. Indeed, the volume $\vol^{(M)}(\bo^{(M)}(x, d^{(M)}(x_0,x)))$ and the volume element $\dvol^{(M_s)}$ can be expanded when the distance between $x$ and $x_0$ tends to $0$, using Lemmas \ref{lem:volunif} and \ref{lem:jacobchtsph} while the contribution of the points $x$ far from $x_0$ is expected to decrease exponentially fast.

\noindent \textit{Step 1: decomposition of $\E[\vol^{(M_s)}(\Cl\cap M_s)]$ into two integrals.} Let $r_\lambda= \lambda^{-\frac{n+1}{n(n+2)}}$ and let us write the integral in \eqref{eq:fubSection} as  
\begin{equation}\label{eq:decoupIntVolSection}
\E[ \vol^{(M_s)}(\Cl\cap M_s)] = I_\lambda + \tilde{I}_\lambda 
\end{equation}
where 
\begin{equation}\label{eq:defIlVolSection} I_\lambda = \int_{\bo^{(M)}(x_0,r_\lambda) \cap M_s} e^{-\lambda \vol^{(M)}(\bo^{(M)}(x, d^{(M)}(x_0,x)))} \dvol^{(M_s)}(x) \end{equation}
and 
\begin{equation}\label{eq:defItilVolSection} \tilde{I}_\lambda = \int_{M_s \setminus \bo^{(M)}(x_0,r_\lambda)} e^{-\lambda \vol^{(M)}(\bo^{(M)}(x, d^{(M)}(x_0,x)))} \dvol^{(M_s)}(x) .\end{equation}
\noindent\textit{Step 2: $\tilde{I}_\lambda$ is negligible.} We prove that $\tilde{I}_\lambda$ given at \eqref{eq:defItilVolSection} is negligible in front of $\frac{1}{\lambda^{\frac{s+2}{n}}}$. Observe that $M_s$ is compact since it is included in $M$, thus the same arguments as in Step 2 of Section \ref{sec:vol} imply that 
\[ \tilde{I}_\lambda \le \vol^{(M_s)}(M_s)  e^{-c\lambda^{1-\frac{n+1}{n+2}}} ,\]
hence 
\begin{equation}\label{eq:ItilNegligVolSection}
 \tilde{I}_\lambda = o\left( \frac{1}{\lambda^{\frac{s+2}{n}}}\right). 
\end{equation}

\noindent\textit{Step 3 : estimate of $I_\lambda$}. We prove now that when $\lambda\to\infty$
\begin{equation}\label{eq:devIlfinalVolSec}
I_\lambda =  v_{n,s}\frac{1}{\lambda^{\frac sn}}
+\left(u_{n,s}{\Scal_{x_0}^{(M)}}- w_{n,s}{\Scal_{x_0}^{(M_s)}} \right)\frac{1}{\lambda^{\frac{s+2}n}} + o\left( \frac{1}{\lambda^{\frac{s+2}{n}}}\right),
\end{equation} 
where the constants $u_{n,s}, v_{n,s}$ and $w_{n,s}$ are given in the statement of Theorem \ref{volSec}. To this end, we apply in $I_\lambda$, the spherical change of variables given by 
\begin{equation}\label{eq:changeSphericalVolSec}
 \varphi_{x_0,V_s} : \left\{\begin{array}{rll}(0,R_{\mbox{\tiny inj}})\times \{u\in V_s: \|u\|_{x_0}=1\}&\longrightarrow& M_s\\(r,u)&\longmapsto &\exp_{x_0}(ru)\end{array}\right.
\end{equation}
which is the restriction to $V_s$ of $\varphi_{x_0}$ given by \eqref{eq:changeSpherical}. Let us denote by $\mathcal{J}_{x_0,s}$ the Jacobian determinant of this change of variables. Thus, we get

\begin{equation}\label{eq:IntegralApChangeVolSection}
I_\lambda = \int_{0}^{r_\lambda} \int e^{-\lambda \vol^{(M)}(\bo^{(M)}(x,r))} |\mathcal{J}_{x_0,s}(r,u)| \dd r \dvol^{(\sph^{s-1})}(u) .
\end{equation} 

We now need to replace $\vol^{(M)}(\bo^{(M)}(x,r))$ and $\mathcal{J}_{x_0,s}(r,u)$ by suitable approximations. Let $\varepsilon>0$. 
The same arguments as in the proof of Lemma \ref{lem:jacobchtsph} imply that $\mathcal{J}_{x_0,s}$ satisfies, when $r\to 0$, 
\begin{equation}\label{eq:devJacobienSphericalSection}
 \mathcal{J}_{x_0,s}(r,u) = r^{s-1} - \frac{ \Ric^{(M_s)}_{x_0}(u)}{6}r^{s+1}  +o(r^{s+1}), 
\end{equation}
where $\Ric^{(M_s)}_{x_0}(u)$ is the Ricci curvature of the submanifold $M_s$ at $u$, that is 
\[  \Ric^{(M_s)}_{x_0}(u) = \sum_{i=2}^{s} K_{x_0}^{(M)}(u,u_i), \]
and $\{ u_1=u,u_2,\dots,u_s \}$ is an orthonormal basis of $V_s$.
It follows that, for $\lambda$ large enough, $0<r <r_\lambda$, $u\in V_s$ with norm $1$, 
 \begin{equation} \label{eq:inegJacobienVolSection}
      r^{s-1} - \frac{ \Ric^{(M_s)}_{x_0}(u)}{6}r^{s+1} -\varepsilon r^{s+1}   \le  \mathcal{J}_{x_0,s}(r,u) \le   r^{s-1} - \frac{ \Ric^{(M_s)}_{x_0}(u)}{6}r^{s+1} +\varepsilon r^{s+1}. 
\end{equation} 
Moreover, the inequalities \eqref{eq:inegVolumeC} still hold, that is 
\begin{align}
     e^{-\lambda \kappa_n r^n} \left[ 1+ \left(\frac{\kappa_n \Scal_{x_0}^{(M)}}{6(n+2)}  -\varepsilon\right)\lambda r^{n+2} \right]&\le e^{-\lambda \vol^{(M)}(\bo^{(M)}(x,r))} 
\le e^{-\lambda \kappa_n r^n} \left[1+\left(\frac{\kappa_n \Scal_{x_0}^{(M)}}{6(n+2)}  +2\varepsilon\right)\lambda r^{n+2}\right].
\label{eq:inegVolumeBouleSection} 
\end{align}
Inserting \eqref{eq:inegJacobienVolSection} and \eqref{eq:inegVolumeBouleSection} into \eqref{eq:IntegralApChangeVolSection}, we obtain 
\begin{equation}\label{eq:encadreIlVolSection}
 I_{\lambda,-} \le I_\lambda \le I_{\lambda,+} , 
\end{equation}
where 
\begin{align}
I_{\lambda,-} &=\int_{0}^{\lambda^{-\frac{n+1}{n(n+2)}}} \int
e^{-\lambda \kappa_n r^n} \left[ 1+ \left(\frac{\kappa_n \Scal_{x_0}^{(M)}}{6(n+2)}  -\varepsilon\right)\lambda r^{n+2} - \left(\frac{ \Ric^{(M_s)}_{x_0}(u)}{6} +\varepsilon\right) r^{2}\right.\nonumber \\
&\hspace*{4cm}\left.-\left(\frac{\kappa_n \Scal_{x_0}^{(M)}}{6(n+2)}  -\varepsilon\right)\left(\frac{ \Ric^{(M_s)}_{x_0}(u)}{6} +\varepsilon\right)\lambda r^{n+4}\right]  \dvol^{(\sph^{s-1})}(u)r^{s-1}\dd r \notag
\end{align}
and 
\begin{align}
I_{\lambda,+} &=\int_{0}^{\lambda^{-\frac{n+1}{n(n+2)}}} \int  
e^{-\lambda \kappa_n r^n} \left[ 1+ \left(\frac{\kappa_n \Scal_{x_0}^{(M)}}{6(n+2)}  +2\varepsilon\right)\lambda r^{n+2} - \left(\frac{ \Ric^{(M_s)}_{x_0}(u)}{6} -\varepsilon\right) r^{2}\right.\nonumber \\
&\hspace*{4cm}\left.-\left(\frac{\kappa_n \Scal_{x_0}^{(M)}}{6(n+2)}  +2\varepsilon\right)\left(\frac{ \Ric^{(M_s)}_{x_0}(u)}{6} -\varepsilon\right)\lambda r^{n+4}\right]  \dvol^{(\sph^{n-1})}(u)r^{s-1}\dd r. \notag
\end{align}

In order to derive upper bounds for $I_{\lambda,+}$, we proceed analogously to the proof of Theorem \ref{thm:propvol}, that is 
\begin{itemize}
\item[-]  we start by applying the change of variables $y=\lambda\kappa_n r^n$, taking $y$ from $0$ to $\infty$,
\item[-]  we use Assumption \mbox{($\mbox{A}_1$)} to show that the integral of the last term of $I_{\lambda,+}$ is bounded by $\lambda^{-\frac4n -\frac sn+1}$ up to a multiplicative constant,
\item[-]  we calculate the integral of the remaining terms over $u$ first, using \eqref{eq:scalRicci} applied to the manifold $M_s$ and then over $y$.
\end{itemize}
Thus, for any $\varepsilon>0$, we obtain that, for $\lambda$ large enough 
\begin{align}
I_{\lambda,+}& \le v_{n,s}\frac{1}{\lambda^{\frac sn}}
+\left(u_{n,s}{\Scal_{x_0}^{(M)}}- w_{n,s}{\Scal_{x_0}^{(M_s)}} \right)\frac{1}{\lambda^{\frac{s+2}n}} +\frac{\varepsilon}{\lambda^{\frac{s+2}n}} \label{eq:I+Volsec}
\intertext{and similarly}
I_{\lambda,-} &\ge v_{n,s}\frac{1}{\lambda^{\frac sn}}
+\left(u_{n,s}{\Scal_{x_0}^{(M)}}- w_{n,s}{\Scal_{x_0}^{(M_s)}} \right)\frac{1}{\lambda^{\frac{s+2}n}} -\frac{\varepsilon}{\lambda^{\frac{s+2}n}} \label{eq:I-Volsec},
\end{align}
Now, \eqref{eq:I+Volsec} and \eqref{eq:I-Volsec} implies \eqref{eq:devIlfinalVolSec} and combining with \eqref{eq:ItilNegligVolSection} and \eqref{eq:decoupIntVolSection}, we obtain Theorem  \ref{volSec} (i). 
\subsection{Proof of Theorem \ref{volSec}, (ii)}
We now derive the explicit formulas in the particular cases $M=\sph_k^{n}$ and $M={\mathcal H}_k^n$. Again, the proof follows from the fact that the Jacobian determinant of the change of variables given by \eqref{eq:changeSphericalVolSec} and the volume of geodesic balls have exact expressions.

\paragraph{Case $M=\sph_k^{n}$.} In \eqref{eq:fubSection},
let us apply the spherical change of variables given by \eqref{eq:changeSphericalVolSec} with $0\le r \le \frac{\pi}k$. 
For $M=\sph_k^n$, the Jacobian determinant of this change of variables satisfies
\begin{equation}
  \label{eq:sphSectionJac}
 \J_{x_0,s}(r,u) = \left(\frac{\sin(kr)}k \right)^{s-1}.
\end{equation}
Moreover recall that, 
\begin{equation}
  \label{eq:boulesphSec}
 \vol^{(\sph^n_k)}(\bo^{(\sph^n_k)}(x,r)) = \sigma_{n-1}\int_{t=0}^{r} \left(\frac{\sin(kt)}k \right)^{n-1} \dd t  
\end{equation}
Combining \eqref{eq:fubSection} for $M=\sph_k^n$, \eqref{eq:sphSectionJac} and \eqref{eq:boulesphSec},  we obtain
\begin{align*}
 \E[\vol^{(M_s)}({\mathcal C}_{x_0,\lambda}^{(\sph^n_k)}\cap   M_s)] &= \int_{r=0}^{\frac{\pi}{k}} \int_{u\in\sph^{s-1}}e^{-\lambda\sigma_{n-1}\int_{t=0}^{r} \left(\frac{\sin(kt)}k \right)^{n-1} \dd t }  \left(\frac{\sin(kr)}k \right)^{s-1} \dd r\dvol^{(\sph^{s-1})}(u) \\
&= \vol^{(\sph^{s-1})}(\sph^{s-1})\int_{r=0}^{\frac{\pi}{k}}e^{-\lambda\sigma_{n-1}\int_{t=0}^{r} \left(\frac{\sin(kr)}k \right)^{n-1} \dd t }  \left(\frac{\sin(kr)}k \right)^{s-1} \dd r, 
\end{align*}
as claimed. 

\paragraph{Case $M=\mathcal{H}_k^{n}$.} We proceed in exactly the same way as in the case of the sphere.
In \eqref{eq:fubSection},
we  apply the change of variables given by  \eqref{eq:changeSphericalVolSec}, that the Jacobian satisfies, when $M= \mathcal{H}_k^n$, $r\ge 0$
\begin{equation}
  \label{eq:HypSectionJac}
 \J_{x_0,s}(r,u) = \left(\frac{\sinh(kr)}k \right)^{s-1}  .
\end{equation}
Moreover recall that, 
\begin{equation}
  \label{eq:bouleHypSec}
 \vol^{(\mathcal{H}^n_k)}(\bo^{(\mathcal{H}^n_k)}(x,r)) = \sigma_{n-1}\int_{t=0}^{r} \left(\frac{\sinh(kt)}k \right)^{n-1} \dd t  
\end{equation}
Combining \eqref{eq:fubSection} for $M=\mathcal{H}_k^n$, \eqref{eq:HypSectionJac} and \eqref{eq:bouleHypSec},  we obtain
\begin{align*}
 \E[\vol^{(M_s)}({\mathcal C}_{x_0,\lambda}^{(\mathcal{H}^n_k)}\cap   M_s)] &= \int_{r\ge 0} \int_{u\in\sph^{s-1}}e^{-\lambda\sigma_{n-1}\int_{t=0}^{r} \left(\frac{\sinh(kt)}k \right)^{n-1} dt }  \left(\frac{\sinh(kr)}k \right)^{s-1} \dd r\dvol^{(\sph^{s-1})}(u) \\
&= \vol^{(\sph^{s-1})}(\sph^{s-1})\int_{r\ge 0}e^{-\lambda\sigma_{n-1}\int_{t=0}^{r} \left(\frac{\sinh(kr)}k \right)^{n-1} dt }  \left(\frac{\sinh(kr)}k \right)^{s-1} \dd r.
\end{align*}
This concludes the proof of Theorem \ref{volSec} (ii). 

\subsection{Proof of Theorem \ref{NSec}, (i)}

Observe that a vertex of ${\mathcal C}_{x_0,\lambda}^{(M)}\cap   M_s$ is the intersection of $M_s$ and a $(n-s)$-face of the cell of $\Cl$. Such a $(n-s)$-face is itself the intersection of $\Cl$ and the cells of $s$ points of the process $\Pl$. As in the proof of Theorem \ref{THM}, the expansion of the mean number of vertices of ${\mathcal C}_{x_0,\lambda}^{(M)}\cap   M_s$ relies on the rewriting of the volume element $\dvol^{(M)}(x_1)\dots \dvol^{(M)}(x_s)$ in the spirit of Theorem \ref{Jacob}. More precisely, let us define the application $\Phi_{x_0}^s :(r,u,u_1,\dots u_s) \mapsto (x_1,\dots,x_s)$ by the identity
\begin{equation*}
 x_i= \exp_{|\exp_{x_0}(ru)}(ru_i), 
\end{equation*}
where $r>0$, $u$ is a unit vector of  $V_s$ and $u_1,\dots, u_s$ are unit vectors of $T_{\exp_{x_0}(ru)}M$. When $r\to 0$, we can show that the Jacobian determinant $\tilde{\J}_{x_0}^s$ of $\Phi_{x_0}^s$, satisfies the expansion 
\begin{equation}\label{eq:JacobienSection}
\tilde{\J}_{x_0}^s(r,u,u_1,\dots,u_s) = s! \Delta_s(u,u_1,\dots,u_s)( r^{sn-1} -\frac{L_{x_0}^s(u,u_1,\dots,u_s)}{6}  r^{sn+1} +o(r^{sn+1}))
\end{equation}
where $L_{x_0}^s(u,u_1,\dots,u_s)=\Ric_{x_0}^{(M_s)}(u) + \sum_{i=1}^{s} \Ric_{x_0}^{(M)}(u_i)$ and $\Delta_s(u,u_1,\dots,u_s)$ is the volume of the simplex spanned by $-u$ and the projection of $u_1,\dots,u_s$ on $V_s$. We omit the proof of \eqref{eq:JacobienSection} since it is very similar to the proof of Theorem \ref{Jacob}. Note that for $r$ small enough, the Jacobian $\J_{x_0}^s$ is different from zero almost everywhere. Then, from the inverse function theorem, the function $\Phi_{x_0}^s$ defines a change of variables. 
Consequently, a reasoning along the lines of the beginning of Section \ref{sec:vert} implies that without loss of generality, for almost all
 $(x_1,\dots,x_s) \in M$, there is at most one circumscribed ball of $x_0,\dots x_s$ centered in $M_s$, denoted by $\bo^{s}(x_0,\dots,x_s)$. 

The expectation of the number of vertices can be written as 
\begin{equation*}
\E[N(\Cl\cap M_s)]= \E[ \sum_{\{x_1,\dots,x_s\} \subset \Pl} \mathds{1}_{\{\bo^{s}(x_0,x_1,\dots,x_s)\text{ exists and }\bo^{s}(x_0,x_1,\dots,x_s)\cap \Pl =\emptyset \}}].
\end{equation*}
Then applying the Mecke-Slivnyak formula, we get that
 \begin{align}\label{eq:sliNsec}
\E[N(\Cl\cap M_s)]&= \frac{\lambda^s}{s!} \int_{\tilde{M}^s} e^{-\lambda \vol^{(M)}(\bo^{s}(x_0,x_1,\dots,x_s))} \dvol^{(M)}(x_1) \dots \dvol^{(M)}(x_s). \end{align}
where 
\[ \widetilde{M}^s= \{ (x_1,\dots,x_s)\in M^s, \bo^{s}(x_0,x_1,\dots,x_s) \text{ exists }\} .\]

%
%

Let $r_\lambda = \lambda^{-\frac{n+1}{n(n+2)}}$ and let us define $M_\lambda^{(s)}$ by 
\[ M_\lambda^{(s)} = \{ (x_1,\dots,x_s)\in \widetilde{M}^s, \text{ the radius of }\bo^{s}(x_0,\dots,x_s) \text{ is smaller than } r_\lambda\} .\]
The expansion of $\E[N(\Cl\cap M_s)]$ is derived by following the same strategy as in the proof of Theorem \ref{THM}, that is 

\begin{itemize}
\item[-]  we decompose the integral in \eqref{eq:sliNsec} and isolate the contribution of points $(x_1,\dots,x_s) \in M_\lambda^{(s)}$,
\item[-]  we show that the integral over $M^s\setminus M_\lambda^{(s)}$ is negligible in front of $\lambda^{-\frac{2}n}$, 
\item[-]  we apply the change of variables given by $\Phi_{x_0}^s$ in the integral over $M_\lambda^{(s)}$ and use the approximations of the Jacobian and of the volume of a ball provided by \eqref{eq:JacobienSection} and Lemma \ref{lem:volunif} respectively,
\item[-]  we apply the change of variables $y=\lambda\kappa_n r^n$, taking $y$ from $0$ to $\infty$, 
\item[-]  we deduce the expansion of $\E[N(\Cl\cap M_s)]$ from the integration over $u,u_1,\dots,u_s$ and then over $y$.  
\end{itemize}

\subsection{Proof of Theorem \ref{NSec}, (ii)}
We now derive exact formulas for $\E[N(\Cl \cap M_s)]$, in the particular cases $M=\sph_k^{n}$ and $M=\mathcal{H}_{k}^n$. Again, these results come from the exact expressions of the Jacobian determinant of the change of variables given by $\Phi_{x_0}^s$ and of the volume of geodesic balls in these two cases. 

\paragraph{Case $M=\sph_k^{n}$.} 
 Observe that for each  $(x_1,\dots,x_s) \in (\sph_k^{n})^s$, there are exactly two balls in $\sph^{n}_k$, with center in $M_s$ and containing $x_0,\dots,x_s$ on their boundaries: the smallest one denoted by $\bo^{(s)}_1(x_0,\dots,x_s)$, with radius $r \le \frac{\pi}{2k}$ and $\bo^{(s)}_2(x_0,\dots,x_s)= \sph_k^{n}\setminus \bo^{(s)}_1(x_0,\dots,x_s)$. Then, applying the Mecke-Slivnyak formula, we obtain 
\begin{align}
\E[N({\mathcal C}_{x_0,\lambda}^{(\sph_k^n)} \cap M_s) &= \E[ \sum_{\{x_1,\dots,x_s\} \subset \Pl} \mathds{1}_{\{ \bo^{(s)}_1(x_0,\dots,x_s) \cap \Pl =\emptyset \}} + \mathds{1}_{\{ \bo^{(s)}_2(x_0,\dots,x_s) \cap \Pl =\emptyset \}} ] \notag \\
&= \int_{(\sph_k^{n})^s} \left( e^{-\lambda \vol^{(\sph_k^n)}(\bo^{(s)}_1(x_0,\dots,x_s))} + e^{-\lambda \vol^{(\sph_k^n)}(\bo^{(s)}_2(x_0,\dots,x_s))} \right) \dvol^{(\sph_k^n)}(x_1) \dots \dvol^{(\sph_k^n)}(x_s) \label{eq:MSsphNsection}
\end{align} 
Let us make the change of variables given by $\Phi_{x_0}^s$ in \eqref{eq:MSsphNsection}, with $r\le \frac{\pi}{2k}$, in the case $M=\sph_k^{n}$. Then, the Jacobian determinant $\tilde{\J}_{x_0}^s$ is given by 
\begin{equation}\label{eq:JacobNsecSph}
\tilde{\J}_{x_0}^s(r,u,u_1,\dots,u_s) = s! \Delta_s(u,u_1,\dots,u_s) \left( \frac{\sin(kr)}k \right)^{ns-1}.
\end{equation}
In these new coordinates, \eqref{eq:volSn} and \eqref{eq:volSnbis} are still valid for the balls $\bo^{(s)}_1(x_0,\dots,x_s)$ and $\bo^{(s)}_2(x_0,\dots,x_s)$. 
The expression of $\E[N({\mathcal C}_{x_0,\lambda}^{(\sph_k^n)} \cap M_s)$ then follows by inserting \eqref{eq:volSn}, \eqref{eq:volSnbis} and \eqref{eq:JacobNsecSph} into \eqref{eq:MSsphNsection} and integrating over $u,u_1,\dots,u_s$ as in the proof of Theorem \ref{THM} (ii). 

\paragraph{Case $M=\mathcal{H}_k^{n}$.} Contrary to the case of the sphere, for each  $(x_1,\dots,x_s) \in (\mathcal{H}_k^{n})^s$, there exists exactly one ball, denoted by $\bo^{(s)}(x_0,\dots,x_s)$ in $\mathcal{H}_k^n$, centered in $M_s$ and containing $x_0,\dots,x_s$  on its boundary. Then applying the Mecke-Slivnyak 
\begin{align}
\E[N({\mathcal C}_{x_0,\lambda}^{(\mathcal{H}_k^n)} \cap M_s) ]&= \E[ \sum_{\{x_1,\dots,x_s\} \subset \Pl} \mathds{1}_{\{ \bo^{(s)}(x_0,\dots,x_s) \cap \Pl =\emptyset \}}  ] \notag \\
&= \int_{(\mathcal{H}_k^{n})^s}  e^{-\lambda \vol^{(\mathcal{H}_k^{n})}(\bo^{(s)}(x_0,\dots,x_s))} \dvol^{(\mathcal{H}_k^{n})}(x_1) \dots \dvol^{(\mathcal{H}_k^{n})}(x_s) \label{eq:MShypNsection}.
\end{align} 
We make the change of variables given by $\Phi_{x_0}^s$ in \eqref{eq:MShypNsection}. 
The Jacobian of the change of variables given by $\Phi_{x_0}^s$ when $M=\mathcal{H}_{k}^n$ satisfies 
\begin{equation}\label{eq:JacobNsecHyp}
\tilde{\J}_{x_0}^s(r,u,u_1,\dots,u_s) = s! \Delta_s(u,u_1,\dots,u_s) \left( \frac{\sinh(kr)}k \right)^{ns-1}.
\end{equation}
Moreover in the new coordinates, the expression for the volume of the ball provided by \eqref{eq:volHn} is still valid for $\bo^{(s)}(x_0,\dots,x_s)$. Thus inserting  \eqref{eq:JacobNsecHyp} and  \eqref{eq:volHn} into \eqref{eq:MShypNsection}, and then integrating over $u,u_1,\dots,u_s$, we obtain the expected formula for $\E[N({\mathcal C}_{x_0,\lambda}^{(\mathcal{H}_k^n)} \cap M_s) ]$ . 
This concludes the proof of Theorem \ref{NSec}, (ii).

\section{Probabilistic proof of the Gauss-Bonnet theorem}\label{sec:GB}
In this section, we consider a compact surface $S$ and we denote by $K:S\to \R$ the Gaussian curvature of $S$. 
The classical Gauss-Bonnet theorem states that the Euler characteristic of $S$, denoted by $\chi(S)$, satisfies the relation
\begin{equation}\label{eq:GB}
\chi(S) = \frac{1}{2\pi} \int_S K(x) \dvol^{(S)}(x).
\end{equation}
We start by showing a slight reinforcement of Theorem \ref{THM} which is the uniformity of the two-term expansion of $\E[N({\mathcal C}_{x_0,\lambda}^{(\R^n)})]$ when $M$ is a compact Riemannian manifold.
\begin{prop}\label{prop:thmprincipalunif}
When $M$ is a compact Riemannian manifold, we get
$$\sup_{x_0\in M}\lambda^{\frac 2n}\left(\E[\Nl] - e_n + d_n\frac{\Scal_{x_0}^{(M)}}{\lambda^{\frac 2n}}\right)\underset{\lambda\to\infty}{\to} 0.$$
\end{prop}
\begin{proof}
Looking closely at the proof of Theorem \ref{THM}, we observe that the approximation of $\E[\Nl]$ comes from two consecutive estimates inside the integral at \eqref{eq:INouvCoor}: the expansion of the Jacobian provided by Theorem \ref{Jacob} and the expansion of the volume of small balls given at \eqref{vol}. As soon as these two estimates are showed to be uniform with respect to $x_0$, the uniformity of the expansion in \eqref{th1} follows. 

The uniformity of the Jacobian and volume approximations are guaranteed by Proposition \ref{prop:jacobunif} and Lemma \ref{lem:volunif} (ii) respectively. 
This completes to proof of Proposion \ref{prop:thmprincipalunif}.
\end{proof}
We can now proceed with the proof of \eqref{eq:GB}. 
Let us denote by $F$ (resp. $E$ and $V$), the total number of Voronoi cells in $S$ (resp. the total number of edges and vertices). Euler's relation applied to the Voronoi graph states that 
\begin{equation}\label{eq:euler}
\chi(S) = F-E +V.
\end{equation} 
Since the Voronoi tessellation is a normal tessellation, see \cite[p. 43]{Mol94}, each vertex is contained in three cells and each edge is contained in two cells, thus 
\begin{equation}\label{eq:normal}
2E=3V. 
\end{equation}
Inserting \eqref{eq:normal} into \eqref{eq:euler}, we obtain
\begin{equation}\label{eq:neweuler}
\chi(S) = F- \frac 12 V. 
\end{equation}
Taking now the expectation of \eqref{eq:neweuler}, we deduce that
\begin{equation}\label{eq:espeuler}
\chi(S) = \E[F] - \frac 12 \E[V]. 
\end{equation}
It remains to calculate both expectations in the right-hand side of \eqref{eq:espeuler}. In particular, the number $F$ of Voronoi cells is the number of points of $\Pl$ in $S$, hence 
\begin{equation}\label{eq:espF}
 \E[ F]= \E[ \#( \Pl \cap S)] = \lambda \vol^{(S)}(S).  
\end{equation}
Now, since each vertex is in three cells, we have
 \begin{equation}\label{eq:espV}
\E[V] = \frac 13 \E[ \sum_{x\in \Pl} N( \mathcal{C}^{(S)}(x,\Pl))], 
 \end{equation}
An application of the Mecke-Slivnyak formula in \eqref{eq:espV} provides the identity 
\begin{equation}\label{eq:MSV}
\E[V] = \frac{\lambda}3 \int_S \E[ N( \mathcal{C}^{(S)}(x, \Pl \cup \{ x\} ))] \dvol^{(S)}(x).
\end{equation}
Applying now Proposition \ref{prop:thmprincipalunif} to the compact surface $S$ and using the fact that $\Scal_{x}^{(S)}=2 K(x)$ , we get that
\begin{equation}\label{eq:thm2}
\sup_{x\in S}\lambda\left(\E[N(\mathcal{C}^{(S)}(x, \Pl \cup \{ x\} ))] - 6 +\frac{3K(x)}{\pi \lambda}\right)\underset{\lambda\to\infty}{\to} 0.
\end{equation}
Let $\varepsilon>0$. The convergence \eqref{eq:thm2} implies that there exists $\lambda_0$ such that for all $\lambda \ge \lambda_0$ and every $x\in S$,
\begin{equation}\label{eq:espNeps}
\bigg|\E[N(\mathcal{C}^{(S)}(x, \Pl \cup \{ x\} ))] - 6 +\frac{3K(x)}{\pi \lambda}\bigg|\le \frac{\varepsilon}{\lambda}.
\end{equation}
Combining \eqref{eq:espNeps} with \eqref{eq:MSV}, we get that
\begin{align}
\bigg| \E[V]-\frac{\lambda}3 \int_{S} \left(6 - \frac{3K(x)}{\lambda \pi}\right)\dvol^{(S)}(x)\bigg|& \le  \frac{\varepsilon}{3}\int_{S} \dvol^{(S)}(x)
\end{align}
which means that
\begin{align*} 
\bigg|\frac12 \E[V] - \lambda \vol^{(S)}(S) +\frac{1}{2\pi} \int_{S}K(x)\dvol^{(S)}(x)\bigg|&\le \frac{\mbox{vol}^{(S)}(S)}{6}\varepsilon.
\end{align*}
This together with \eqref{eq:espeuler} and \eqref{eq:espF}, implies that for every $\varepsilon>0$,
\begin{equation*} 
\bigg|\chi(S)- \frac{1}{2\pi} \int_{S} K(x) \dvol^{(S)}(x)\bigg|\le \frac{\mbox{vol}^{(S)}(S)}{6}\varepsilon.
\end{equation*}
which concludes the proof of \eqref{eq:GB}. 
\section{Appendix}
This section contains the proofs of Lemmas \ref{lem:jacobchtsph}, \ref{lem:volunif}, \ref{lem:borneinfvolboule}, \ref{lem:secondmoment}, \ref{lem:stabilization}.\\

\noindent{\it Proof of Lemma \ref{lem:jacobchtsph}.}
\noindent (i)  Let $\mathcal{V} = \{ u=u_1,u_2,\dots, u_{n} \}$ be an orthonormal basis of $T_{x_0}M$. By parallel transport of $\mathcal{V}$ along  $\gamma(s)= \exp_{x_0}(su)$, we get a basis $\mathcal{V}(s) = \{ u(s)=u_1(s), \dots, u_{n}(s) \}$ of $T_{\gamma(s)}M$. 
	By definition of $\gamma$, the derivative of $\exp_{x_0}(ru)$ with respect to $r$ is $u(r)$, parallel transport of $u$ along $\gamma$ so that the first column of the Jacobian matrix of $\varphi_{x_0}$ is the vector $(1,0,\dots,0)^{T}$. 
	 For $i=2,\dots,n$, applying Lemma \ref{lem:exodocarmo} with $v(s)= u+su_i$, it follows that the derivative of $\exp_{x_0}(tu)$ with respect to $u$ in the direction $u_i$ is 
    \begin{equation} \label{eq:deriveUi}\frac{\partial \exp_{x_0}(ru)}{\partial u_i} = J_i(r) \end{equation}
		where $J_i$ is the Jacobi field along $\gamma$ such that $J_i(0)=0$ et $J_i'(0)=u_i$.
		Such a Jacobi field is a normal vector field, that is $\langle u_1(s),J_i(s) \rangle_{\gamma(s)} =0$ for all $s$. It follows from \eqref{eq:deriveUi} that 
	the Jacobian determinant of $h$ is
		\begin{equation} \label{matriceJac}  \J_{x_0}^{(M)}(r,u) = \det \begin{pmatrix} 
		1& 0&\cdots & 0 \\
		0& ~&~&~ \\
		\vdots & ~ &A(r,u) &~\\
		0 & ~&~&~
		     \end{pmatrix} \end{equation}
				where $A(r,u)$ is a matrix of size $(n-1)\times (n-1)$ with entries $A_{i,j}(r,u) = \langle u_{i+1}(r), J_{j+1}(r)\rangle_{\gamma(r)}$. Even though the dependency is not visible, we emphasize here the fact that $A_{i,j}$ does depend on $u$ through the construction of the orthonormal basis ${\mathcal V}$, its parallel transported ${\mathcal V}(r)$  and the Jacobi field $J_{j+1}$. In order to get an asymptotic expansion of the Jacobian determinant, we first develop the coefficients of the matrix $A$ using Taylor's formula for fixed $u$. Let us determine the successive derivatives of $A_{i.j}$ with respect to $r$: the vector $u_i(r)$ being a parallel transport, its derivative with respect to $r$ is $0$ so that for every $m\ge 1$ and every couple $(u,r)$,
\begin{equation}\label{eq:partderivA}
  \frac{\partial^m A_{i,j}}{\partial r^m}(r,u)=\langle u_{i+1}(r), J_{j+1}^{(m)}(r) \rangle_{x_0}.
\end{equation}
where $J_{j+1}^{(m)}$ is the $m$-th covariant derivative of $J_{j+1}$ with respect to $\gamma'$.
In particular, \eqref{eq:partderivA} applied to $m=0$ and $m=1$ in the particular case $r=0$ provides
		\begin{align}
		A_{i,j}(0,u)=\langle u_{i+1}, J_{j+1}(0) \rangle_{x_0}&=0 ,\label{eq:derive0}
\end{align}
and
\begin{align}
		\frac{\partial A_{i,j}}{\partial r}(0,u)=\langle u_{i+1}, J_{j+1}'(0) \rangle_{x_0} &= \delta_{i,j}.\label{eq:derive1}
		\end{align}
	Moreover, since $J$ satisfies the Jacobi equation \eqref{eqJacobi}, we get
		\begin{equation} \label{eq:derive2} \frac{\partial^2 A_{i,j}}{\partial r^2}(0,u)=\langle u_{i+1}, J_{j+1}''(0) \rangle_{x_0} = - \langle u_{i+1}, \mathcal{R}_{x_0}^{(M)}(J_{j+1}(0),u)u \rangle_{x_0}= 0. \end{equation}
The calculation of the third derivative of $J_{j+1}$ is more delicate but thanks to the identity $J_{j+1}(0)=0$ and to the trick explained in \cite[p.115]{Do92}, we can show that $J_{j+1}'''(0) =-\mathcal{R}_{x_0}^{(M)}(J'_{j+1}(0),u)u$. Inserting that result into \eqref{eq:partderivA} with $m=3$, we obtain
		 \begin{equation}\label{eq:derive3}
\frac{\partial^3 A_{i,j}}{\partial r^3}(0,u)=
- \langle u_{i+1}, \mathcal{R}_{{x_0}}^{(M)}(J'_{j+1}(0),u)u \rangle_{x_0} = - \langle u_{i+1}, \mathcal{R}_{{x_0}}^{(M)}(u_{j+1},u)u \rangle_{x_0}. \end{equation}
	Observe that in the particular case $i=j$, the definition of the sectional curvature given at \eqref{eq:defseccurv} implies that 
        \begin{equation}
          \label{eq:3rdderivi=j}
          \frac{\partial^3 A_{i,i}}{\partial r^3}(0,u)=-K_{x_0}^{(M)}(u,u_{i+1}).
        \end{equation}
We may now insert  \eqref{eq:derive0}, \eqref{eq:derive1}, \eqref{eq:derive2}, \eqref{eq:derive3} and \eqref{eq:3rdderivi=j} into Taylor's formula applied to the function $A_{i,j}$ as a function of $r$ for fixed $u$, at the second order for $i\ne j$ and at the third order for $i=j$. We obtain the two expansions
		\begin{align}
		A_{i,j}(r,u) &= \reste_{i,j}(r,u) r^2, i\neq j, \label{AnonDiag}\\
		A_{i,i}(r,u) &= r - \frac{ K^{(M)}_{x_0}(u,u_{i+1})}6 r^3 +\reste_{i,i}(r,u)r^3, \label{Adiag}
		\end{align}
		
	where $\reste_{i,j}(r,u)$ tends to $0$ when $r$ tends to $0$ for fixed $u$. 

Next, we prove that the terms $\reste_{i,j}(r,u)$ tend to $0$ independently of $u$. To do so, we aim at applying Taylor's inequality, which only requires to show uniform bounds on the third and fourth derivatives with respect to $r$ of $A_{i,j}(r,u)$, $i\ne j$, and $A_{i,i}(r,u)$ respectively. Using both the Jacobi equation \eqref{eqJacobi} and the fact that $u_{i+1}'(r) =u_{j+1}'(r) = u'(r)=0$, we get
	\begin{equation*}
	   \frac{\partial^2A_{i,j}}{\partial r^2}(r,u) =\langle u_{i+1}(r), \mathcal{R}_{\gamma(r)}^{(M)}(J_{j+1}(r),u(r))u(r) \rangle_{\gamma(r)}. 
	\end{equation*}
This implies that
	\begin{align*}
	  \frac{\partial^3A_{i,j}}{\partial r^3}(r,u) &=\langle u_{i+1}(r), \frac{\partial}{\partial r} \mathcal{R}_{\gamma(r)}^{(M)}(J_{j+1}(r),u(r))u(r) \rangle_{\gamma(r)},
\end{align*}
and
\begin{align*}
	  \frac{\partial^4A_{i,j}}{\partial r^4}(r,u) &=\langle u_{i+1}(r), \frac{\partial^2}{\partial r^2} \mathcal{R}_{\gamma(r)}^{(M)}(J_{i+1}(r),u(r))u(r) \rangle_{\gamma(r)}.
	\end{align*}

Since the curvature tensor $\mathcal{R}_{x}^{(M)}$ is a ${\mathcal C}^{\infty}$ function of $x\in M$ and of its two entries, both functions $\displaystyle\frac{\partial^3A_{i,j}}{\partial r^3}$ and $\displaystyle\frac{\partial^4A_{i,j}}{\partial r^4}$ are continuous on the compact set
$[0,R_{\mbox{\tiny inj}}]\times \sph^{n-1}$, where we use a slight abuse of notation by denoting by $\sph^{n-1}$ the unit sphere of $T_{x_0}M$. Consequently, there exists positive constants $c_{i,j}$ such that  
	 	\begin{align}
	  \bigg|\frac{\partial^3A_{i,j}}{\partial r^3}(r,u)\bigg| &\le c_{i,j}, i\neq j \label{eq:unifd3}\\
	  \bigg|\frac{\partial^4A_{i,j}}{\partial r^4}(r,u)\bigg| &\le c_{i,i}. \label{eq:unifd4}
	\end{align}
Inserting  \eqref{eq:unifd3} and \eqref{eq:unifd4} into Taylor's inequality applied at the third order for $i\ne j$ and at the fourth order for $i=j$, we obtain that there exists a positive constant $c>0$ such that for every $(r,u)$ with $r\le R_{\mbox{\tiny{inj}}}$, 
\begin{align*}
|\reste_{i,j}(r,u)| r^2 &\le  \sup_{(r,u)\in [0,R_{\mbox{\tiny{inj}}}]\times \sph^{n-1}} \bigg|\frac{\partial^3A_{i,j}}{\partial r^3}(r,u)\bigg| \frac{r^3}{3}\le c r^3,\quad 
 \mbox{ for every $i\neq j$,}
\end{align*}
and
\begin{align*}
|\reste_{i,i}(r,u)| r^3 &\le  \sup_{(r,u)\in [0,R_{\mbox{\tiny{inj}}}]\times \sph^{n-1}} \bigg|\frac{\partial^4A_{i,j}}{\partial r^4}(r,u)\bigg| \frac{r^4}{24}\le c r^4 \mbox{ for every $i$.}
\end{align*} 
	It then follows that there exists a positive constant not depending on $u$ such that for every $i,j$ and $(r,u)$ with $r\le R_{\mbox{\tiny{inj}}}$,
	\begin{equation}\label{eq:borneunifreste}
	|\reste_{i,j}(r,u)| \le cr.
	\end{equation}
			We now expand the determinant of the matrix $A(r,u)$. Let us denote by $\mathfrak{S}_{n-1}$ the group of permutation of $\{2,\dots, n\}$, by $\mathrm{Id}$ the identity permutation of $\mathfrak{S}_{n-1}$ and  by $\mathrm{sgn}(\sigma)$ the signature of the permutation $\sigma$. Leibniz formula for the determinant of matrices states that
\begin{align} \det (A(r,u)) &= \sum_{\sigma \in \mathfrak{S}_{n-1}} \mathrm{sgn}(\sigma) \prod_{j=1}^{n-1} A_{\sigma(j),j}(r,u)\notag \\
                &=  \prod_{j=1}^{n-1}  \langle u_{j+1}(r), J_{j+1}(r)\rangle_{x} +  \sum_{\sigma \in \mathfrak{S}_{n-1} \backslash \{ \mathrm{Id} \}} \mathrm{sgn}(\sigma) \prod_{j=1}^{n-1}  \langle u_{\sigma(j+1)}(r), J_{j+1}(r)\rangle_{x}. \label{detA}
								\end{align} 
It is expected that the first term in \eqref{detA} is dominant in the expansion of $\det(A(r,u))$. Let us expand this particular term: using \eqref{Adiag}, the fact that $\mathcal{V}$ is an orthonormal basis and the definition of the Ricci curvature at \eqref{eq:defric}, we get
 \begin{align} 
 \prod_{j=1}^{n-1}  \langle u_{j+1}(r), J_{j+1}(r)\rangle_{x} & =  \prod_{j=1}^{n-1} \left(r - \frac{ K^{(M)}_{x_0}(u,u_{j+1})}6 r^3 +\reste_{j,j}(r,u)r^3 \right) \notag\\
     &= r^{n-1} - \frac{\Ric^{(M)}_{x_0}(u)}6 r^{n+1} + \reste_1(r,u)r^{n+1}, \label{MainDet}
\end{align}
where 
$\reste_1(r,u)$ tends to 0 when $r$ tends to 0. Moreover, this convergence is independent of $u$ since $\reste_1(r,u)$ is simply a linear combination of remainders $\reste_{i,i}$ and of terms of the form \eqref{Adiag}, for all $i=2,\dots, n$.   
	It remains to show that the second term in \eqref{detA} satisfies 
	\begin{equation}\label{negligDet} \lim_{r\to 0}\frac{1}{r^{n+1}} \sum_{\sigma \in \mathfrak{S}_{n-1} \backslash \{ \mathrm{Id} \}} \mathrm{sgn}(\sigma) \prod_{j=1}^{n-1}  \langle u_{\sigma({j+1})}(r), J_{j+1}(r)\rangle_{x} = 0 \end{equation}
and that this convergence is uniform with respect to $u$.

Indeed, let $\sigma \in\mathfrak{S}_{n-1} \backslash \{ \mathrm{Id} \}$. There are at least two indices $j$ such that $\sigma(j)\neq j$. Without loss of generality, we assume that $\sigma(2) \neq 2$ and $\sigma(3)\neq 3$. Then, 
\begin{align}
\prod_{j=1}^{n-1} | \langle u_{\sigma({j+1})}(r), J_{j+1}(r)\rangle_{x}| &= | \langle u_{\sigma({2})}(r), J_{2}(r)\rangle_{x}| |\langle u_{\sigma({3})}(r), J_{3}(r)\rangle_{x}|
\prod_{\substack{j=1\\j\neq 1,2}}^{n-1}  |\langle u_{\sigma({j+1})}(r), J_{j+1}(r)\rangle_{x}| \notag \nonumber\\
 & \le r^4 |\reste_{\sigma({2}),2}(r,u)\reste_{\sigma({3}),3}(r,u)|\prod_{\substack{j=1\\j\neq 1,2}}^{n-1}  |\langle u_{\sigma({j+1})}(r), J_{j+1}(r)\rangle_{x}|. \label{inegProd}
\end{align}

Let $j \neq 1,2$. If $\sigma(j+1)=j+1$ then by \eqref{Adiag}, for every $u$ and $r\le R_{\mbox{\tiny{inj}}}$
\begin{align}\label{eq:majorantentreematrice}
|\langle u_{\sigma({j+1})}(r), J_{j+1}(r)\rangle_{x}|&= \left|r - \frac{ K^{(M)}_{x_0}(u,u_{j+1})}6 r^3 +\reste_{j,j}(r,u)r^3 \right|\nonumber\\
														&\le r\left| 1 -\frac{ K^{(M)}_{x_0}(u,u_{j+1})}6 r^2 +\reste_{j,j}(r,u)r^2 \right|\nonumber\\
														&\le cr
\end{align}
for some positive constant $c$, independent of $u$, thanks to Assumption $\mbox{(}{\mbox{A}}_1\mbox{)}$ and \eqref{eq:borneunifreste}.

If
$\sigma(j+1)\neq j+1$, we deduce from \eqref{AnonDiag} and \eqref{eq:borneunifreste} that \eqref{eq:majorantentreematrice} still holds.

Inserting \eqref{eq:majorantentreematrice} into \eqref{inegProd}, we obtain that for every $r \le R_{\mbox{\tiny{inj}}}$ and every $u$, there exists a positive constant $c$ such that
\begin{equation*}
\prod_{j=1}^{n-1} | \langle u_{\sigma({j+1})}(r), J_{j+1}(r)\rangle_{x}| \le c|\reste_{\sigma({2}),2}(r)\reste_{\sigma({3}),3}(r)|r^{n+1}.
\end{equation*}
Thus we get \eqref{negligDet} which, combined with \eqref{MainDet} and \eqref{detA} provides the required expansion 
\[ \J_{x_0}^{(M)}(r,u)=\det (A(r,u)) = r^{n-1} - \frac{\Ric^{(M)}_{x_0}(u)}6 r^{n+1} + \reste_2(r)r^{n+1} \]
where $\reste_2(r)$ goes to $0$ independently of $u$ when $r$ tends to $0$.\\~\\ 
\noindent{(ii)} Only a slight modification of (i) is needed here. The dependency with respect to $x_0$ of both the Jacobian $\J_{x_0}^{(M)}$ and its coefficients has been overlooked in the proof of (i) and is now relevant. In particular, both functions $\frac{\partial^3A_{i,j}}{\partial r^3}$ and $\frac{\partial^4A_{i,j}}{\partial r^4}$ are continuous on the compact set
$M\times[0,R_{\mbox{\tiny inj}}]\times \sph^{n-1}$, which implies that the bounds from \eqref{eq:unifd3} and \eqref{eq:unifd4} are uniform with respect to both $x_0$ and $u$. The rest of the proof is then identical line by line to the proof of (i).\\~\\
\noindent{(iii)}	Again, the method goes along similar lines as the proof of (i). 
 From now on, we carefully discuss the uniformity of each estimate with respect to $x_0\in M$. 

Thanks to \eqref{eq:partderivA} and to the Jacobi equation given at \eqref{eqJacobi}, we get for every couple $(u,r)$ and $x_0\in M$
\begin{equation*}
  \label{eq:deriveesecondecoeff}
 \frac{\partial^2 A_{i,j}}{\partial r^2}(r,u)=\langle u_{i+1}(r), J_{j+1}''(r) \rangle_{x_0} = - \langle u_{i+1}(r), \mathcal{R}_{\gamma(r)}^{(M)}(J_{j+1}(r),u(r))u(r) \rangle_{\gamma(r)}.
\end{equation*}
Consequently, when $J_{j+1}(r)\ne 0$, we obtain that
\begin{equation}
  \label{eq:deriveesecondecoeffmajorant}
\bigg| 
\frac{\partial^2 A_{i,j}}{\partial r^2}(r,u)\bigg| \le 
\|J_{j+1}(r)\|_{\gamma(r)}\bigg|\langle u_{i+1}(r), \mathcal{R}_{\gamma(r)}^{(M)}\left(\frac{J_{j+1}(r)}{ \|J_{j+1}(r)\|_{\gamma(r)}},u(r)\right)u(r) \rangle_{\gamma(r)}\bigg|.  
\end{equation}
In particular, Assumption \mbox{($\mbox{A}_1$)} implies that the second term in the right hand side of \eqref{eq:deriveesecondecoeffmajorant} is uniformly bounded with respect to $(r,u)$ and $x_0\in M$. Indeed, any scalar product of the form $\langle v_1, \mathcal{R}_{x}^{(M)}\left(v_2,v_3\right)v_4 \rangle_{x}$, where $x\in M$ and $v_1,\dots,v_4$ are unit vectors in $T_xM$, can be written as a linear combination of sectional curvatures, see e.g. \cite[Formula 1.10 p. 16]{Che08}. Moreover, since Assumption \mbox{($\mbox{A}_1$)} is satisfied, we can also apply Rauch's theorem, see e.g. \cite[Theorem 2.3]{Do92}, which shows that the first term $\|J_{j+1}(r)\|_{\gamma(r)}$ is bounded from above by the norm of a Jacobi field in a Riemannian manifold with constant curvature. Such a Jacobi field is known explicitly, see for instance \cite[Example 2.3]{Do92}. It does not depend on $x_0$ or $u$ and its norm is bounded by a fixed constant as soon as $r\le r_1$ for some fixed $r_1>0$. Consequently, we deduce from \eqref{eq:deriveesecondecoeffmajorant} that there exist $r_1>0$ and a positive constant $c>0$ such that for every $x_0\in M$, $u\in T_{x_0}M$ with $\|u\|_{x_0}=1$ and $r\in [0,r_1]$,
\begin{equation}
  \label{eq:deriveesecondecoeffmajorant2}
\bigg| \frac{\partial^2 A_{i,j}}{\partial r^2}(r,u)\bigg| \le c.
\end{equation}
Inserting \eqref{eq:derive0}, \eqref{eq:derive1} and \eqref{eq:deriveesecondecoeffmajorant2} into Taylor's inequality at the second order applied to  $A_{i,j}(r,u)$ seen as a function of $r$ for fixed $u$, we obtain that there exists a positive constant $c$ such that for every $x_0\in M$, $r\le r_1$ and $u$, 
\begin{equation}
  \label{eq:coeffmajorant3}
|A_{i,j}(r,u)-\delta_{i,j} r|\le c r^2.
\end{equation}
We proceed now as in the proof of (i), that is we use the rewriting of $\det(A(r,u))$ given at \eqref{detA}, then replace each coefficient by its estimate obtained at \eqref{eq:coeffmajorant3}. The term in the sum obtained for $\sigma=\mbox{Id}$ is equal to $r^{n-1}$ up to $c r^{n}$ for some positive constant $c$ and the remaining sum for $\sigma\ne \mbox{Id}$ is equal to zero up to $c r^{n+2}$. Consequently, for $r$ small enough, we get
$$ |\J_{x_0}^{(M)}(r,u)-r^{n-1}|\le c r^{n}$$
for every $x_0$ and $u$. This implies the required result when choosing $r_1\le \frac{1}{2c}$.
\begin{flushright}
$\Box$
\end{flushright}							
\noindent{\it Proof of Lemma \ref{lem:volunif}.} (i) Let $\varepsilon>0$. We fix $r_1$ small enough such that the ball ${\mathcal B}^{(M)}(x_0,r_1)$ is a compact set. We then apply point (ii) of Lemma \ref{lem:jacobchtsph} to the manifold ${\mathcal B}^{(M)}(x_0,r_1)$ and take $r_0<r_1$ such that for $r\le r_0$,
\begin{align}\label{eq:jacobunifestime}
\sup_{x\in {\mathcal B}^{(M)}(x_0,r_1)}\sup_{u\in T_{x}M, \|u\|_{x}=1}r^{-(n+1)}|\J_{x}^{(M)}(r,u) - r^{n-1} + \frac{ \Ric^{(M)}_{x}(u)}{6}r^{n+1}|&\le \varepsilon.
\end{align}
Because of point (iii) of Lemma \ref{lem:jacobchtsph}, we may assume that $\J_{x}^{(M)}(r,u)$ is non-negative for every $x\in {\mathcal B}^{(M)}(x_0,r)$, $r\le r_0$ and $u\in T_{x_0}M$ with $\|u\|_{x_0}=1$. In particular, the rewriting of the volume element given at \eqref{eq:Jru} implies that
\begin{equation}
  \label{eq:volboulesimple}
\vol^{(M)}(\bo(x,r))=\int_{t=0}^{r}\int_{u\in {\mathcal S}^{n-1}}\J_{x}^{(M)}(t,u)\mathrm{d}t\mathrm{d}{\mbox{vol}}^{({\mathcal S}^{n-1})}(u).
\end{equation}
Inserting \eqref{eq:jacobunifestime} into \eqref{eq:volboulesimple} and integrating over $t$ and $u$, we get
\begin{equation}
  \label{eq:presquevolunif}
r^{-(n+2)}\left(\vol^{(M)}(\bo^{(M)}(x,r))-\kappa_n r^n + \frac{1}{6(n+2)} \int \Ric^{(M)}_{x}(u)\mathrm{d}{\mbox{vol}}^{({\mathcal S}^{n-1})}(u)\cdot r^{n+2}\right)\le \frac{\varepsilon}{2(n+2)}.
\end{equation}
Recalling \eqref{eq:scalRicci}, we may replace the integral in \eqref{eq:presquevolunif} by $\kappa_n \Scal^{(M)}_x$. 
Even if it means taking $r_0$ smaller, we may then assume by continuity of the function $x\mapsto \Scal_x^{(M)}$ at $x_0$ that for $x\in {\mathcal B}^{(M)}(x_0,r_0)$, we have 
\begin{equation}
  \label{eq:contscalaire}
|\Scal_x^{(M)}-\Scal_{x_0}^{(M)}|\le \varepsilon.  
\end{equation}
Combining now \eqref{eq:presquevolunif}, \eqref{eq:scalRicci} and \eqref{eq:contscalaire} provides the required result.\\~\\
(ii) The proof is almost identical to the proof of (i), save for the fact that we apply now point (ii) of Lemma \ref{lem:volunif} to the whole manifold $M$ and that we use the uniform continuity of the function $x\mapsto \Scal(x)$ on the compact set $M$ instead of the standard continuity.


 \begin{flushright}
$\Box$
\end{flushright}								
\noindent{\it Proof of Lemma \ref{lem:borneinfvolboule}.}
Let $x_0\in M$ and $r>0$. Since $M$ is non-compact and complete, there exists a unit speed geodesic $\gamma$ of length $r$ which emanates from $x_0$ and such that $d^{(M)}(\gamma(0),\gamma(r))=r$. Indeed, if not, the closure of the ball ${\mathcal B}^{(M)}(x_0,r)$ would be a closed and bounded subset equal to $M$ so a compact set by Hopf-Rinow theorem \cite[Theorem 52]{Ber03}. 

Let us fix $\varepsilon\in (0,r/2)$ and let us choose a set $\{x_1,\dots,x_N\}$ of $N=\lfloor \frac{r}{4\varepsilon}\rfloor$ points on the first half of $\gamma$ such that $d^{(M)}(x_{i+1},x_i)=2\varepsilon$. The union $\cup_{i=1}^N {\mathcal B}^{(M)}(x_i,\varepsilon)$ is constituted of $N$ balls which are all included in ${\mathcal B}^{(M)}(x_0,r)$ because of the triangular inequality. Moreover, these balls are disjoint: indeed, if not, let $y\in {\mathcal B}^{(M)}(x_i,\varepsilon)\cap {\mathcal B}^{(M)}(x_j,\varepsilon)$ with $j>i$. In particular, $$d^{(M)}(x_0,x_j)\le d^{(M)}(x_0,x_i)+d^{(M)}(x_i,y)+d^{(M)}(y,x_i)<d^{(M)}(x_0,x_i)+2\varepsilon$$ which contradicts the fact that the distance measured along $\gamma$ is the actual distance $d^{(M)}$.

Consequently, we get that
\begin{equation}
  \label{eq:guirlande}
\mbox{vol}^{(M)}({\mathcal B}^{(M)}(x_0,r))\ge \sum_{i=1}^N \mbox{vol}^{(M)}({\mathcal B}^{(M)}(x_i,\varepsilon)).   
\end{equation}
Now let us choose $\varepsilon$ smaller than the $r_0$ provided by Lemma \ref{lem:jacobchtsph} (iii). Applying the change of variables given at \eqref{eq:Jru}, we obtain for every $1\le i\ne N$,
\begin{equation}
  \label{eq:boulesguirlande}
\mbox{vol}^{(M)}({\mathcal B}^{(M)}(x_i,\varepsilon))=\int_{{\mathcal S}^{n-1}}\int_0^{\varepsilon}|\J_{x_0}^{(M)}(r,u)| \dd r \dvol^{(\sph^{n-1})}(u)\ge \frac{1}{2}\sigma_{n-1}\int_0^{\varepsilon}r^{n-1}\mathrm{d}r=\frac{1}{2n}\sigma_{n-1}\varepsilon^n.  
\end{equation}
We conclude by considering two cases:\\
\- either $r>r_0$, in which case we insert \eqref{eq:boulesguirlande} into \eqref{eq:guirlande} for $\varepsilon =r_0/4$ and we use the fact that $N=\lfloor \frac{r}{4\varepsilon}\rfloor$ to get a linear lower bound for $\mbox{vol}^{(M)}({\mathcal B}^{(M)}(x_0,r))$,\\
\- or $r\le r_0$, in which case we apply \eqref{eq:boulesguirlande} to the ball  ${\mathcal B}^{(M)}(x_0,r)$ and we get a lower bound proportional to $r^n$.\\
This completes the proof of Lemma \ref{lem:borneinfvolboule}.
\begin{flushright}
$\Box$
\end{flushright}
\noindent{\it Proof of Lemma \ref{lem:secondmoment}}.
We start by proving \eqref{eq:momentsfinisvol}. It is enough to show that $\E[\vol^{(M)}(\Cl)]$ is bounded as the same result for $\E[\vol^{(M)}({\widetilde{\mathcal C}}^{(M)}_{x_0,\lambda})^2]$ follows along similar lines.
By Fubini's theorem, we can rewrite the second moment of $\vol^{(M)}(\Cl)$ in the following way:
\begin{align}
&\E[\vol^{(M)}(\Cl)^2]\\&=\iint \proba[x_1,x_2\in \Cl]\mathrm{d}\mbox{vol}^{(M)}(x_1)\mathrm{d}\mbox{vol}^{(M)}(x_2)\nonumber\\
&=\iint \proba[\Pl\cap ({\mathcal B}^{(M)}(x_1,d^{(M)}(x_0,x_1))\cup{\mathcal B}^{(M)}(x_2,d^{(M)}(x_0,x_2))=\emptyset]\dM(x_1)\dM(x_2)\nonumber\\  
&=\iint e^{-\lambda \vol^{(M)}({\mathcal B}^{(M)}(x_1,d^{(M)}(x_0,x_1))\cup{\mathcal B}^{(M)}(x_2,d^{(M)}(x_0,x_2))}\dM(x_1)\dM(x_2).\label{eq:volcarreinterm}
\end{align}
Let $r_0$ be given by Lemma \ref{lem:jacobchtsph} (iii). 
Thanks to Lemma \ref{lem:borneinfvolboule}, we get that there exists a positive constant $c$ such that for every $x_1,x_2\in M\setminus {\mathcal B}^{(M)}(x_0,r_0)$
\begin{equation}
  \label{eq:volunion}
\vol^{(M)}({\mathcal B}^{(M)}(x_1,d^{(M)}(x_0,x_1))\cup{\mathcal B}^{(M)}(x_2,d^{(M)}(x_0,x_2))\ge \frac{c}{2}(d^{(M)}(x_0,x_1)+d^{(M)}(x_0,x_2)).
\end{equation}
and for every $x_1,x_2\in  {\mathcal B}^{(M)}(x_0,r_0)$,
\begin{equation}
  \label{eq:volunion2}
\vol^{(M)}({\mathcal B}^{(M)}(x_1,d^{(M)}(x_0,x_1))\cup{\mathcal B}^{(M)}(x_2,d^{(M)}(x_0,x_2))\ge \frac{c}{2}(d^{(M)}(x_0,x_1)^n+d^{(M)}(x_0,x_2)^n).
\end{equation}
Inserting \eqref{eq:volunion} and \eqref{eq:volunion2} into \eqref{eq:volcarreinterm} and using again Fubini's theorem, we obtain that
\begin{align}
\sqrt{\E[\vol^{(M)}(\Cl)^2]}&\le \int_{{\mathcal B}^{(M)}(x_0,r_0)}e^{-\lambda \frac{c}{2}d^{(M)}(x_0,x_1)^n}\dM(x_1)+\int_{ M\setminus {\mathcal B}^{(M)}(x_0,r_0)} e^{-\lambda \frac{c}{2}d^{(M)}(x_0,x_1)}\dM(x_1)\label{eq:intermvolborne}.
\end{align}
We bound separately each of the two terms from the right-hand side of \eqref{eq:intermvolborne}. The first term is treated thanks to Lemma \eqref{lem:jacobchtsph} (iii):
\begin{align}
\int_{{\mathcal B}^{(M)}(x_0,r_0)}e^{-\lambda \frac{c}{2}d^{(M)}(x_0,x_1)^n}\dM(x_1)&=\int_{{\mathcal S}^{n-1}}\int_0^{r_0} e^{-\lambda \frac{c}{2} r^n}|\J_{x_0}^{(M)}(r,u)|\mathrm{d}r\mathrm{d}\mbox{vol}^{{\mathcal S}^{n-1}}(u)\nonumber\\
&\le \frac{3}{2}\sigma_{n-1}\int_0^{r_0} e^{-\lambda \frac{c}{2} r^n}r^{n-1}\mathrm{d}r\nonumber\\
&\le \frac{3\sigma_{n-1}}{cn}\frac{1}{\lambda}\label{eq:1ertermevol}.   
\end{align}
We turn now to the second term which can be decomposed in the following way.
\begin{align}
\int e^{-\lambda \frac{c}{2}d^{(M)}(x_0,x_1)}\dM(x_1)& \le \sum_{l=0}^{\infty}\vol^{(M)}(\{x_1\in M: l\le d^{(M)}(x_0,x_1)\le l+1\}) e^{-\lambda \frac{c}{2} l}\nonumber\\
& \le \sum_{l=0}^{\infty}\vol^{(M)}({\mathcal B}^{(M)}(x_0,l+1)) e^{-\lambda \frac{c}{2} l}.\label{eq:serieexpo}
\end{align}
Because of the assumption ($\mbox{A}_{1}$), the manifold $M$ has a Ricci curvature bounded from below by some constant $(n-1)\delta$. Without loss of generality, we can assume that $\delta<0$. Thanks to the Bishop-Gromov theorem, see e.g. \cite[Theorem 107]{Ber03}, this implies that the volume $\vol^{(M)}({\mathcal B}^{(M)}(x_0,l+1))$ is bounded from above by the volume of a ball with same radius in a manifold of constant curvature $\delta$. In other words, this means that for some constants $c,c'>0$, we get
\begin{equation}
  \label{eq:volcourbureminoree}
\vol^{(M)}(B^{(M)}(x_0,l+1))\le \vol^{({\mathcal H}_{\delta}^n)}(B^{(M)}(x_0,l+1))\le c\int_0^{l+1}\sinh^{n-1}(t)\mathrm{d}t\le c'e^{c'l}.  
\end{equation}
Inserting \eqref{eq:volcourbureminoree} into \eqref{eq:serieexpo}, we deduce that for $\lambda$ large enough, there exists a positive constant $c$ such that
\begin{equation}
  \label{eq:2emetermevol}
\int e^{-\lambda \frac{c}{2}d^{(M)}(x_0,x_1)}\dM(x_1)\le  \frac{c}{\lambda}. 
\end{equation}
Combining \eqref{eq:intermvolborne}, \eqref{eq:1ertermevol} and \eqref{eq:2emetermevol}, we deduce the required result \eqref{eq:momentsfinisvol}.

We turn now to \eqref{eq:momentsfinisN} and prove that $\E[N(\Cl)^2]$ is bounded. We first rewrite $\Nl$ in a convenient way: each vertex of the Voronoi cell  ${\mathcal C}^{(M)}(x_0,{\mathcal P}_{\lambda}\cup\{x_0\})$ belongs to that cell and to exactly $n$ other Voronoi cells. In other words, it is the center of an open geodesic ball which does not meet ${\mathcal P}_{\lambda}$ and contains $(n+1)$ distinct points $x_0,\dots,x_n$ of ${\mathcal P}_{\lambda}$ on its boundary. 
Actually, the circumscribed ball of $(n+1)$ fixed points needs not to be unique but Assumption \mbox{($\mbox{A}_3$)} garantees that its number $n(x_0,\dots,x_n)$ is bounded by $c>0$. If ${\mathcal B}^{(M)}_1(x_0,\dots,x_n),\dots,{\mathcal B}^{(M)}_n(x_0,\dots,x_n)$ are these balls, then
$$\Nl=\sum_{x_1,\dots,x_n\in {\mathcal P}_{\lambda}}\sum_{i=1}^{n(x_0,\dots,x_n)}{\bf 1}_{\{{\mathcal B}^{(M)}_i(x_0,\dots,x_n)\cap {\mathcal P}_{\lambda}=\emptyset\}}$$
and
$$\Nl^2=\sum_{x_1,\dots,x_n\in {\mathcal P}_{\lambda}}\sum_{x_{n+1},\dots,x_{2n}\in {\mathcal P}_{\lambda}}\sum_{i=1}^{n(x_0,\dots,x_n)}\sum_{j=1}^{n(x_0,x_{n+1},\dots,x_{2n})}{\bf 1}_{\{[{\mathcal B}^{(M)}_i(x_0,\dots,x_n)\cup {\mathcal B}^{(M)}_j(x_0,x_{n+1},\dots,x_{2n})]\cap {\mathcal P}_{\lambda}=\emptyset\}}.$$
When estimating $\E(\Nl^2)$, we observe that each term for fixed $i$ and $j$ will provide the same upper bound and that the number of such terms will be bounded by $c^2$. Consequently, for sake of simplicity, we may as of now assume that circumscribed ball of $(n +1) $ fixed points is unique and is denoted by ${\mathcal B}^{(M)}(x_0,\dots,x_n)$.

Using the Mecke-Slivnyak formula \cite[Proposition 4.1.1]{Mol94}, we get
\begin{align}
&\E(\Nl^2)\nonumber\\&= \E\left(\sum_{x_1,\dots,x_n\in {\mathcal P}_{\lambda}}\sum_{x_{n+1},\dots,x_{2n}\in {\mathcal P}_{\lambda}}{\bf 1}_{\{[{\mathcal B}^{(M)}(x_0,\dots,x_n)\cup {\mathcal B}^{(M)}(x_0,x_{n+1},\dots,x_{2n})]\cap {\mathcal P}_{\lambda}=\emptyset\}}\right)\nonumber\\
& = \sum_{k=0}^n\frac{\lambda^{n+k}}{(n+k)!}\binom{n}{k}\int\proba(\{[{\mathcal B}^{(M)}(x_0,\dots,x_n)\cup {\mathcal B}^{(M)}(x_0,x_1,\dots,x_{n-k},x_{2n-k+1},\dots,x_{2n})]\cap {\mathcal P}_{\lambda}=\emptyset)\nonumber\\
&\hspace*{5cm}\dM (x_1)\dots\dM (x_{n-k})\dM (x_{2n-k+1})\dots\dM (x_{2n})\nonumber\\
& = \sum_{k=0}^n\frac{1}{(n+k)!}\binom{n}{k}I_k(\lambda)\label{eq:defIk}
\end{align}
where \begin{align}
  \label{eq:Ik}
I_k(\lambda)&=\lambda^{n+k}\int e^{-\lambda \vol^{(M)}({\mathcal B}(x_0,\dots,x_n)\cup {\mathcal B}(x_0,x_1,\dots,x_{n-k},x_{2n-k+1},\dots,x_{2n}))}\nonumber\\&\hspace*{4cm}\dM (x_1)\dots\dM (x_{n})\dM (x_{2n-k+1})\dots\dM (x_{2n}).  
\end{align}
Let us treat separately each integral $I_k(\lambda)$. Thanks again to Lemma \ref{lem:borneinfvolboule}, for some constant $c>0$, we have
\begin{align}
  \label{eq:minvoldecomp}
&\vol^{(M)}(\bo^{(M)}(x_0,\dots,x_n)\cup \bo^{(M)}(x_0,x_1,\dots,x_{n-k},x_{2n-k+1},\dots,x_{2n}))\nonumber\\
&\hspace*{1cm}\ge \frac{c}{n+k} \sum_{i\in \{1,\dots,n,2n-k+1,\dots,2n\}}({d}^{(M)}(x_0,x_i)^n{\bf 1}_{\{x_i\in {\mathcal B}^{(M)}(x_0,r_0)\}}+{d}^{(M)}(x_0,x_i){\bf 1}_{\{x_i\not\in {\mathcal B}^{(M)}(x_0,r_0)\}}).  
\end{align}
Inserting \eqref{eq:minvoldecomp} into \eqref{eq:Ik} then using Fubini's theorem and a decomposition similar to \eqref{eq:serieexpo}, we obtain that
\begin{align}
I_k(\lambda)^{\frac{1}{n+k}} &\le  \lambda \left(\int_{{\mathcal B}^{(M)}(x_0,r_0)}e^{-\lambda \frac{c}{n+k} \mathrm{d}^{(M)}(x_0,x_1)^n}\dM (x_1)+\int e^{-\lambda \frac{c}{n+k} \mathrm{d}^{(M)}(x_0,x_1)}\dM (x_1)\right).\nonumber
\end{align}
We treat each of the two terms in the exact same way as for \eqref{eq:intermvolborne} and conclude that $I_k(\lambda)$ is bounded from above by a constant. Using finally \eqref{eq:defIk}, we get the required result \eqref{eq:momentsfinisN}.
\begin{flushright}
$\Box$
\end{flushright}
\noindent{\it Proof of Lemma \ref{lem:stabilization}}.
Let us define 
the Voronoi flower ${\mathcal F}_{x_0,\lambda}^{(M)}$ associated with the cell $\Cl$ as 
$${\mathcal F}_{x_0,\lambda}^{(M)}=\cup_{y\in \Cl}{\mathcal B}^{(M)}(y,d^{(M)}(y,x_0)).$$
In particular, the set ${\mathcal F}_{x_0,\lambda}^{(M)}$ does not meet $\Pl$ and contains all the Voronoi neighbors of $x_0$ on its boundary. In particular, only the intersection of $\Pl$ with ${\mathcal F}_{x_0,\lambda}^{(M)}$ is needed to construct the Voronoi cell $\Cl$. Consequently, if the sets $\Cl$ and ${\widetilde{\mathcal C}}^{(M)}_{x_0,\lambda}$ differ, then the set ${\mathcal F}_{x_0,\lambda}^{(M)}$ is not included in ${\mathcal B}^{(M)}(x_0,r)$. This implies that there exists an empty ball ${\mathcal B}^{(M)}(y,d^{(M)}(y,x_0))$ of radius at least $r/2$. In other words,
\begin{equation}
  \label{eq:boulevide}
\proba(\Cl\ne {\widetilde{\mathcal C}}^{(M)}_{x_0,\lambda})\le \proba(\exists y\in M\setminus {\mathcal B}^{(M)}(x_0,r/2):{\mathcal B}^{(M)}(y,d^{(M)}(y,x_0))\cap \Pl=\emptyset).   
\end{equation}
Recalling Assumption (A1), we can use Gromov's packing lemma, see e.g. \cite[Lemma 2.2.A]{Gro81} 
 which implies that for $r$ small enough, there exist $m$ points where $m$ is independent of $r$ such that the ball ${\mathcal B}^{(M)}(x_0,r)$ is covered by the balls $B_1,\dots,B_m$ centered at the $m$ points and of radius $r/8$. In particular, any ball ${\mathcal B}^{(M)}(y,d^{(M)}(y,x_0))$ with $d^{(M)}(x_0,y)\ge r/2$ contains one of the $B_i$. Consequently,
\begin{align}\label{eq:majexpo}
\proba(\exists y\in M\setminus {\mathcal B}^{(M)}(x_0,r/2):{\mathcal B}^{(M)}(y,d^{(M)}(y,x_0))\cap \Pl=\emptyset)&\le \proba(\exists i \in \{1,\dots,m\} \mbox{ s.t. } B_i\cap \Pl=\emptyset)\nonumber\\
&\le \sum_{i=1}^m\proba(B_i\cap \Pl=\emptyset)\nonumber\\
&\le m e^{-\lambda cr^n}
\end{align}
where the estimate $\mbox{vol}^{(M)}(B_i)\ge c r^n$ comes from the combination of \eqref{eq:Jru} and Lemma \ref{lem:jacobchtsph} (iii). Inserting \eqref{eq:majexpo} into \eqref{eq:boulevide}, we obtain the statement of Lemma \ref{lem:stabilization}.
\begin{flushright}
$\Box$
\end{flushright}
\noindent{\bf Acknowledgements.} This work was partially supported by the French ANR grant PRESAGE (ANR-11-BS02-003), the French research group GeoSto (CNRS-GDR3477) and the Institut Universitaire de France.
\bibliography{BiblioVoronoi}
\bibliographystyle{alpha}
\end{document}